\documentclass[a4paper]{amsart}
\usepackage{common}
\usepackage{graphicx} 
\usepackage{todonotes}

\newcommand{\mc}{\mathcal}
\newtheorem*{remark}{Remark}

\DeclareMathOperator*{\essinf}{ess\,inf}

\usepackage{cleveref}

\author{Francisco Gon\c{c}alves}
\address{Delft Institute of Applied Mathematics, Delft University of
Technology, P.O. Box 5031, 2600 GA Delft, The Netherlands}
\email{f.g.j.diasdecarvalho@tudelft.nl}
\author{Tuomas Oikari}
\address{Aalto University, Department of Mathematics and Systems Analysis, P.O. Box 11100,
	FI-00076 Aalto, Finland}
\email{tuomas.oikari@gmail.com}

\title{Bounds for the maximal and Riesz potential operators with variable fractionality}
\date{\today}

\begin{document}

	\begin{abstract} 
		We prove $L^{p(\cdot)}$-to-$L^{q(\cdot)}$ bounds for variable versions of the fractional maximal $M^{\alpha(\cdot)}$ and Riesz potential $I^{\alpha(\cdot)}$ operators. The changing  fractionality in these operators is given by averaging the function $\alpha(\cdot)$ over balls.
		The bounds for $M^{\alpha(\cdot)}$ are in terms of a three-exponent Muckenhoupt condition relating $p(\cdot),q(\cdot),$ and $\alpha(\cdot)$, while the bounds for 
        $I^{\alpha(\cdot)}$ are in terms of the boundedness of $M^{\alpha(\cdot)}$ and a packing condition on $\alpha(\cdot).$ 
        These bounds hold under Hardy--Littlewood maximal function boundedness and Muckenhoupt conditions on the individual exponents $p(\cdot),q(\cdot),\alpha(\cdot).$ 
		The proofs are based on an adaptation of sparse domination to variable fractionality and an embedding into variable sequential spaces. 
	\end{abstract}

\maketitle
	

\section{Introduction}

\subsection{Background}
The Hardy--Littlewood maximal function $M$ \cite{HarLit1930max} is bounded on $L^p(\mathbb{R}^n)$ for $1<p\leq\infty$, while 
for $\alpha\in(0,1)$ and $1<p<1/\alpha$, the fractional maximal operator $M^{\alpha}$ and the Riesz potential $I^{\alpha}$ are $L^p(\mathbb{R}^n)$-to-$L^q(\mathbb{R}^n)$ bounded when $1/q=1/p-\alpha$, as shown by Hardy--Littlewood--Sobolev  \cites{HarLit1928frac, Sob1938frac}. 
We study boundedness of these operators, and their generalizations, on the scale $L^{p(\cdot)}(\mathbb{R}^n)$ of variable exponent Lebesgue spaces,  which allow integrability to change through a variable exponent $p(\cdot):\mathbb{R}^n\to[1,\infty].$
Variable exponent Lebesgue spaces have been studied for their intrinsic interest and for their applications to PDEs and variational problems with non-standard growth. Applications include image restoration and models of electrorheological fluids; see \cites{ChaMul97,Fan03PxLap,HarHasLeNuo10,Mingione06,Ruzicka00,AcMi02,DienRuz03}.   

A central problem has been to characterize when $M$ is bounded on $L^{p(\cdot)}$ in terms of testing conditions on the exponent $p(\cdot).$ 
A logarithmic H\"older continuity condition $\mathcal{LH}$, consisting of its local $\mathcal{LH}_0$ and global $\mathcal{LH}_{\infty}$ parts, was shown to be sufficient in the combined works of
Cruz-Uribe--Fiorenza--Neugebauer \cite{CruFioNeu03} and Diening \cites{Dien04, Dien07Habil}.
The condition $\mathcal{LH}_{\infty}$ was then weakened by Nekvinda \cite{Nekv04} to an integrated condition $``\mathcal{N}_{\infty}"$, while
Kopaliani \cite{Kop2007} weakened $\mathcal{LH}_0$ to the condition $``\mathcal{K}_0"$, which is, as a necessary condition, highly relevant. The condition $p(\cdot)\in \mathcal{K}_0$ is equivalent to the finiteness of the following Muckenhoupt type quantity,
\[
[1]_{A_{p(\cdot)}(\mathbb{R}^n)} := \sup_B\frac{\|1_B\|_{L^{p(\cdot)}}\|1_B\|_{L^{p'(\cdot)}}}{|B|},
\]
where $\sup_B$ is over all balls. 
Full characterizations are more subtle and currently assume the exponent to be uniformly bounded away from infinity.
Diening's \cite{Dien05} characterization is in terms of the uniform boundedness of sums of disjointly supported averaging operators, and Lerner's \cites{Ler25MaxVar, Ler26MaxVar} in terms of certain $A_{\infty}$-type conditions.
We direct the reader to \cite{Ler26MaxVar} for an up-to-date history of this ongoing research direction. 
For the fractional maximal operator $M^{\alpha}$ the corresponding $L^{p(\cdot)}$-to-$L^{q(\cdot)}$ problem is naturally a two-exponent problem, which is largely open. Quite recently, Cruz-Uribe--Roberts \cite{CruTro2024} gave necessary conditions in this setting, while the special case under the $\mathcal{LH}$ assumption was treated in Diening and Capone--Cruz-Uribe--Fiorenza \cites{Dien04RieszPot, CapCruFio07}.
In the present article, we allow fractionality to change through a function $\alpha(\cdot)$, which leads to the operators $M^{\alpha(\cdot)},I^{\alpha(\cdot)}$ (see \eqref{eq:defn:VarMax}, \eqref{eq:defn:VarRiesz} below). We recognize that any full three-exponent characterization for $L^{p(\cdot)}$-to-$L^{q(\cdot)}$ boundedness of these operators is a very difficult problem; our goal is more modest. We aim to make minimal assumptions on the
individual exponents $\alpha(\cdot),p(\cdot),q(\cdot)$ (see \eqref{eq:intro:Struc1}, \eqref{eq:intro:Struc2}, \eqref{eq:intro:Struc3} below) and, under these assumptions, prove $L^{p(\cdot)}$-to-$L^{q(\cdot)}$ boundedness results for $M^{\alpha(\cdot)},I^{\alpha(\cdot)}$. 

We use averaging $\langle \alpha\rangle_B := \frac{1}{|B|}\int_B\alpha(\cdot)$ to introduce variable fractionality into the following generalizations of the maximal and Riesz potential operators
\begin{align}
	M^{\alpha(\cdot)}f(x) &:= \sup_{B}\frac{1_B(x)}{|B|^{1-\langle \alpha\rangle_B}}\int_B|f(y)|\ud y, \label{eq:defn:VarMax} \\ 
	I^{\alpha(\cdot)}f(x) & := \int_{\mathbb{R}^n\setminus\{x\}}\frac{f(y)\ud y}{|x-y|^{n\left(1-\langle\alpha\rangle_{B(x,|x-y|)}\right)}}.  \label{eq:defn:VarRiesz}
\end{align}
Taking averages of $\alpha(\cdot)$ is natural when constructing a maximal operator. On a more technical side, Theorem \ref{thm:MainIntro2} below demonstrates that such definitions can also lead to good quantitative bounds. 

\begin{remark}
	To the best of our knowledge, these exact operators \eqref{eq:defn:VarMax}, \eqref{eq:defn:VarRiesz} have not been considered before.
	Under the $\mathcal{LH}$ assumption, the quantities $|B|\|1_B\|_{L^{\frac{1}{\alpha(\cdot)}}}^{-1}$ were used in the denominators by Melchiori--Pradolini \cite{MelPra18} and Pastrana--Riveros--Vidal \cite{PasRivVid26}.  If $1\in A_{\frac{1}{\alpha(\cdot)}}$, then $\|1_B\|_{L^{\frac{1}{\alpha(\cdot)}}} \sim |B|^{\langle\alpha\rangle_B}$, by Lemma \ref{lem:IndicatorKopaliani} in the body text, and hence those operators are equivalent to those considered in the present article.  
	In a different direction, and now under the $\mathcal{LH}_0$ assumption,
	the pointwise evaluations $|B|^{\alpha(x)}$, for $x\in B,$ appear in the works \cites{HarSyr18ModRieszPot, HarSyr21VarRieszPot, KokSam04, AlmHasSam08} on bounded domains $\Omega$. Again, these operators are equivalent to those considered here, by Lemma \ref{lem:LogHolderStability} in Appendix \ref{sect:logHolder}.
\end{remark}

 \subsection{Main results}
 The log-H\"older classes $\mathcal{LH}_0, \mathcal{LH}_{\infty}$ allow much of the classical harmonic analysis on \emph{constant} exponent Lebesgue spaces to be worked out in the variable exponent setting; we refer to the classics \cites{CruzFio2013Book, DienHHR2011Book}.
 The challenge in using log-H\"older type assumptions is their qualitative character, which leaves little room for quantitative bounds.
Our main results replace logarithmic H\"older continuity assumptions by the following quantities
\begin{align}
	&[1]_{A_{1/\alpha(\cdot)}(\mathbb{R}^n)}, \label{eq:intro:Struc1}  \\
	&\|M\|_{L^{q'(\cdot)}(\mathbb{R}^n)}:= \|M\|_{L^{q'(\cdot)}(\mathbb{R}^n)\to L^{q'(\cdot)}(\mathbb{R}^n)}, \label{eq:intro:Struc2} 
     \\ 
	&\|M\|_{L^{p(\cdot)}(\mathbb{R}^n)} :=\|M\|_{L^{p(\cdot)}(\mathbb{R}^n)\to L^{p(\cdot)}(\mathbb{R}^n)}. \label{eq:intro:Struc3} 
\end{align}
\begin{remark}
     By Kopaliani \cite{Kop2007} the quantity 
    $[1]_{A_{\frac{1}{\alpha(\cdot)}}(\mathbb{R}^n)}$ is finite whenever $\|M\|_{L^{\frac{1}{\alpha(\cdot)}}(\mathbb{R}^n)}$ is finite. The reverse implication does not hold: even for a constant parameter $\alpha$, $\|M\|_{L^{\frac{1}{\alpha}}(\mathbb{R}^n)}$ blows up as $\alpha\to 1$ (the maximal function is unbounded on $L^1$), whereas $[1]_{A_{1/\alpha}(\mathbb{R}^n)} = 1$, uniformly in $\alpha\in [0,1].$ Thus \eqref{eq:intro:Struc1} avoids imposing maximal-function boundedness on $L^{1/\alpha(\cdot)}$ and, already for constant fractionality, does not deteriorate as $\alpha\to1.$ 
\end{remark}
We treat the three hypotheses \eqref{eq:intro:Struc1}, \eqref{eq:intro:Struc2}, \eqref{eq:intro:Struc3} as independent background assumptions.  For relating $p(\cdot)$, $q(\cdot)$ and $\alpha(\cdot)$ to each other, we often do not demand the usual pointwise relation $\alpha(\cdot) = \tfrac{1}{p(\cdot)}-\tfrac{1}{q(\cdot)}$\footnote{See Corollary \ref{cor:AlphaRelation} in the body text.} to hold, but instead quantify their interaction through the following three-exponent Muckenhoupt type quantity
	\[
	[1]_{A_{p(\cdot),q(\cdot)}^{\alpha(\cdot)}(\mathbb{R}^n)} := \sup_{B} \frac{\|1_B\|_{L^{q(\cdot)}}\|1_B\|_{L^{p'(\cdot)}}}{|B|^{1-\langle \alpha\rangle_B}}.
	\]
We write $1\in A_{p(\cdot),q(\cdot)}^{\alpha(\cdot)}(\mathbb{R}^n)$ when $[1]_{A_{p(\cdot),q(\cdot)}^{\alpha(\cdot)}(\mathbb{R}^n)}<\infty.$
The following theorem is our first result. To the best of our knowledge, these bounds are new already when only one of the exponents $\alpha(\cdot),p(\cdot),q(\cdot)$ is allowed to vary.
 	\begin{theorem}\label{thm:MainIntro1}
	Let  $p(\cdot), q(\cdot), 1/\alpha(\cdot)\in \mathcal{P}(\mathbb{R}^n)$ be such that $1<p(\cdot)\leq q(\cdot)<\infty.$ Then, 
	\begin{equation*}
		\begin{split}
			[1]_{A_{p(\cdot),q(\cdot)}^{\alpha(\cdot)}(\mathbb{R}^n)} &\lesssim  \| M^{\alpha(\cdot)}\|_{L^{p(\cdot)}(\mathbb{R}^n)\to L^{q(\cdot)}(\mathbb{R}^n)} \\ 
			&\lesssim_n [1]_{A_{p(\cdot),q(\cdot)}^{\alpha(\cdot)}(\mathbb{R}^n)} \big([1]_{A_{\frac{1}{\alpha(\cdot)}}(\mathbb{R}^n)}\big)^3     \|M\|_{ L^{q'(\cdot)}(\mathbb{R}^n)} \|M\|_{ L^{p(\cdot)}(\mathbb{R}^n)}.
		\end{split}
	\end{equation*}
\end{theorem} 
 Theorem \ref{thm:MainIntro1} follows from passing between non-dyadic and dyadic objects -- costing two powers of $[1]_{A_{\frac{1}{\alpha(\cdot)}}(\mathbb{R}^n)}$ -- and the following purely dyadic bounds of Theorem \ref{thm:MainIntro2}. We refer to Section \ref{sect:dyadic} for the precise definitions involved.
 \begin{theorem}\label{thm:MainIntro2}
 	Let $(X,\mu)$ be a $\sigma$-finite measure space and $p(\cdot), q(\cdot),1/\alpha(\cdot)\in \mathcal{P}(X,\mu)$ be such that $1<p(\cdot)\leq q(\cdot)<\infty.$  Let $\mathcal{D}$ be an arbitrary collection of dyadic sets on $(X,\mu).$
 	Then, 
	\begin{equation*}
		\begin{split}
			[1]_{A_{p(\cdot),q(\cdot)}^{\alpha(\cdot)}(\mathcal{D})} &\lesssim  \| M^{\alpha(\cdot)}_{\mu,\mathcal{D}}\|_{L^{p(\cdot)}(X)\to L^{q(\cdot)}(X)} \\ 
			&\lesssim [1]_{A_{p(\cdot),q(\cdot)}^{\alpha(\cdot)}(\mathcal{D})}  [1]_{A_{\frac{1}{\alpha(\cdot)}}(\mathcal{D})}    \|M_{\mu,\mathcal{D}}\|_{ L^{q'(\cdot)}(X)} \|M_{\mu,\mathcal{D}}\|_{ L^{p(\cdot)}(X)},
		\end{split}
	\end{equation*}
where the implicit constants are absolute and do not depend on any of the displayed data.\end{theorem}
The lower bounds and the dyadic reduction steps are straightforward, so the main challenge is the upper bound in Theorem \ref{thm:MainIntro2}. Our approach is through sparse domination adapted to variable fractionality.

The stopping time (Proposition \ref{prop:SDomSub}) is taken with respect to the corresponding variable fractional averages. Consequently, the major-subset condition is measured by the sub-additive set function $A\mapsto \mu(A)^{1-\langle\alpha\rangle_A}$ (Lemma \ref{lem:PsiSubAdd}, Proposition \ref{prop:SDomSub}). For constant fractionality this sparse notion is equivalent to ordinary sparsity after changing the sparseness parameter; its role is that it remains stable when $\alpha(\cdot)$ varies. We follow the resulting pointwise domination (Proposition \ref{prop:SDomAlpha}) with a new sparse Carleson-type embedding (Proposition \ref{prop:CarlesonEmbedding}) taking values in variable sequence spaces. These methods appear to be new in the context of variable exponent Lebesgue spaces.

We turn to variable Riesz potentials and start with the dyadic formulation on $\mathbb{R}^n.$ We say that a collection of cubes $\mathcal{D} = \cup_{k\in\mathbb{Z}}\mathcal{D}_k$ is a dyadic lattice (or, filtration) on $\mathbb{R}^n$ if $Q\cap P\in \{Q,P,\emptyset\}$, for $Q,P\in\mathcal{D}$, and $\mathcal{D}_k = \{Q\in\mathcal{D} : \ell(Q) = 2^k\}$ and $\mathbb{R}^n = \cup\mathcal{D}_k$, for all $k\in\mathbb{Z}.$
We will control the dyadic Riesz potential 
	\begin{align*}
	I^{\alpha(\cdot)}_{\mathcal{D}}f(x) := \sum_{Q\in\mathcal{D}} \frac{1_Q(x)}{|Q|^{1-\langle\alpha\rangle_Q}}\int_Q f
\end{align*}
with the smaller operator $M^{\alpha(\cdot)}_{\mathcal{D}}$
	and the following packing condition
 \begin{align}\label{eq:thm:DyadRieszNecINTRO}
 	C_{\downarrow}^{\alpha(\cdot)}(\mathcal{D}) := \sup_{Q_0\in\mathcal{D}}  \frac{\sum_{Q\in\mathcal{D}(Q_0)}  |Q|^{1+\langle \alpha\rangle _Q}}{|Q_0|^{1+\langle \alpha\rangle _{Q_0}}} < \infty,
 \end{align}
where $\mathcal{D}(Q_0)= \{Q\in\mathcal{D}: Q\subset Q_0\}.$
When $\alpha(\cdot)=\alpha>0$ is constant, this condition is automatic on dyadic lattices because the sum over descendants is geometric. For variable $\alpha(\cdot)$, the condition records the amount of scale summation that is no longer guaranteed by the value at the top cube alone. 
In the next theorem, under mild stability assumptions on the individual exponents, we show $C^{\alpha(\cdot)}_{\downarrow}(\mathcal{D})$'s finiteness and $M^{\alpha(\cdot)}_{\mathcal{D}}$'s boundedness to be both sufficient and necessary for the boundedness of the variable fractional Riesz potential. 
	\begin{theorem}\label{thm:MainIntro3} 
		Let  $p(\cdot),$  $q(\cdot),$ $1/\alpha(\cdot)\in \mathcal{P}(\mathbb R^n)$ be such that $1<p(\cdot)\leq q(\cdot)<\infty.$
        Let $\mathcal{D}$ be a dyadic lattice on $\mathbb{R}^n.$
	The following assertions hold. 
		\begin{itemize} 
			\item Suppose that $M_{\mathcal{D}}$ is bounded on $L^{q'(\cdot)}$ and $1\in A_{\frac{1}{\alpha(\cdot)}}(\mathcal{D}).$ If $C_{\downarrow}^{\alpha(\cdot)}(\mathcal{D}) < \infty$ and $M^{\alpha(\cdot)}_{\mathcal{D}}:L^{p(\cdot)}\to L^{q(\cdot)}$ is bounded, then $I^{\alpha(\cdot)}_{\mathcal{D}}:L^{p(\cdot)}\to L^{q(\cdot)}$ is bounded.
			\item  Suppose that $1\in A_{p(\cdot)}(\mathcal{D})\cap A_{q(\cdot)}(\mathcal{D})$ and $\alpha(\cdot) = \frac{1}{p(\cdot)}-\frac{1}{q(\cdot)}.$ If $I^{\alpha(\cdot)}_{\mathcal{D}}:L^{p(\cdot)}\to L^{q(\cdot)}$ is bounded, then $C_{\downarrow}^{\alpha(\cdot)}(\mathcal{D}) < \infty$ and $M^{\alpha(\cdot)}_{\mathcal{D}}:L^{p(\cdot)}\to L^{q(\cdot)}$ is bounded.
		\end{itemize}
	\end{theorem}
        \begin{remark}
The assumption
	$\alpha(\cdot) = \tfrac{1}{p(\cdot)}-\tfrac{1}{q(\cdot)}$  in the second bullet is made to avoid a side condition; we direct the reader to the brief Remark \ref{rem:AlphaRelRieszPot}.
Moreover, we refer to Theorem \ref{thm:RieszPot} in Section \ref{sect:dyadic2} for the full dyadic statement on a $\sigma$-finite measure space $(X,\mu)$ and with more degrees of freedom. 
\end{remark}

 The main challenge in Theorem \ref{thm:MainIntro3} is to prove the first bullet.
 We initiate the proof by stopping with respect to the averages $\langle |f|\rangle_Q$, resulting in the standard principal stopping cubes.
 The remaining scale summation is then handled by a John--Nirenberg type inequality adapted to the packing condition $C^{\alpha(\cdot)}_{\downarrow}$, together with the self-improving property of $M_{\mathcal{D}}$ on $L^{q'(\cdot)}.$

 Notably, the stopping condition used to control the Riesz potential $I^{\alpha(\cdot)}$ is different from the one used for $M^{\alpha(\cdot)}$.
 For the maximal operator the stopping time is adapted to the variable fractional averages, while for the Riesz potential the fractionality is carried by the subsequent summation over scales.
	
To pass from the dyadic Riesz potential result to the non-dyadic one, we will use the condition \eqref{eq:thm:DyadRieszNecINTRO} over all dyadic lattices and a separate upward packing condition
	\begin{align}
	C_{\downarrow}^{\alpha(\cdot)} &:=	\sup_{\mathcal{D}}C_{\downarrow}^{\alpha(\cdot)}(\mathcal{D}), \\
	C_{\uparrow}^{\alpha(\cdot)} &:=
	\sup_{\mathcal{D}}\sup_{Q_0\in\mathcal{D}}
	\sum_{Q_0\subset Q\in\mathcal{D}}
	\frac{|Q_0|^{1-\langle\alpha\rangle_{Q_0}}}{|Q|^{1-\langle\alpha\rangle_Q}}.
	\end{align}
The condition $C_{\uparrow}^{\alpha(\cdot)}$ controls any dyadic Riesz potential $I^{\alpha(\cdot)}_{\mathcal{D}}$ by $I^{\alpha(\cdot)}$; recall its definition \eqref{eq:defn:VarRiesz} above. Moreover, $C_{\downarrow}^{\alpha(\cdot)}$ controls, by itself, simultaneously all dyadic Riesz potentials.

\begin{theorem}\label{thm:MainIntro4}
Let $p(\cdot), q(\cdot), 1/\alpha(\cdot)\in \mathcal{P}(\mathbb{R}^n)$ be such that $1<p(\cdot)\leq q(\cdot)<\infty$. 
The following assertions hold.
\begin{itemize}
\item Suppose that $M$ is bounded on $L^{q'(\cdot)}$ and $1\in A_{\frac{1}{\alpha(\cdot)}}(\mathbb{R}^n).$ If $C_{\downarrow}^{\alpha(\cdot)} <\infty$ and
\(
M^{\alpha(\cdot)}:L^{p(\cdot)}(\mathbb R^n)\to L^{q(\cdot)}(\mathbb R^n)
\)
is bounded,
 then 
\(
I^{\alpha(\cdot)}:L^{p(\cdot)}(\mathbb R^n)\to L^{q(\cdot)}(\mathbb R^n)
\)
is bounded.

\item Suppose that $1\in A_{p(\cdot)}(\mathbb{R}^n)\cap A_{q(\cdot)}(\mathbb{R}^n)$ and that $1\in A_{\frac{1}{\alpha(\cdot)}}(\mathbb{R}^n)$, $\alpha(\cdot) = \tfrac{1}{p(\cdot)}-\tfrac{1}{q(\cdot)}$ and $C_{\uparrow}^{\alpha(\cdot)} <\infty.$ If \(
I^{\alpha(\cdot)}:L^{p(\cdot)}(\mathbb R^n)\to L^{q(\cdot)}(\mathbb R^n)
\)
is bounded, then 
$C_{\downarrow}^{\alpha(\cdot)} <\infty$ and \(
M^{\alpha(\cdot)}:L^{p(\cdot)}(\mathbb R^n)\to L^{q(\cdot)}(\mathbb R^n)
\)
is bounded.
\end{itemize}
\end{theorem}

Beyond the scale of variable exponent Lebesgue spaces, we also study the action of $M^{\alpha(\cdot)},I^{\alpha(\cdot)}$ (and their dyadic variants) on versions of variable exponent Morrey spaces that are defined through the norms
\[
\|f\|_{\mathcal M^{r(\cdot)}_{p(\cdot)}(\mathcal B)}
:=
\sup_{B\in\mathcal{B}}
|B|^{\frac1{r_B}-\frac1{p_B}}
\|1_Bf\|_{L^{p(\cdot)}},
\]
where $\mathcal{B}$ can be taken to be the collection of all balls, for example, and the harmonic means are defined through the formula $s_B^{-1}:=\frac1{|B|}\int_B\frac1{s(x)}\,\ud x.$ 
Similar spaces with pointwise evaluations $r(x)^{-1}$ and under the $\mathcal{LH}_0$ assumption have been studied on bounded domains by Almeida--Hasanov--Samko \cite{AlmHasSam08}; we direct the reader to Appendix \ref{sect:logHolder} for a verification of the equivalence of the two definitions under the $\mathcal{LH}_0$ assumption.
 We give boundedness criteria for $M^{\alpha(\cdot)}$ and $I^{\alpha(\cdot)}$ on $\mathcal M^{r(\cdot)}_{p(\cdot)}(\mathcal B)$ in Theorem \ref{thm:MorreyIff}, which is stated in the body text.
\subsection{Plan of the paper}  
Section \ref{sect:preliminaries} contains well-known preliminaries;
there we also prove the sparse-Carleson type embedding theorem (Proposition \ref{prop:CarlesonEmbedding}).
Section \ref{sect:dyadic} constitutes the technical core of the paper and it contains the proofs of all of our dyadic results.  The dyadic results are recorded for their cleanliness and easy adjustability; notably, they are formulated with respect to general dyadic collections of sets. 
In subsection \ref{sect:dyadic1} we prove results for $M^{\alpha(\cdot)}_{\mu,\mathcal{D}}$, in subsection \ref{sect:dyadic2} for $I^{\alpha(\cdot)}_{\mu,\mathcal{D}}$, and in subsection \ref{sect:dyadic3} we prove boundedness criteria for these operators to be $\mathcal{M}^{r(\cdot)}_{p(\cdot)}(\mathcal{D})$-to-$\mathcal{M}^{s(\cdot)}_{q(\cdot)}(\mathcal{D})$ bounded. In Section \ref{sect:cont-to-dyadic} we deduce the non-dyadic results from their dyadic counterparts. 
Lastly, in Section \ref{sect:logHolder} we briefly discuss how our results look under logarithmic H\"older continuity assumptions.

\subsection*{Acknowledgements and AI tool disclosure}
We thank E. Lorist and T. H\"anninen for discussions on the paper. 
F.G. was supported by the Dutch Research Council (NWO)
through grant VI.Vidi.223.019.
T.O. was supported by the Research Council of Finland, projects 
360184 and 371637. He is a member of the Finnish Centre of Excellence in Randomness and Structures.

The authors used OpenAI's ChatGPT/Codex to identify typographical errors and potential omissions or inconsistencies in the reasoning of a draft manuscript. All revisions and editorial decisions were made by the authors, who take full responsibility for the final content.


\section{Preliminaries}\label{sect:preliminaries}	
	\subsection{Banach function spaces and duality}\label{sect:BFSproperties}
Variable exponent Lebesgue spaces are examples of Banach function spaces. We recall the basic notions needed below; for a detailed discussion of these spaces, see \cite{LorNie2024}.

    	Let $(X,\mu)$ be a $\sigma$-finite measure space, and let $L^0(X)$ denote the space of measurable functions on $(X,\mu)$, identified up to equality almost everywhere. A normed function lattice $Y$ over $(X,\mu)$ is a normed linear subspace of $L^0(X)$, identified up to equality almost everywhere, with the ideal property: if $g\in Y$ and $f\in L^0(X)$ satisfy $|f|\le |g|$ almost everywhere, then $f\in Y$ and $\|f\|_Y\le \|g\|_Y$. In particular, if $f\in Y$, then $|f|\in Y$ and $\||f|\|_Y=\|f\|_Y$. 
    If $Y$ is complete, we call it a Banach function lattice.	
    If, in addition, $Y$ satisfies the saturation property, namely that for every measurable set $E\subset X$ with $\mu(E)>0$ there exists a measurable subset $F\subset E$ with $\mu(F)>0$ and $1_F\in Y$, then $Y$ is called a Banach function space over $(X,\mu)$.

    We say that a Banach function lattice $Y$ has the Fatou property if, whenever $0\le f_n\uparrow f$ almost everywhere and $\sup_n\|f_n\|_Y<\infty$, then $f\in Y$ and
\[
    \|f\|_Y=\sup_{n\in\mathbb N}\|f_n\|_Y=\lim_{n\to\infty}\|f_n\|_Y.
\]

For any normed function lattice \(Y\) over \((X,\mu)\), we define its
associate class by
\[
    \|g\|_{Y'}:=\sup_{\|f\|_Y\le 1}\|fg\|_{L^1(X)},
    \qquad
    Y':=\{g\in L^0(X):\|g\|_{Y'}<\infty\}.
\]
This definition always gives the H\"older-type estimate
\[
    \int_X |fg|\,\mathrm d\mu\le \|f\|_Y\|g\|_{Y'}
\]
for $f\in Y$ and $g\in Y'$. For a general normed function lattice this may
only define a seminorm on \(Y'\). If \(Y\) is saturated, in particular if
\(Y\) is a Banach function space, then it is a norm; in that case \(Y'\) is
the K\"othe dual, or associate space, of \(Y\).
A Banach function space $Y$ has the Fatou property if and only if $Y''=Y$ with equal norm. Equivalently, for every $f\in Y$,
\[
    \|f\|_Y=\sup_{\|g\|_{Y'}\le 1}\int_X |fg|\,\mathrm d\mu.
\]

   \subsection{Notation regarding constants}
    Before turning to variable exponent Lebesgue spaces, we fix notation for implicit constants used throughout the paper. 
    \begin{itemize}
    \item A constant $C>0$ is absolute if it does not depend on the parameters or auxiliary objects appearing in the statement, e.g. $C=1,2,2^{10}.$
    \item We write $A\lesssim B$ if $A\leq CB$ for some absolute constant $C>0.$ 
    \item We write $A\sim B$ if $A\lesssim B$ and $B\lesssim A.$ 
    \item Subscripted or parenthesized data on constants $(C_{a,b}, C(a,b))$ and quantifiers $(\lesssim_{a,b}, \sim_{a,b})$ signify their dependence on that data. 
    \end{itemize}

  \subsection{Variable exponent Lebesgue spaces} 
    We gather in the next list the basic definitions regarding variable exponent Lebesgue spaces.
\begin{itemize}
	\item We assume throughout the article that the underlying measure space $(X,\mu)$ is $\sigma$-finite.
  \item The set of variable exponents $\mathcal P(X,\mu)$ consists of measurable functions $p(\cdot):X\rightarrow[1,\infty]$.
  \item 
Given $p(\cdot)\in \mathcal{P}(X,\mu)$, the norm of the variable exponent Lebesgue space $L^{p(\cdot)}(X,\mu)$ is
\begin{equation*}
  \|f\|_{L^{p(\cdot)}(X,\mu)}:=\inf\{\lambda>0:\rho_{p(\cdot)}(f/\lambda)\leq 1\},
\end{equation*}
where the modular is given by 
\begin{equation*}
			\rho_{p(\cdot)}(f):=
			\int_{X\setminus \{p(\cdot)=\infty\}} |f(x)|^{p(x)}\,\mathrm d\mu(x)
			+
			\| f\|_{L^\infty(\{p(\cdot)=\infty\})}.
		\end{equation*}
We write $f\in L^{p(\cdot)}(X,\mu)$ if $\|f\|_{L^{p(\cdot)}(X,\mu)}<\infty.$
\item The conjugate exponent $p'(\cdot)\in\mathcal{P}(X,\mu)$ is given by the relation 
$$
\frac{1}{p(\cdot)}+ \frac{1}{p'(\cdot)}=1,
$$
with the convention $\frac1\infty=0$.
 \item  The harmonic mean of $p(\cdot)\in\mathcal P(X, \mu)$ over a set of finite measure $B\subset X$ is defined with 
 \begin{equation*}
 	\frac 1{p_B^\mu}:=\frac1{\mu(B)}\int_B\frac 1{p(x)}\mathrm d\mu(x).
 \end{equation*}
\item Given $p(\cdot)\in\mathcal{P}(X,\mu)$ and a set $E\subset X$ we denote
   $$
   p_{-}(E)=\essinf_{E}p(\cdot),\qquad p_{+}(E)=\esssup_{E}p(\cdot).
   $$ 
        \item We often write $\mathcal{P}(X,\mu) = \mathcal{P}(X)$, $\|f\|_{L^{p(\cdot)}(X,\mu)} = \|f\|_{p(\cdot)}$ and $p_B^{\mu} = p_B$, and so on, when the meaning is clear from the context. 
\end{itemize}
		Combining Theorem 3.2.7 and Theorem 2.3.17(d) of \cite{DienHHR2011Book} gives the following.
		If $(X,\mu)$ is $\sigma$-finite, then for any $p(\cdot)\in\mc P(X,\mu)$, the space $L^{p(\cdot)}(X,\mu),$ endowed with the norm $\|f\|_{p(\cdot)},$ is a Banach function space. 
We will also use the Fatou property of $L^{p(\cdot)}$ in several limiting arguments.
    \begin{proposition}[{\cite{DienHHR2011Book}*{Theorem 3.2.13}}]\label{prop:Lp=Fatou}
Suppose $(X,\mu)$ is a $\sigma$-finite measure space. Let $p(\cdot)\in\mathcal P(X,\mu)$. Then $L^{p(\cdot)}(X,\mu)$ is a Banach function space satisfying the Fatou property.
\end{proposition}

	We next go through the variable exponent estimates that will be used repeatedly. In the cases where we were not able to find a direct citation, a full proof has been provided.
	\begin{lemma}[{\cite{DienHHR2011Book}*{Theorem 3.2.20}}]
		\label{lem:Holder}
		Let $p(\cdot) \in \mathcal{P}(X, \mu).$ Then for all $f \in L^{p(\cdot)}(X, \mu)$ and $g \in L^{p'(\cdot)}(X, \mu)$ there holds that 
		\begin{equation*}
			\int_X |f(x) g(x)|  \mathrm d\mu(x) \leq 4\| f \|_{p(\cdot)} \| g \|_{p'(\cdot)}.
		\end{equation*}
	\end{lemma}
	\begin{rem}[{\cite{CruzFio2013Book}*{Theorem 2.26}}] The constant $4$ appears as a consequence of bounding the quantity
\begin{equation*}
K_{p(\cdot)}=\Bigl(\frac1{\essinf p}-\frac1{\esssup p}+1\Bigr)\|1_{X_*}\|_{L^\infty}+\|1_{X_\infty}\|_{L^\infty}+\|1_{X_1}\|_{L^\infty}\leq 4,
\end{equation*}
where $X_\infty= \{p(\cdot)=\infty\}$, $X_1= \{p(\cdot)=1\}$ and $X_*=X\setminus(X_1\cup X_\infty)$.
		Instead of carrying the more complicated left-hand side around, the absolute constant $4$ is good enough for our purposes.
	\end{rem}
Closely related to H\"older's inequality is the following duality formula.  The constants are consequences of those in the preceding lemma.
	
	\begin{lemma}[{\cite{DienHHR2011Book}*{Corollary 3.2.14}}]\label{lem:dualityFormula}
		Let $p(\cdot)\in \mc P(X, \mu)$. Then
		\begin{equation*}
\frac 14 \sup_{\|g\|_{L^{p'(\cdot)}}=1}\int_X |fg|\,\mathrm d \mu\leq\|f\|_{L^{p(\cdot)}}\leq 4 \sup_{\|g\|_{L^{p'(\cdot)}}=1}\int_X |fg|\,\mathrm d \mu.		\end{equation*}
	\end{lemma}

	\subsection{Harmonic mean and the Muckenhoupt condition}
	In the context of variable exponents, the following Muckenhoupt-type condition originates in Kopaliani \cite{Kop2007},
	\begin{align}\label{eq:MuckForExponent}
		 [1]_{A_{p(\cdot)}(\mathcal{B})}=\sup_{B\in\mathcal{B}}\frac{\|1_B\|_{L^{p(\cdot)}}\|1_B\|_{L^{p'(\cdot)}}}{\mu(B)}.
	\end{align}
	Emphasizing the role of $p(\cdot)$ this condition is also often denoted by $[p(\cdot)]_{K^{0}(\mathcal{B})} :=[1]_{A_{p(\cdot)}(\mathcal{B})};$ however, we use the convention  \eqref{eq:MuckForExponent}.
	When $\mathcal{B} = \{B\}$ is a singleton, we write $[1]_{A_{p(\cdot)}(B)}.$
    For a collection $\mathcal Q$ of measurable sets of finite positive measure, we write $\langle \alpha\rangle_Q:=\mu(Q)^{-1}\int_Q\alpha\,\ud\mu$ and
\[
[1]_{A_{p(\cdot),q(\cdot)}^{\alpha(\cdot)}(\mathcal Q)}
:=
\sup_{Q\in\mathcal Q}
\mu(Q)^{\langle\alpha\rangle_Q-1}
\|1_Q\|_{q(\cdot)}
\|1_Q\|_{p'(\cdot)}.
\]
		Norms of indicators can be estimated with the harmonic average and the Muckenhoupt-type condition by the next lemma. A proof in the Euclidean setting can be found in Cruz-Uribe and Roberts \cite{CruTro2024}*{Proposition 3.8}.
\begin{lemma}\label{lem:IndicatorKopaliani}
		Let $p(\cdot)\in \mc P(X, \mu)$ and suppose $B\subset X$ is such that $0<\mu(B)<\infty$. Then
\begin{equation}\label{eq: cube norm averaging}
			\frac 1{5}\mu(B)^\frac 1{p_{B}^\mu}\leq\|1_B\|_{p(\cdot)}\leq 5  [1]_{A_{p(\cdot)}(B)}  \mu(B)^\frac 1{p_{B}^\mu}.
		\end{equation}
	\end{lemma}
	\begin{proof}
		Since the exponential is a convex function, Jensen's inequality yields
		\begin{align*}
			\mu(B)^{-\frac 1{(p')_B^\mu}}
			&=\exp\Bigl(-\frac1{\mu(B)}\int_B \frac1{p'(x)}\mathrm d\mu(x) \log \mu(B)\Bigr)\\
			&\leq\frac1{\mu(B)}\int_B \exp\Bigl(-\frac1{p'(x)}\log \mu(B)\Bigr)\mathrm d\mu(x)
			= \frac1{\mu(B)}\int_B \mu(B)^{-\frac 1{p'(x)}}\mathrm d\mu(x).
		\end{align*}
		Multiplying by $\mu(B)$ and using $1/p_B^\mu+1/(p')_B^\mu=1$ and denoting $B_1:=B\cap\{p(\cdot)=1\}$ we have  
		\begin{align*}
					\mu(B)^{\frac 1{p_B^\mu}}\leq \int_B \mu(B)^{-\frac 1{p'(x)}}\mathrm d\mu(x) =    \int_{B_1} \mu(B)^{-\frac 1{p'(x)}}\mathrm d\mu(x)+\int_{B\setminus B_1}\mu(B)^{-\frac1{p'(x)}}\,\mathrm d\mu(x).
		\end{align*}
		We next bound both terms from above as $\lesssim \|1_B\|_{p(\cdot)}$ and in doing so prove the left inequality in \eqref{eq: cube norm averaging}. For the first term, noting that $p'(\cdot) = \infty$ on $B_1$, we have 
		\[
		 \int_{B_1} \mu(B)^{-\frac 1{p'(x)}}\mathrm d\mu(x)=  \mu(B_1) = \|1_{B_1}\|_{L^{p(\cdot)}(B_1)} \leq  \|1_{B}\|_{L^{p(\cdot)}}.
		\]
		For the second term, by H\"older's inequality, Lemma \ref{lem:Holder}, we have 
		\[
		\int_{B\setminus B_1}\mu(B)^{-\frac1{p'(x)}}\,\mathrm d\mu(x) \leq 4\|1_{B\setminus B_1}\|_{L^{p(\cdot)}} \|\mu(B)^{-\frac1{p'(\cdot)}}\|_{L^{p'(\cdot)}(B\setminus B_1)} \leq 4\|1_{B}\|_{L^{p(\cdot)}}
		\]
		where we used the bound $\|\mu(B)^{-\frac1{p'(\cdot)}}\|_{L^{p'(\cdot)}(B\setminus B_1)}  \leq 1;$ indeed, since $p'(\cdot) < \infty$ almost everywhere on $B\setminus B_1$ we have that 
		\[
			\rho_{p'(\cdot)}\bigl( \mu(B)^{-\frac1{p'(\cdot)}}\bigr)
			=
			\int_{B\setminus B_1}\mu(B)^{-1}\,\mathrm d\mu
			\leq 1;
		\]
		 and thus directly from the definition of the norm that $\|\mu(B)^{-\frac1{p'(\cdot)}}\|_{L^{p'(\cdot)}(B\setminus B_1)}  \leq 1.$

Then we verify the right bound.
Applying the left bound with $p'(\cdot)$ we have 
$	\|1_B\|_{p'(\cdot)} \geq \mu(B)^{\frac 1{(p')_B^\mu}}/5,$ and using the definition of $[1]_{A_{p(\cdot)}(B)}$ we get  
\begin{equation*}
	\|1_B\|_{p(\cdot)} = [1]_{A_{p(\cdot)}(B)} \frac{\mu(B)}{	\|1_B\|_{p'(\cdot)}}\leq 
	 5[1]_{A_{p(\cdot)}(B)}\mu(B) {\mu(B)^{-\frac 1{(p')_B^\mu}}}
	=5[1]_{A_{p(\cdot)}(B)}\,\mu(B)^{\frac1{p_B^\mu}},
\end{equation*}
proving the second inequality of \eqref{eq: cube norm averaging}.
\end{proof}

	The next lemma is a formal generalisation of \cite{diening_maximal_2009}*{Lemma 6.1} to arbitrary measure spaces.
	\begin{lemma}\label{lem:ConstAveBound}
		Let $(X,\mu)$ be any measure space, let $B$ be any measurable set with
		$0<\mu(B)<\infty$ and let $p(\cdot):X\to[1,\infty)$ be measurable.
		Then, for every $c>0$, there holds that 
		\begin{equation*}
			\Bigl(\frac{c}{2}\Bigr)^{p_{B}^{\mu}}
			\leq \frac{1}{\mu(B)}\int_B c^{p(x)}\,\mathrm d\mu(x).
		\end{equation*}
	\end{lemma}
	\begin{proof}
Let $\nu:=\mu(B)^{-1}\mu|_B$. Then $\nu$ is a probability measure on $B$ and
\(
    \int_B \frac1{p(x)}\,\mathrm d\nu(x)=\frac1{p_B^\mu}.
\)
Fix $c>0$ and define $\Phi:(0,1]\to(0,\infty)$ by
\(
    \Phi(s):=s c^{1/s}.
\)
Since $\Phi''(s)=(\log c)^2s^{-3}c^{1/s}\geq0$, the function $\Phi$ is convex. Hence Jensen's inequality gives
\[
    \frac1{p_B^\mu}c^{p_B^\mu}
    =\Phi\left(\int_B\frac1{p(x)}\,\mathrm d\nu(x)\right)
    \leq \int_B \Phi\left(\frac1{p(x)}\right)\,\mathrm d\nu(x)
    =\int_B \frac1{p(x)}c^{p(x)}\,\mathrm d\nu(x).
\]
Since $p(x)\geq1$, this is bounded above by
\[
    \int_B c^{p(x)}\,\mathrm d\nu(x)
    =\frac1{\mu(B)}\int_B c^{p(x)}\,\mathrm d\mu(x).
\]
Finally, $t\leq 2^t$ for $t\geq1$, so
\(
    (\tfrac c2)^t\leq \tfrac1t c^t.
\)
Applying this with $t=p_B^\mu$ gives the desired estimate.
\end{proof}

\subsection{Embeddings of sequential spaces}
Consider the measure space $(X,\mu)=(\mathbb{N},\#)$, where $\#$ denotes the counting measure. 
In this discrete setting, a variable exponent $p(\cdot)$ is a sequence 
$\{p_n\}_{n\in\mathbb{N}}$ and we write
\(
\ell^{p(\cdot)}:=\ell^{p_n} := L^{p(\cdot)}(\mathbb{N},\#).
\)
For a collection $\mathcal B$ of measurable sets with $0<\mu(B)<\infty$, write
\[
M_{\mathcal B}h(x):=\sup_{B\in\mathcal B}\frac{1_B(x)}{\mu(B)}\int_B|h|\,\ud\mu.
\]
The following is a Carleson-type embedding theorem in the variable exponent setting. 
\begin{proposition}\label{prop:CarlesonEmbedding}
	Let $s(\cdot)\in \mathcal{P}(X,\mu)$ be such that $1< s(\cdot)<\infty$ almost everywhere. Let $\mathcal{B}$ be a countable collection of measurable sets such that $0<\mu(B)<\infty$ for every $B\in\mathcal{B}$. Let $E(B)\subset B$ be pairwise disjoint measurable subsets such that $0<\mu(E(B))<\infty$ and let $a_B\geq 0$ for each $B\in\mathcal{B}.$
	Then, there holds that 
	\begin{align}\label{eq:prop:CarlesonEmbedding}
		\left\| \left\{  a_B \mu(E(B))^\frac1{s^\mu_{E(B)}} \right\}_{B\in\mathcal{B}}\right\|_{\ell^{s^\mu_{E(B)}}}   \lesssim	\Bigl\| \sum_{B\in\mathcal{B}}a_B1_{E(B)} \Bigr\|_{L^{s(\cdot)}(X,\mu)}.
	\end{align}
Moreover, the reverse bound holds in the following form 
	\begin{align}\label{eq:prop:CarlesonEmbedding2}
\Bigl\| \sum_{B\in\mathcal{B}}a_B1_{E(B)} \Bigr\|_{L^{s(\cdot)}(X,\mu)} 
	\lesssim
	\|M_{\mathcal{B}}\|_{L^{s'(\cdot)}(X,\mu)}
	\left\|  \left\{  \frac{\mu(B)}{\mu(E(B))}   \cdot a_B \mu(E(B))^{1/s^{\mu}_{E(B)}} \right\}\right\|_{\ell^{s^{\mu}_{E(B)}}(\mathcal{B})}.
\end{align}
\end{proposition}
\begin{proof}
	We verify \eqref{eq:prop:CarlesonEmbedding} first. By homogeneity we may assume that 
	$$
		\Bigl\| \sum_{B\in\mathcal{B}}a_B1_{E(B)} \Bigr\|_{L^{s(\cdot)}(X,\mu)}= 1,
	$$
	and it is enough, under this standing assumption, to show that 
	\begin{align}\label{eq:CarlesonEmbed1}
		\rho_{\{s^\mu_{E(B)}\}_{B\in\mathcal{B}}}\Bigl(\Bigl\{\frac{a_B}{2}\mu\bigl(E(B)\bigr)^\frac1{s^\mu_{E(B)}}\Bigr\}_{B\in\mathcal{B}}\Bigr) \leq 1.
	\end{align}
	We have by Lemma \ref{lem:ConstAveBound} that 
	\begin{equation}\label{eq:homog6}
		\begin{split}
			\sum_{B\in\mathcal{B}}	\mu\bigl(E(B)\bigr) \left(\frac{a_B }{2}\right)^{s_{E(B)}^{\mu}}  &\leq	\sum_{B\in\mathcal{B}}\mu\bigl(E(B)\bigr) \frac{1}{\mu\bigl(E(B)\bigr)}   \int_{E(B)}a_B^{s(x)}\mathrm d\mu(x) \\
			&=\int_X\Bigl(  \sum_{B\in\mathcal{B}}a_B1_{E(B)}(x) \Bigr)^{s(x)}\mathrm d\mu(x) = 1.
		\end{split}
	\end{equation}
This proves \eqref{eq:CarlesonEmbed1}, and hence \eqref{eq:prop:CarlesonEmbedding}.

Then we verify \eqref{eq:prop:CarlesonEmbedding2}. By Lemma \ref{lem:dualityFormula} and pairwise disjointness of the sets $E(B)$ we have 
\[
\Bigl\| \sum_{B\in\mathcal{B}}a_B1_{E(B)} \Bigr\|_{L^{s(\cdot)}(X,\mu)} \lesssim \sup_{\|g\|_{{s'(\cdot)}}\leq 1}\sum_{B\in\mathcal{B}}a_B \int_{E(B)}g\ud\mu.
\]
By Lemma \ref{lem:Holder} we bound the right-hand side as 
\begin{multline*}
	\sum_{B\in\mathcal{B}}a_B \int_{E(B)}g\ud\mu \leq \sum_{B\in\mathcal{B}} \frac{\mu(B)}{\mu(E(B))} \mu(E(B))^{1/s^{\mu}_{E(B)}}a_B  \mu(E(B))^{1/(s')^{\mu}_{E(B)}} \langle |g|\rangle_{B} \\
	\lesssim 
	\left\|  \left\{  a_B\frac{\mu(B)}{\mu(E(B))} \mu(E(B))^{1/s^{\mu}_{E(B)}} \right\}\right\|_{\ell^{s^{\mu}_{E(B)}}(\mathcal{B})}	\left\|  \left\{  \mu(E(B))^{1/(s')^{\mu}_{E(B)}} \frac{\langle |g|\rangle_{B}}{2} \right\}\right\|_{\ell^{(s')^{\mu}_{E(B)}}(\mathcal{B})}.
\end{multline*}
To bound the right-most term, we have by Lemma \ref{lem:ConstAveBound} that 
\begin{align*}
	\sum_{B\in\mathcal{B}}\mu(E(B)) \Bigl(\frac{\langle |g|\rangle_{B}}{2} \Bigr)^{(s')^{\mu}_{E(B)}}  \leq 	\sum_{B\in\mathcal{B}}\int_{E(B)} \langle |g|\rangle_{B}^{s'(x)}\ud\mu(x) \leq \int_X\left(M_{\mathcal{B}}g(x)\right)^{s'(x)}\ud\mu(x).
\end{align*}
By homogeneity it follows that 
\[
\left\|  \left\{  \mu(E(B))^{1/(s')^{\mu}_{E(B)}} \frac{\langle |g|\rangle_{B}}{2} \right\}\right\|_{\ell^{(s')^{\mu}_{E(B)}}(\mathcal{B})} \leq  \|M_{\mathcal{B}}g\|_{L^{s'(\cdot)}(X,\mu)}
\]
and chaining the above estimates together concludes the proof.
\end{proof}
Proposition \ref{prop:CarlesonEmbedding} with the choice $a_B=\ave{|f|}_{\mu, B}$ gives the following corollary.
\begin{corollary}\label{corollary:carlesonembeddingapplication}
 Let $s(\cdot)\in \mathcal{P}(X,\mu)$ be such that $1< s(\cdot)<\infty$ almost everywhere. Let $\mathcal{B}$ be a countable collection of measurable sets such that $0<\mu(B)<\infty$ for every $B\in\mathcal{B}$, with associated pairwise disjoint measurable subsets $E(B)\subset B$ satisfying $0<\mu(E(B))<\infty.$
	Then, there holds that 
\begin{align}\label{eq:cor:CarlesonEmbeddingapplication}
		\left\| \left\{  \ave{|f|}_{\mu,B}\mu(E(B))^\frac1{s^\mu_{E(B)}} \right\}_{B\in\mathcal{B}}\right\|_{\ell^{s^\mu_{E(B)}}}   \lesssim 	\left\|M_{\mathcal{B}}\right\|_{L^{s(\cdot)}(X,\mu)}\left\| f\right\|_{L^{s(\cdot)}(X,\mu)}.
	\end{align}
\end{corollary}
\begin{proof} Applying the left bound of \eqref{eq:prop:CarlesonEmbedding} with the choice $a_B=\langle |f|\rangle_{\mu,B}$ gives
	\[
		\left\| \left\{  \ave{|f|}_{\mu,B}\mu(E(B))^\frac1{s^\mu_{E(B)}} \right\}_{B\in\mathcal{B}}\right\|_{\ell^{s^\mu_{E(B)}}}  \lesssim \Bigl\| \sum_{B\in\mathcal{B}}\ave{|f|}_{\mu, B}1_{E(B)} \Bigr\|_{L^{s(\cdot)}(X,\mu)} \leq	\left\|  M_{\mathcal{B}}f\right\|_{L^{s(\cdot)}(X,\mu)},
	\]
    where we have used the pointwise estimate $\langle |f|\rangle_{\mu,B}\leq M_\mathcal{B}f$ and the pairwise disjointness of $E(B)$ in the last estimate.
\end{proof}

We will also use the following embedding of sequential variable exponent spaces. 
\begin{lemma}\label{lem:sequential}
	Let $p(\cdot) = \{p_n\}_n$ and $q(\cdot) = \{q_n\}_n$ satisfy $1\leq p_n\leq q_n < \infty,$ for all indices $n$. Then, 
	$$
	\|\cdot\|_{\ell^{\infty}}\leq \|\cdot\|_{\ell^{q(\cdot)}}\leq \|\cdot\|_{\ell^{p(\cdot)}}.
	$$
\end{lemma}
\begin{proof}
	Let \(\lambda>0\) be such that
	\(
	\rho_{q(\cdot)}(\tfrac x\lambda)\le 1.
	\)
	Then for every \(k\), we have
	\(
	|\tfrac{x_k}{\lambda}|^{q_k}\le 1,
	\)
	and hence \(|x_k|\le \lambda\). Taking the supremum over \(k\), and then the infimum over all such \(\lambda\), gives
	the first inequality of the statement.
	For the second inequality, let \(\lambda>\|x\|_{\ell^{p(\cdot)}}\). Then
	\(
	\rho_{p(\cdot)}(x/\lambda)\le 1.
	\)
	In particular \(|x_n|/\lambda\le 1\) for all \(n\). Since \(p_n\le q_n\), we have
	\(
	\left|\frac{x_n}{\lambda}\right|^{q_n}
	\le
	\left|\frac{x_n}{\lambda}\right|^{p_n}
	\)
	and thus
	\[
	\rho_{q(\cdot)}(x/\lambda)
	=
	\sum_n \left|\frac{x_n}{\lambda}\right|^{q_n}
	\le
	\sum_n \left|\frac{x_n}{\lambda}\right|^{p_n}
	=
	\rho_{p(\cdot)}(x/\lambda)
	\le 1.
	\]
Taking the infimum over
	\(\lambda>\|x\|_{\ell^{p(\cdot)}}\) yields $	\|x\|_{\ell^{q(\cdot)}}\le \|x\|_{\ell^{p(\cdot)}}.$
\end{proof}

\subsection{Variable exponent Morrey spaces}
We next give a definition of Morrey spaces where the fractionality of the scaling is allowed to change, leading to Morrey spaces with variable fractionality. 

\begin{definition}\label{defn:MorreyNorm}
	Let $\mathcal B$ be a nonempty collection of measurable sets on a $\sigma$-finite measure space $(X,\mu)$ such that
	 \(0<\mu(B)<\infty\) for every
	\(B\in\mathcal B\). Assume that \(U:=\bigcup_{B\in\mathcal B}B\)
	is measurable and that there is a countable subfamily
	\(\mathcal B_0\subset\mathcal B\) such that
	\(\mu(U\setminus\bigcup_{B\in\mathcal B_0}B)=0\).
	Let $p(\cdot), r(\cdot)\in \mathcal P(U,\mu)$.
	We define $\mathcal M^{r(\cdot)}_{p(\cdot)}(\mathcal B,\mu)$ to consist of all measurable functions on \(U\),
	identified up to equality almost everywhere, 
	such that
	\begin{align}\label{eq:defn:Morrey}
		\|f\|_{\mathcal M^{r(\cdot)}_{p(\cdot)}(\mathcal B)} :=	\|f\|_{\mathcal M^{r(\cdot)}_{p(\cdot)}(\mathcal B,\mu)} 
		:=
		\sup_{B\in\mathcal B}
		\mu(B)^{\frac1{r_B^\mu}-\frac1{p_B^\mu}}
		\|1_B f\|_{L^{p(\cdot)}} < \infty.
	\end{align}
\end{definition}
Definition \ref{defn:MorreyNorm} is a variable analogue of the usual Morrey space norm originating in \cite{Morrey1938}. Indeed, when the collection $\mathcal{B}$ consists of all balls and the exponents are constant, the factor becomes $\mu(B)^{1/r-1/p},$ recovering the classical Morrey spaces. Under the hypothesis that $p(\cdot),r(\cdot)\in \mathcal{LH}_0(X)$ (see Section \ref{sect:logHolder} for the definition) similar spaces have been considered on bounded domains $X$ in \cite{AlmHasSam08}. In Section \ref{sect:logHolder} we verify the equivalence of Definition \ref{defn:MorreyNorm} with that given in \cite{AlmHasSam08} provided that $p(\cdot),r(\cdot)\in \mathcal{LH}_0(X)$.

\begin{proposition}\label{prop:MorreyProperties}
	Let \(\mathcal B,p(\cdot),r(\cdot)\) be as in Definition
	\ref{defn:MorreyNorm}, and set \(U:=\bigcup\mathcal B\). 
	Then, $\mathcal M^{r(\cdot)}_{p(\cdot)}(\mathcal B,\mu)$ is a Banach
	function lattice over \((U,\mu)\) and has the Fatou property.
    If, in addition, \(p(\cdot)\leq r(\cdot)\) almost everywhere on \(U\) and
	\([1]_{A_{p(\cdot)}(\mathcal B)}<\infty\),
	then \(\mathcal M^{r(\cdot)}_{p(\cdot)}(\mathcal B,\mu)\) is saturated,
	and hence is a Banach function space over \((U,\mu)\).

\end{proposition}

\begin{proof}
	Since \(U\) is measurable and \(X\) is \(\sigma\)-finite, the measure space
	\((U,\mu)\) is \(\sigma\)-finite. Hence Proposition \ref{prop:Lp=Fatou}
	applies to \(L^{p(\cdot)}(U,\mu)\).

	\emph{Normed lattice.}
	The absolute homogeneity, triangle inequality, and lattice property of the
	Morrey norm follow from the corresponding properties of
	\(L^{p(\cdot)}(U,\mu)\). If
	\(\|f\|_{\mathcal M^{r(\cdot)}_{p(\cdot)}(\mathcal B,\mu)}=0\), then
	\(1_Bf=0\) almost everywhere for every \(B\in\mathcal B\). The countable
	subfamily covers \(U\) up to a null set, so \(f=0\) almost everywhere on
	\(U\). Thus \(\mathcal M^{r(\cdot)}_{p(\cdot)}(\mathcal B,\mu)\) is a
	normed function lattice.

	\emph{Completeness.}
	Choose an enumeration \((B_j)_{j\ge1}\) of \(\mathcal B_0\) and let \((f_n)_{n\ge1}\) be a Cauchy sequence in
	\(\mathcal M^{r(\cdot)}_{p(\cdot)}(\mathcal B,\mu)\). For every \(j\),
	\[
		\|1_{B_j}(f_n-f_m)\|_{L^{p(\cdot)}(U,\mu)}
		\le
		\mu(B_j)^{\frac1{p_{B_j}^\mu}-\frac1{r_{B_j}^\mu}}
		\|f_n-f_m\|_{\mathcal M^{r(\cdot)}_{p(\cdot)}(\mathcal B,\mu)}.
	\]
	Hence \((1_{B_j}f_n)_n\) is Cauchy in \(L^{p(\cdot)}(U,\mu)\). Let
	\(F_j\) be its limit. After changing representatives on null sets, the
	functions \(F_j\) vanish outside \(B_j\) and agree on overlaps. Indeed, for
	fixed \(j,k\), both \(1_{B_j\cap B_k}F_j\) and
	\(1_{B_j\cap B_k}F_k\) are limits of \(1_{B_j\cap B_k}f_n\). This can be
	done for all \(j,k\) simultaneously, because only countably many null sets
	are involved.

	Define \(D_1:=B_1\) and \(D_j:=B_j\setminus\bigcup_{i<j}B_i\) for
	\(j\ge2\), and set
	\[
		f:=\sum_{j=1}^{\infty}1_{D_j}F_j
		\qquad\text{on } U.
	\]
	Then \(1_{B_k}f=F_k\) almost everywhere for every \(k\). If
	\(B\in\mathcal B\), then \((1_Bf_n)_n\) is Cauchy in
	\(L^{p(\cdot)}(U,\mu)\). The limit is \(1_Bf\). Indeed, on each
	\(B\cap B_j\), the restriction of this limit and \(1_{B\cap B_j}f\) are
	both limits of \(1_{B\cap B_j}f_n\), and \(\bigcup_j B_j\) covers \(U\) up
	to a null set. Thus
	\(1_Bf_n\to1_Bf\) in \(L^{p(\cdot)}(U,\mu)\) for every \(B\in\mathcal B\).

	We now prove convergence in the Morrey norm. Let \(\varepsilon>0\) and
	choose \(N\) such that
	\[
	\|f_n-f_m\|_{\mathcal M^{r(\cdot)}_{p(\cdot)}(\mathcal B,\mu)}
	\le \varepsilon,\qquad n,m\ge N.
	\]
	For each fixed \(B\in\mathcal B\), this gives
	\[
	\mu(B)^{\frac1{r_B^\mu}-\frac1{p_B^\mu}}
	\|1_B(f_n-f_m)\|_{L^{p(\cdot)}(U,\mu)}
	\le \varepsilon,\qquad n,m\ge N.
	\]
	Since \(1_Bf_m\to 1_Bf\) in
	\(L^{p(\cdot)}(U,\mu)\), we may pass to the limit \(m\to\infty\)
	and obtain
	\[
	\mu(B)^{\frac1{r_B^\mu}-\frac1{p_B^\mu}}
	\|1_B(f_n-f)\|_{L^{p(\cdot)}(U,\mu)}
	\le \varepsilon,\qquad n\ge N.
	\]
	Taking the supremum over \(B\in\mathcal B\) shows that \(f_n\to f\) in
	\(\mathcal M^{r(\cdot)}_{p(\cdot)}(\mathcal B,\mu)\). Hence
	\(f\in\mathcal M^{r(\cdot)}_{p(\cdot)}(\mathcal B,\mu)\), and the space is
	complete.
	
	\emph{Fatou property.}
	Let \(0\le f_n\uparrow f\) almost everywhere on \(U\), and suppose that
	\[
	\sup_n
	\|f_n\|_{\mathcal M^{r(\cdot)}_{p(\cdot)}(\mathcal B,\mu)}
	<\infty.
	\]
	For each \(B\in\mathcal B\), the Fatou property of
	\(L^{p(\cdot)}(U,\mu)\) gives
	\(
	\|1_B f\|_{L^{p(\cdot)}(U,\mu)}
	=\sup_{n\in\mathbb N}\|1_B f_n\|_{L^{p(\cdot)}(U,\mu)}.
	\)
	Taking the supremum over \(B\in\mathcal B\) gives
	\begin{align*}
		\|f\|_{\mathcal M^{r(\cdot)}_{p(\cdot)}(\mathcal B,\mu)}
		&=
		\sup_{B\in\mathcal B}
		\sup_{n\in\mathbb N}
		\mu(B)^{\frac1{r_B^\mu}-\frac1{p_B^\mu}}
		\|1_B f_n\|_{L^{p(\cdot)}(U,\mu)}
		=
		\sup_{n\in\mathbb N}
		\|f_n\|_{\mathcal M^{r(\cdot)}_{p(\cdot)}(\mathcal B,\mu)}
		<\infty.
	\end{align*}
	This proves the Fatou property.

	\emph{Saturation.}
	Assume that \(p(\cdot)\leq r(\cdot)\) almost everywhere and
	\([1]_{A_{p(\cdot)}(\mathcal B)}<\infty\). We first show that
	\(1_{B_0}\in \mathcal M^{r(\cdot)}_{p(\cdot)}(\mathcal B,\mu)\) for
	every \(B_0\in\mathcal B\). For \(B\in\mathcal B\), the pointwise
	inequality \(p(\cdot)\leq r(\cdot)\) gives
	\(
		\frac1{r_B^\mu}-\frac1{p_B^\mu}\leq0.
	\)
	If \(\mu(B)\leq1\), then the lattice property of \(L^{p(\cdot)}\) and
	Lemma \ref{lem:IndicatorKopaliani} give
	\[
	\begin{aligned}
		\mu(B)^{\frac1{r_B^\mu}-\frac1{p_B^\mu}}
		\|1_{B\cap B_0}\|_{L^{p(\cdot)}}
		&\leq
		\mu(B)^{\frac1{r_B^\mu}-\frac1{p_B^\mu}}
		\|1_B\|_{L^{p(\cdot)}}  \\
		&\leq
		5[1]_{A_{p(\cdot)}(\mathcal B)}\,
		\mu(B)^{\frac1{r_B^\mu}}
		\leq
		5[1]_{A_{p(\cdot)}(\mathcal B)}.
	\end{aligned}
	\]
	If \(\mu(B)>1\), then the sign of the exponent and the lattice property give
	\[
		\mu(B)^{\frac1{r_B^\mu}-\frac1{p_B^\mu}}
		\|1_{B\cap B_0}\|_{L^{p(\cdot)}}
		\leq
		\|1_{B_0}\|_{L^{p(\cdot)}}<\infty,
	\]
	because finite-measure indicators belong to \(L^{p(\cdot)}\). Indeed, if
	\(\lambda\geq \mu(B_0)+1\), then
	\[
		\rho_{p(\cdot)}(1_{B_0}/\lambda)
		\leq \frac{\mu(B_0)}{\lambda}+\frac1{\lambda}
		\leq 1.
	\]
	Thus \(\|1_{B_0}\|_{L^{p(\cdot)}}<\infty\). Taking the supremum over
	\(B\in\mathcal B\) proves \(1_{B_0}\in
	\mathcal M^{r(\cdot)}_{p(\cdot)}(\mathcal B,\mu)\).

	Let \(E\subset U\) be measurable with \(\mu(E)>0\). Since the enumeration
	\((B_j)_{j\geq1}\) of \(\mathcal B_0\) fixed above covers \(U\) up to a
	null set,
	\(\mu(E\cap B_j)>0\) for some \(j\). Setting \(F:=E\cap B_j\), the
	lattice property gives \(1_F\in
	\mathcal M^{r(\cdot)}_{p(\cdot)}(\mathcal B,\mu)\). Thus the space is
	saturated.
\end{proof}


\section{Bounds for the dyadic models}\label{sect:dyadic}
We assume in this section that the measure space $(X,\mu)$ is $\sigma$-finite. In particular $L^{p(\cdot)}(X,\mu)$ is a Banach function space, whenever $p(\cdot)\in\mathcal{P}(X,\mu)$, and all the results of Section \ref{sect:preliminaries} and \ref{sect:necessary} (below) can be used.
The goal is to prove norm bounds for the operators
\begin{align*}
	M^{\alpha(\cdot)}_{\mathcal{D}}f = \sup_{Q\in\mathcal{D}}\frac{1_Q}{ \mu(Q)^{1-\langle \alpha\rangle_Q}}\int_Q |f|\ud\mu,\qquad I^{\alpha(\cdot)}_{\mathcal{D}}f=\sum_{Q\in\mathcal{D}}\frac{1_Q}{ \mu(Q)^{1-\langle \alpha\rangle_Q}}\int_Q f\ud\mu,
\end{align*}
defined over a dyadic collection $\mathcal{D}$ of sets in an underlying measure space $(X,\mu)$, and where
\[
\langle \alpha\rangle_Q:=\frac{1}{\mu(Q)}\int_Q\alpha(x)\,\mathrm d\mu(x).
\] 
\subsubsection*{Dyadic structure}
We assume that the collection $\mathcal{D}$ is countable, that $0<\mu(Q)<\infty$ for each $Q\in\mathcal{D},$ and that $\mathcal{D}$ is dyadic, meaning that if $Q,Q' \in\mathcal{D}$ then $Q\cap Q'\in \{Q,Q',\emptyset\}.$ We call the elements of a dyadic collection of sets cubes. 
\subsubsection*{Sparsity with respect to a function}
An important concept will be major subsets and sparsity with respect to a function. More precisely, we will invoke sparsity with the choice of functions $A\mapsto \mu(A)$ and $A\mapsto \mu(A)^{1-\langle \alpha\rangle_{A}},$ resulting in standard sparsity ($\mu(\cdot)$-sparsity) and $\mu(\cdot)^{1-\langle\alpha\rangle_{(\cdot)}}$-sparsity, respectively. 
 This motivates the following definition. 
\begin{definition}\label{defn:MajorSparse}
	Let $\eta\in(0,1]$ and let $\psi:\mathcal A\to [0,\infty)$, where $\mathcal{A}$ is a collection of sets. 
    Whenever $\psi$ is evaluated below, the relevant set is assumed to belong to $\mathcal A$.
	\begin{itemize}
		\item 	A set $E(B)\subset B$ is an $(\eta,\psi)$-major subset of $B$, if $B\setminus E(B)$ is a $(1-\eta,\psi)$-minor subset of $B,$ meaning that:		\begin{align}\label{eq:defn:MajorSparse}
				\psi(B\setminus E(B)) \leq (1-\eta)\psi(B).
		\end{align}
		\item 	The collection $\mathcal{B}$ is said to be $(\eta,\psi)$-sparse, provided that for each $B\in\mathcal B$ there exists an $(\eta,\psi)$-major subset $E(B)\subset B$, and these sets are pairwise disjoint: if $B_1\not=B_2$, then $E(B_1)\cap E(B_2)=\emptyset.$
	\end{itemize}
\end{definition}
If $\psi = \mu$ is a measure and $\mathcal{A}$ is the collection of measurable sets, then \eqref{eq:defn:MajorSparse} is equivalent to $\mu(E(B)) \geq \eta\mu(B)$ and we recover the standard major-subset condition. 
When $\alpha(\cdot)=\alpha\in[0,1)$ is constant and $\psi_\alpha(A)=\mu(A)^{1-\alpha}$, the same condition is equivalent to standard sparsity after changing the parameter: $\psi_\alpha(B\setminus E(B))\leq (1-\eta)\psi_\alpha(B)$ is the same as $\mu(B\setminus E(B))\leq (1-\eta)^{\frac{1}{1-\alpha}}\mu(B)$. Thus in the constant-fractionality case, $\psi_\alpha$-sparsity is not a different sparse geometry; the point of the formulation is that it extends to the variable-fractionality setting.
If $\psi$ is sub-additive (see Definition \ref{defn:SubAdditive} below), which is for example the case with our key function $A\mapsto \mu(A)^{1-\langle \alpha\rangle_{A}},$ see Lemma \ref{lem:PsiSubAdd} below, then 
\begin{align}\label{eq:minor-to-major}
\psi(B) \leq \psi(E(B)) + \psi(B\setminus E(B)) \leq  \psi(E(B)) + (1-\eta)\psi(B) \Longrightarrow \psi(E(B)) \geq \eta\psi(B).
\end{align}

\subsubsection*{Sparse operators}
The main auxiliary operators we will be using are of the following form 
\begin{align*}
	\mathcal{A}^{\alpha(\cdot)}_{\mu,\mathcal{B}}(f) = \sum_{B\in\mathcal{B}}\frac{1_B}{ \mu(B)^{1-\langle\alpha\rangle_B}} \int_Bf\ud\mu
\end{align*}
and if $\mathcal{B}$ is $\mu(\cdot)^{1-\langle\alpha\rangle_{(\cdot)}}$-sparse, we call them $\mu(\cdot)^{1-\langle\alpha\rangle_{(\cdot)}}$-sparse operators. Note that when $\alpha(\cdot) = 0,$ we recover the usual sparse operators.

\subsubsection*{Structure of the sections to follow}
In the next Subsection \ref{sect:necessary} we record a general result, Proposition \ref{prop:Necessity}, that provides, perhaps after minor adjustments, all of our lower bounds/necessary conditions for boundedness that will be proved in the later Subsections \ref{sect:dyadic1}, \ref{sect:dyadic2} and \ref{sect:dyadic3}. 
In Subsection \ref{sect:dyadic1} we prove in Lemma \ref{lem:PsiSubAdd} that $A\mapsto \mu(A)^{1-\langle\alpha\rangle_A}$ is sub-additive and in the subsequent Proposition  \ref{prop:SDomAlpha}, by a principal cubes stopping time algorithm, a pointwise domination of $M^{\alpha(\cdot)}_{\mathcal{D}(Q_0)}$, for a given fixed $Q_0\in\mathcal{D},$ by $\mu(\cdot)^{1-\langle\alpha\rangle_{(\cdot)}}$-sparse operators $\mathcal{A}^{\alpha(\cdot)}_{\mathcal S}$. 
In Theorem \ref{thm:SparseLpLq} we prove the $L^{p(\cdot)}$-to-$L^{q(\cdot)}$ norm bounds for $\mu(\cdot)^{1-\langle\alpha\rangle_{(\cdot)}}$-sparse operators $\mathcal{A}^{\alpha(\cdot)}_{\mathcal S}$ by using Proposition \ref{prop:CarlesonEmbedding}, that is, the sparse Carleson embedding for variable exponent Lebesgue spaces.
In Subsection \ref{sect:dyadic2} we study the dyadic Riesz potential $I^{\alpha(\cdot)}_{\mu,\mathcal{D}}$ and in Theorem \ref{thm:RieszPot} we characterize its boundedness in terms of the smaller operator $M^{\alpha(\cdot)}_{\mu,\mathcal{D}}$ and a packing condition involving the exponent $\alpha(\cdot).$ 
In Subsection \ref{sect:dyadic3} we prove variable Morrey space bounds for $I^{\alpha(\cdot)}_{\mu,\mathcal{D}}$ and $M^{\alpha(\cdot)}_{\mu,\mathcal{D}}.$

\subsection{Necessary conditions}\label{sect:necessary}
Let $(X,\mu)$ be a measure space. We assume that all collections of sets and functions that appear below are measurable. Let $\mathcal B$ be a countable family of subsets of $X$ such that $0<\mu(B)<\infty$ for every $B\in\mathcal B$, and let $\psi:\mathcal B\to(0,\infty)$. We define the maximal operator and Riesz potential associated with $\psi$ and $\mathcal B$ by
\[
M^\psi_{\mathcal B}f(x):=\sup_{B\in\mathcal B}\frac{1_B(x)}{\psi(B)}\int_B |f|\,\mathrm d\mu,\qquad I^\psi_{\mathcal B}f(x):=\sum_{B\in\mathcal B}\frac{1_B(x)}{\psi(B)}\int_B f\,\mathrm d\mu.
\]
The sum is understood pointwise whenever the expression is well-defined. In the applications below, $\mathcal B$ will be a countable dyadic collection. 
The following proposition gives the lower bounds and necessary conditions used later.
\begin{proposition}\label{prop:Necessity}
	Let $Z$ be a Banach function lattice over a measure space $(X,\mu)$. Let $Z'$ be the K\"othe dual of $Z$. Let $Y$ be a normed function lattice over $(X,\mu)$.
	Let $\mathcal B$ be a countable family of measurable sets such that $0<\mu(B)<\infty,$ for $B\in\mathcal{B},$ and let
	$\psi:\mathcal B\to (0,\infty)$. Then, the following bounds hold
	\begin{align}
		\sup_{B\in\mathcal B}
		\frac{\|1_B\|_Y\|1_B\|_{Z'}}{\psi(B)} &\le \|M^\psi_{\mathcal{B}}\|_{Z\to Y}, \label{eq:prop:NecessityMax} \\
		\sup_{B_0\in\mathcal B}
		\|1_{B_0}\|_Y \, \biggl\|\sum_{\substack{B\in\mathcal B\\ B\supset B_0}}
		\frac{1_B}{\psi(B)}\biggr\|_{Z'}
		&\le \|I^{\psi}_{\mathcal{B}}\|_{Z\to Y}, \label{eq:prop:NecessityPotential1} \\
		\sup_{B_0\in\mathcal B}
		\frac{1}{\|1_{B_0}\|_Z}	\biggl\| \sum_{\substack{B\in\mathcal B\\ B\subset B_0}}
		\frac{\mu(B) }{\psi(B)}1_{B} \biggr\|_{Y}
		&\le \|I^{\psi}_{\mathcal{B}}\|_{Z\to Y}. \label{eq:prop:NecessityPotential}
	\end{align}
\end{proposition}

\begin{proof}
	We start by proving \eqref{eq:prop:NecessityMax}. Fix $B\in\mathcal B$. Then
	\[
	\frac{1_B(x)}{\psi(B)}
	\int_B |f|\,\mathrm d\mu
	\le
	M^\psi_{\mathcal{B}} f(x).
	\]
	Taking the norm in $Y$ and using the lattice property of $Y$ gives 
	\[
	\frac{\|1_B\|_Y}{\psi(B)}
	\int_B |f|\,\mathrm d\mu
	\le
	\|M^\psi_{\mathcal{B}}f\|_Y\leq \|M^\psi_{\mathcal{B}}\|_{Z\to Y}\|f\|_Z .
	\]
	By the definition
	of the K\"othe dual norm, we have
	\[
	\|1_B\|_{Z'}
	=
	\sup_{\|f\|_Z\le 1}
	\int_B |f|\,\mathrm d\mu .
	\]
	Inserting this into the previous estimate gives
	\[
	\frac{\|1_B\|_Y\|1_B\|_{Z'}}{\psi(B)}
	\le \|M^\psi_{\mathcal{B}}\|_{Z\to Y},
	\]
	which proves \eqref{eq:prop:NecessityMax}.
	
	We move to verifying \eqref{eq:prop:NecessityPotential1}. 
	Fix $B_0\in\mathcal B$ and let $f\geq 0$ be non-negative. Restricting the sum to the sets containing $B_0$ gives 
	\[
	1_{B_0}I^\psi_{\mathcal{B}} f
	=
	\sum_{B\in\mathcal B}
	\frac{1_B1_{B_0}}{\psi(B)}
	\int_B f\,\mathrm d\mu
	\ge
	1_{B_0}\sum_{\substack{B\in\mathcal B\\ B_0\subset B}}
	\frac{1}{\psi(B)}
	\int_B f\,\mathrm d\mu 	=
	1_{B_0}
	\int_X f
	\sum_{\substack{B\in\mathcal B\\ B_0\subset B}}
	\frac{1_B}{\psi(B)}
	\,\mathrm d\mu,
	\]
	where we used Tonelli's theorem in the last equality.
	Therefore,
	\[
	1_{B_0}\int_X f
	\sum_{\substack{B\in\mathcal B\\ B_0\subset B}}
	\frac{1_B}{\psi(B)}
	\,\mathrm d\mu
	\le
	1_{B_0}I^\psi_{\mathcal{B}} f \leq I^\psi_{\mathcal{B}} f.
	\]
	By the lattice property of $Y$, we obtain
	\[
	\|1_{B_0}\|_Y
	\int_X f\sum_{\substack{B\in\mathcal B\\ B_0\subset B}}
	\frac{1_B}{\psi(B)}
	\,\mathrm d\mu
	\le
	\|I^\psi_{\mathcal{B}}f\|_Y \le
	\|I^{\psi}_{\mathcal{B}}\|_{Z\to Y}\| f \|_Z .
	\]
	Since the function inside the $Z'$-norm below is non-negative and $Z$ is a lattice, its K\"othe dual norm can be computed by testing against non-negative functions only.
	We have 
	\begin{align*}
		\|1_{B_0}\|_Y \,\Bigl \|\sum_{\substack{B\in\mathcal B\\ B_0\subset B}}
		\frac{1_B}{\psi(B)}\Bigr\|_{Z'} &= 		\|1_{B_0}\|_Y\sup_{\substack{\|f\|_Z\le 1\\ f\geq 0}}\int_X f\sum_{\substack{B\in\mathcal B\\ B_0\subset B}}
		\frac{1_B}{\psi(B)}
		\,\mathrm d\mu \\
		&\le   \sup_{\|f\|_Z\le 1 }\|I^{\psi}_{\mathcal{B}}\|_{Z\to Y}\| f \|_Z \leq \|I^{\psi}_{\mathcal{B}}\|_{Z\to Y}.
	\end{align*}

	We then turn our attention to \eqref{eq:prop:NecessityPotential}. 
	Restricting the sum to the sets contained in \(B_0\) gives 
	\[
	\sum_{\substack{B\in\mathcal B\\ B\subset B_0}}
	\frac{\mu(B) }{\psi(B)}1_{B} \leq 
	\sum_{\substack{B\in\mathcal B \\ B\cap B_0\not=\emptyset}}
	\frac{1_B}{\psi(B)}
	\int_B 1_{B_0}\,\mathrm d\mu
	=	I^\psi_{\mathcal{B}} 1_{B_0}
	\]
	Therefore the lattice property of $Y$ gives
	\[
	\biggl\| \sum_{\substack{B\in\mathcal B\\ B\subset B_0}}
	\frac{\mu(B) }{\psi(B)}1_{B} \biggr\|_{Y} \leq 	\bigl\|  	I^\psi_{\mathcal{B}} 1_{B_0} \bigr\|_{Y} \leq \|I^{\psi}_{\mathcal{B}}\|_{Z\to Y}	\left\|  1_{B_0} \right\|_{Z}.\qedhere
	\]
\end{proof}

\subsubsection{On the pointwise relation of $\alpha(\cdot)$, $1/p(\cdot)$ and $1/q(\cdot)$.}
We use this subsection to briefly discuss whether any relation on $\alpha(\cdot)$, $1/p(\cdot)$ and $1/q(\cdot)$ is forced, since
we are not assuming that $\alpha(\cdot) = 1/p(\cdot)-1/q(\cdot)$ (a well-known necessary condition in the case when $\alpha,p,q$ are all constants). Only one side of the estimate holds, provided $\mathcal{B}$ is a differentiation basis with the shrinking property.
\begin{corollary}\label{cor:AlphaRelation}
	Let $(X,\mu)$ be a non-atomic measure space and $p(\cdot),q(\cdot),1/\alpha(\cdot)\in\mathcal P(X,\mu)$. Suppose $\mathcal{B}$ is a differentiation basis that differentiates $\alpha(\cdot)$, $1/p(\cdot)$ and $1/q(\cdot)$ and admits shrinking sets at almost every differentiation point. More precisely, for almost every such point $x$, there exist $B_k\in\mathcal B$ with $x\in B_k$, $\mu(B_k)\to0$, and
	\[
	\lim_{k\to\infty}\langle \alpha\rangle_{B_k}=\alpha(x),\qquad
	\lim_{k\to\infty}1/p^\mu_{B_k}=1/p(x),\qquad
	\lim_{k\to\infty}1/q^\mu_{B_k}=1/q(x).
	\]	If $M^{\alpha(\cdot)}_{\mathcal{B}}:L^{p(\cdot)}(X,\mu)\to L^{q(\cdot)}(X,\mu)$ is bounded, then 
	$\alpha(\cdot) \geq 1/p(\cdot)-1/q(\cdot)$ holds $\mu$ almost everywhere.   
\end{corollary}
\begin{remark}The bound $|M^{\alpha(\cdot)}_{\mathcal{B}} f|\leq I^{\alpha(\cdot)}_{\mathcal{B}}|f|$ shows that if $ I^{\alpha(\cdot)}_{\mathcal{B}}:L^{p(\cdot)}(X,\mu)\to L^{q(\cdot)}(X,\mu)$ is bounded, then $ M^{\alpha(\cdot)}_{\mathcal{B}}:L^{p(\cdot)}(X,\mu)\to L^{q(\cdot)}(X,\mu)$ is bounded and the conclusion of Corollary \ref{cor:AlphaRelation} holds also.
\end{remark}
\begin{proof}[Proof of Corollary \ref{cor:AlphaRelation}]
	Lemma \ref{lem:dualityFormula} and \eqref{eq:prop:NecessityMax} of Proposition \ref{prop:Necessity} imply that 
	\begin{multline}\label{eq:cor:AlphaRelation1}
		\mu(B)^{\langle \alpha\rangle_B-\left(1/p^{\mu}_B-1/q^{\mu}_B\right)}
		=\frac{\mu(B)^{1/q^{\mu}_B}\mu(B)^{1/(p')_B^\mu}}{\mu(B)^{1-\langle \alpha\rangle_B}} \\ \leq 16	\frac{\|1_B\|_{L^{q(\cdot)}}\|1_B\|_{L^{p'(\cdot)}}}
		{\mu(B)^{1-\langle \alpha\rangle_B}}
		\lesssim
		\bigl\|M^{\alpha(\cdot)}_{\mathcal B}\bigr\|_{L^{p(\cdot)}\to L^{q(\cdot)}} < \infty. 
	\end{multline}
	Let $x\in X$ be a common differentiation point at which $\mathcal B$ admits shrinking sets. Choose $B_k\in\mathcal B$ with $x\in B_k$, $\mu(B_k)\to0$, and
	\[
	\lim_{k\to\infty}\langle \alpha\rangle_{B_k}-\left(1/p^\mu_{B_k}-1/q^\mu_{B_k}\right)
	=
	\alpha(x)-\left(1/p(x)-1/q(x)\right).
	\]
	Suppose that $\alpha(x)-\left(1/p(x)-1/q(x)\right)<0$. Then there exists $\varepsilon>0$ such that, for all sufficiently large $k$,
	\(
	\langle \alpha\rangle_{B_k}-\left(1/p^\mu_{B_k}-1/q^\mu_{B_k}\right)\leq -\varepsilon.
	\)
	Using that $\mu(B_k)\to0$, we obtain
	\[
	\lim_{k\to\infty} \mu(B_k)^{\langle \alpha\rangle_{B_k}-\left(1/p^{\mu}_{B_k}-1/q^{\mu}_{B_k}\right)} \geq \lim_{k\to\infty} \mu(B_k)^{-\varepsilon} = \infty,
	\]
    which contradicts \eqref{eq:cor:AlphaRelation1}. Thus $\alpha(x)\geq 1/p(x)-1/q(x)$ at every common differentiation point at which the shrinking condition holds. These points have full measure, and the conclusion follows.
\end{proof}

\subsection{Dyadic maximal operator with variable fractionality}\label{sect:dyadic1}

The next goal is to prove Theorem \ref{thm:MainIntro2}.
The lower bound follows from Proposition \ref{prop:Necessity} by testing the operator with indicators. It remains to prove the upper bound. 
The first step is a sparse domination for $M^{\alpha(\cdot)}_{\mu,\mathcal{D}}$, with sparsity measured by $Q\mapsto\mu(Q)^{1-\langle\alpha\rangle_Q}$. 
\begin{proposition}\label{prop:SDomAlpha} Let $\mathcal{D}$ be a finite collection of dyadic cubes and $1/\alpha(\cdot)\in\mathcal{P}(X,\mu)$ and $\eta\in (0,1).$ Let $f\in L^1_{\mathrm{loc}}$ and $Q_0\in\mathcal{D}$ be a fixed cube. Then, there exists a family $\mathcal{S}_f(Q_0) \subset\mathcal{D}(Q_0)$ that is $(\eta,\psi_\alpha)$-sparse for $\psi_\alpha(Q):=\mu(Q)^{1-\langle\alpha\rangle_Q}$, so that 
	\begin{equation*}
		M^{\alpha(\cdot)}_{\mu,\mathcal{D}(Q_0)}f\leq \frac{1}{1-\eta}\sum_{S\in\mathcal{S}_f(Q_0)}\frac{1_{E(S)}}{\mu(S)^{1-\langle\alpha\rangle_S}} \int_S|f|\ud\mu \leq \frac{1}{1-\eta}\mathcal{A}^{\alpha(\cdot)}_{\mu,\mathcal{S}_f(Q_0)}|f|
	\end{equation*}
	holds $\mu$ almost everywhere.
\end{proposition}
Proposition \ref{prop:SDomAlpha} is a special case of Proposition \ref{prop:SDomSub} below formulated in terms of a general sub-additive function $\psi.$
The upper bound in Theorem \ref{thm:MainIntro2} follows by combining Proposition \ref{prop:SDomAlpha} with the following norm bounds for $\left(\eta, \mu(\cdot)^{1-\langle\alpha\rangle_{(\cdot)}}\right)$-sparse operators with a specific choice of $\eta.$
\begin{theorem}\label{thm:SparseLpLq} 
	Let  $1<p(\cdot), q(\cdot),1/\alpha(\cdot)\in \mathcal{P}(X,\mu)$ be such that $p(\cdot)\leq q(\cdot)<\infty$ almost everywhere. Let 	
	\[
	\eta([1]_{A_{1/\alpha(\cdot)}}) := 1- \frac{1}{50[1]_{A_{1/\alpha(\cdot)}(\mathcal{S})}}
	\] 
and $\mathcal{S}$ be an $\left( \eta([1]_{A_{1/\alpha(\cdot)}}), \mu(\cdot)^{1-\langle\alpha\rangle_{(\cdot)}}\right)$-sparse collection.
	Then, there holds that 
	\begin{equation}\label{eq:thm:Unweighted:2}
		\begin{split}
			[1]_{A_{p(\cdot),q(\cdot)}^{\alpha(\cdot)}(\mathcal{S})} \lesssim  \| \mathcal{A}^{\alpha(\cdot)}_{\mathcal S} \|_{L^{p(\cdot)}\to L^{q(\cdot)}} \lesssim [1]_{A_{p(\cdot),q(\cdot)}^{\alpha(\cdot)}(\mathcal{S})} \|M_{\mathcal S}\|_{ L^{q'(\cdot)}} \|M_{\mathcal S}\|_{ L^{p(\cdot)}}.
		\end{split}
	\end{equation}
\end{theorem}

Next we give a detailed proof of the three results stated above and we begin with the pointwise sparse domination involving a sub-additive function.
\begin{definition}\label{defn:SubAdditive} Let $(X,\mu)$ be a measure space and $\psi$  a function evaluated on the measurable sets of $X.$ We say that $\psi$ is sub-additive provided that 
	\[
	A\cap B=\emptyset \Longrightarrow \psi(A\cup B) \leq \psi(A) + \psi(B).
	\]
\end{definition}

The sub-additivity of our key function $\mu(\cdot)^{1-\langle\alpha\rangle_{(\cdot)}}$ is a consequence of the concavity of the root function; notice that $1-\alpha(\cdot)\in [0,1]$ is a standing assumption, since $1/\alpha(\cdot)\in\mathcal{P}(X,\mu).$ We use the logarithm to write the proof.
\begin{lemma}\label{lem:PsiSubAdd}
Let $s(\cdot)\in \mathcal{P}(X,\mu)$. Define
\(
\psi^{1/s(\cdot)}_{\mu}(A):=\mu(A)^{1/s_A^\mu}
\)
for measurable sets $A\subset X$ with $0<\mu(A)<\infty$. Then $\psi^{1/s(\cdot)}_{\mu}$ is sub-additive.
\end{lemma}
\begin{proof} Writing
	\[
	\frac{1}{s^{\mu}_{A\cup B}} = \frac{\mu(A)}{\mu(A\cup B)}\frac{1}{s^{\mu}_{A}} +  \frac{\mu(B)}{\mu(A\cup B)}\frac{1}{s^{\mu}_{B}} 
	\]
	and using properties of the logarithm gives 
	\begin{align*}
		\ln\left( 	\psi^{1/s(\cdot)}(A\cup B)\right) &= \ln\left(  \mu(A\cup B)^{1/s^{\mu}_{A\cup B}}\right) \\
		&= \left(\frac{\mu(A)}{\mu(A\cup B)}\frac{1}{s^{\mu}_{A}} +  \frac{\mu(B)}{\mu(A\cup B)}\frac{1}{s^{\mu}_{B}} \right) \ln\left(  \mu(A\cup B)\right) \\
		&= \frac{\mu(A)}{\mu(A\cup B)}\ln\left(  \mu(A\cup B)^{1/s^{\mu}_A}\right) + \frac{\mu(B)}{\mu(A\cup B)}\ln\left(  \mu(A\cup B)^{1/s^{\mu}_B}\right).
	\end{align*}
	As $A,B$ are disjoint we have 
	$
	\frac{\mu(A)}{\mu(A\cup B)} +  \frac{\mu(B)}{\mu(A\cup B)} = 1
	$ and the concavity of logarithm gives 
	\begin{multline}\label{eq:propPsiSubAdd}
		\ln\left( 	\psi^{1/s(\cdot)}(A\cup B)\right) \\ \leq\ln\left( \frac{\mu(A)}{\mu(A\cup B)}  \mu(A\cup B)^{1/s^{\mu}_A}+ \frac{\mu(B)}{\mu(A\cup B)}  \mu(A\cup B)^{1/s^{\mu}_B}\right) 
		\leq \ln\left( \psi^{1/s(\cdot)}(A) + \psi^{1/s(\cdot)}(B)\right),
	\end{multline}
	where we used H\"older's inequality in the last estimate. Indeed, by the assumption $s(\cdot)\geq 1$ we have $s_{V}^{\mu}\geq 1$ for both $V=A,B$ and hence 
	\begin{multline*}
		\frac{\mu(V)}{\mu(A\cup B)}  \mu(A\cup B)^{1/s^{\mu}_V} = \left(\frac{1}{\mu(A\cup B)}\int_{A\cup B}1_V\ud\mu \right) \mu(A\cup B)^{1/s^{\mu}_V} \\
		\leq \left(\frac{1}{\mu(A\cup B)}\int_{A\cup B}1_V\ud\mu \right)^{1/s^{\mu}_V} \mu(A\cup B)^{1/s^{\mu}_V} = \mu(V)^{1/s^{\mu}_V} = 	\psi^{1/s(\cdot)}(V).
	\end{multline*}
	Taking the exponential of \eqref{eq:propPsiSubAdd} concludes the proof. 
\end{proof}

Absolute continuity $\mu\ll \psi^{1/\sigma(\cdot)}_{\mu}$ makes the stopping time algorithm pick up full measure.

\begin{definition} Let $\psi\geq 0$  be a function evaluated on the measurable sets of $(X,\mu).$ 
	We write $\mu\ll\psi$ if for all decreasing sequences $A_1\supset A_2\supset \cdots$ of measurable sets there holds that 
	\[
	\lim_{k\to\infty}\psi(A_k) = 0 \Longrightarrow \lim_{k\to\infty}\mu(A_k) = 0.
	\]
\end{definition}

\begin{proposition}\label{prop:SDomSub}
	Let  $\mathcal{D}$ be a \emph{finite} collection of dyadic cubes on a measure space $(X,\mu).$ Let $\psi$ be a sub-additive function such that $\mu\ll \psi$ and $0<\psi(Q)<\infty$ for all $Q\in\mathcal{D}.$ Let $\eta\in (0,1)$ be fixed.
	Let $Q_0\in\mathcal{D}$ be a fixed top cube and $f\in L^1_{\mathrm{loc}}(X, \mu).$  Then, there exists a $(\eta,\psi)$-sparse family $\mathcal{S}_f(Q_0) \subset\mathcal{D}(Q_0)$ such that 
	\begin{equation*}
		M^{\psi}_{\mu,\mathcal{D}(Q_0)}f\leq \frac{1}{1-\eta}\sum_{S\in\mathcal{S}_f(Q_0)}\langle |f|\rangle_{\mu,S}^{\psi}1_{E(S)}
	\end{equation*}
	holds $\mu$ almost everywhere.
\end{proposition}

\begin{proof}
	We apply the "principal cubes algorithm" to construct a $\psi$-sparse family of cubes. Let 
	\begin{align*}
		\mathcal{S}&:= \bigcup_{k=0}^\infty\mathcal{S}_{k},&\text{where}&&&\mathcal{S}_0:=\{Q_0\}\\
		\mathcal{S}_{k+1}&:=\bigcup_{S\in \mathcal{S}_k}\text{ch}(S),&\text{where}&&&\text{ch}(S):=\Bigl\{Q\subsetneq S\text{ maximal }: \langle |f|\rangle^{\psi}_{\mu,Q}>\frac{1}{1-\eta}\langle| f|\rangle^{\psi}_{\mu,S}\Bigr\}.
	\end{align*}
	If the stopping time terminates in some cube $S$ we set $E(S) := S$ and the terminal cubes are their own $\psi$-major subsets. If $S\in\mathcal{S}_k$ is not a terminal cube, i.e. there exists some $S'\in\mathcal{S}_{k+1}$ such that $S'\subsetneq S,$ then we set $E(S):=S\setminus\bigcup_{Q\in\text{ch}(S)}Q$. By construction the sets $E(S)$ are pairwise disjoint. 
    
    We next verify that $S\setminus E(S) := \bigcup\mathrm{ch}(S)$ is $\psi$-minor (recall from Definition \ref{defn:MajorSparse} that we have defined majority with the minority of the complement) by using the finite sub-additivity of $\psi.$ If $S$ is a terminal cube, there is nothing to prove. If $S$ is not a terminal cube we use the finite sub-additivity of $\psi$ ($\mathrm{ch}(S)$ is a finite collection, since $\mathcal{D}$ is a finite collection) to bound
	\begin{equation}\label{eq:sdomproof1}
		\begin{split}
		    \psi\big( \bigcup\mathrm{ch}(S)\big)\leq \sum_{Q\in \mathrm{ch}(S)}\psi(Q) &\leq  (1-\eta)\frac{\sum_{Q\in \mathrm{ch}(S)}\int_Q|f|\mathrm d\mu}{\langle |f|\rangle^{\psi}_{\mu,S}} \\ &\leq (1-\eta)\frac{\int_S|f|\mathrm d \mu}{\langle |f|\rangle^{\psi}_{\mu,S}}\leq (1-\eta)\psi(S).
		\end{split}
	\end{equation}
This demonstrates that $S\setminus E(S) := \bigcup_{Q\in\mathrm{ch}(S)}Q$ is $\psi$-minor and hence that $\mathcal{S}$ is a $\psi$-sparse family.
	
	Using the bound \eqref{eq:sdomproof1} repeatedly shows that 
	\begin{align*}
    \psi\bigl(\bigcup_{S\in\mathcal{S}_k}S\bigr)
    &\leq \sum_{S\in\mathcal{S}_{k}}\psi(S)
    = \sum_{\widehat{S}\in \mathcal{S}_{k-1}} \sum_{S\in\text{ch}(\widehat{S})}\psi(S) \\
    &\leq  (1-\eta) \sum_{\widehat{S}\in \mathcal{S}_{k-1}} \psi(\widehat{S})
    \leq \cdots \leq (1-\eta)^k\psi(Q_0).
	\end{align*}
	The sets $\cup_{S\in\mathcal{S}_{k}}S$ decrease with $k$. Since $\eta < 1$ and $\mu\ll\psi$, the previous estimate gives
	\[
	\mu\bigl(Q_0\setminus \cup_{S\in\mathcal{S}}E(S)\bigr) = \mu\bigl(\cap_{k}\cup_{S\in\mathcal{S}_k}S\bigr) = 0.
	\]
	Hence, $\mu$-almost everywhere,
	\begin{align}\label{eq:sdomproof2}
		1_{Q_0} = \sum_{S\in\mathcal{S}}1_{E(S)}.
	\end{align}	
	Fix $S\in\mathcal{S}.$ We first show that every cube $Q\in\mathcal{D}(Q_0)$ with $Q\cap E(S)\not=\emptyset$ satisfies
	\[
	\langle |f|\rangle_{\mu,Q}^{\psi} \leq \frac{1}{1-\eta} \langle |f|\rangle_{\mu,S}^{\psi}.
	\]
	If $Q\subset S$, this follows from the maximality of the stopping children of $S$, since $Q$ is either $S$ itself or is not contained in any stopping child. If $Q\not\subset S$, then dyadicity gives $S\subset Q$. Let $Q_0=S_0\supset S_1\supset\cdots\supset S_m=S$ be the chain of stopping ancestors of $S$. If $Q\neq Q_0$, choose $i$ with $S_i\subset Q\subset S_{i-1}$. The maximality of the stopping child $S_i$ of $S_{i-1}$ gives
	\[
	\langle |f|\rangle_{\mu,Q}^{\psi}\leq \frac{1}{1-\eta}\langle |f|\rangle_{\mu,S_{i-1}}^{\psi}\leq \frac{1}{1-\eta}\langle |f|\rangle_{\mu,S}^{\psi}.
	\]
	The case $Q=Q_0$ is even simpler, because $\langle |f|\rangle_{\mu,Q_0}^{\psi}\leq \langle |f|\rangle_{\mu,S}^{\psi}$. Therefore
\begin{equation}\label{eq:sdomproof3}
		1_{E(S)}M^{\psi}_{\mu,\mathcal{D}(Q_0)}f \leq \frac{1}{1-\eta}\langle |f|\rangle_{\mu,S}^{\psi}1_{E(S)}.
	\end{equation}
	Together \eqref{eq:sdomproof2} and \eqref{eq:sdomproof3} give
	\begin{equation*}
		M^{\psi}_{\mu,\mathcal{D}(Q_0)}f=1_{Q_0}M^{\psi}_{\mu,\mathcal{D}(Q_0)}f= \sum_{S\in\mathcal{S}}1_{E(S)}M^{\psi}_{\mu,\mathcal{D}(Q_0)}f\leq \frac{1}{1-\eta}\sum_{S\in\mathcal{S}}\langle |f|\rangle_{\mu,S}^{\psi}1_{E(S)}.\qedhere
	\end{equation*}
\end{proof}

\begin{proof}[Proof of Theorem \ref{thm:MainIntro2}]  
	The lower bound is immediate by Proposition \ref{prop:Necessity}.
	For the other direction we argue on finite subcollections and then pass to the limit. Since $\mathcal D$ is countable, write $\mathcal D=\{Q_j\}_{j=1}^\infty$ and set $\mathcal D_N:=\{Q_j:1\leq j\leq N\}.$ Then 
	\(
	M^{\alpha(\cdot)}_{\mu,\mathcal D_N}f\uparrow M^{\alpha(\cdot)}_{\mu,\mathcal D}f
	\)
	pointwise as $N\to\infty$. We prove a bound independent of $N.$
	Define $\psi(A) := \mu(A)^{1-\langle\alpha\rangle_A}$ for measurable sets $A$ of finite measure
	so that $M^{\alpha(\cdot)}_{\mu,\mathcal{D}_N} = M^{\psi}_{\mu,\mathcal{D}_N}.$
	Since $1/\alpha(\cdot)\in\mathcal P(X,\mu)$, we have $0\leq \alpha(\cdot)\leq 1$ almost everywhere, and hence $(1-\alpha(\cdot))^{-1}\in\mathcal P(X,\mu).$ Lemma \ref{lem:PsiSubAdd} gives that $\psi$ is sub-additive. Moreover, $\mu\ll \psi$ on every finite top cube. Indeed, if $A\subset Q_0$, $\mu(Q_0)<\infty$ and $\mu(A)>0$, then $1-\langle\alpha\rangle_A\in [0,1]$ and
	\[
	\psi(A)=\mu(A)^{1-\langle\alpha\rangle_A}\geq \min\{\mu(A),1\}.
	\]
	Thus, for any decreasing sequence $A_k\subset Q_0$, $\psi(A_k)\to0$ implies $\mu(A_k)\to0$. Now let 

	\[
	\eta := \eta([1]_{A_{1/\alpha(\cdot)}(\mathcal{D})}) := 1-\frac{1}{50[1]_{A_{1/\alpha(\cdot)}(\mathcal{D})}}.
	\]
	Let $\mathcal T_N$ be the collection of maximal cubes in $\mathcal D_N$. Since $\mathcal D$ is dyadic, the cubes in $\mathcal T_N$ are pairwise disjoint and every cube in $\mathcal D_N$ is contained in exactly one cube of $\mathcal T_N$. For each fixed $Q_0\in \mathcal T_N$, Proposition \ref{prop:SDomSub} applied to $\mathcal D_N(Q_0):=\{Q\in\mathcal D_N:Q\subset Q_0\}$ yields an $(\eta,\psi)$-sparse collection $\mathcal{S}_f(Q_0)\subset\mathcal D_N(Q_0)$. The disjoint union
	\(
	\mathcal{S}_N :=  \bigcup_{Q_0\in\mathcal T_N}\mathcal{S}_f(Q_0) \subset \mathcal D_N
	\)
	is automatically $(\eta,\psi)$-sparse. Thus, 
	\begin{align*}
		M^{\alpha(\cdot)}_{\mu,\mathcal D_N}f
		&= \sum_{Q_0\in\mathcal T_N}M^{\psi}_{\mu,\mathcal D_N(Q_0)}f \\
		&\leq \frac{1}{1-\eta}\sum_{Q_0\in\mathcal T_N}\sum_{S\in\mathcal{S}_f(Q_0)}\langle |f|\rangle_{\mu,S}^{\psi}1_{E(S)}
		\lesssim  [1]_{A_{\frac{1}{\alpha(\cdot)}}(\mathcal{D})}\sum_{S\in\mathcal{S}_N}\langle |f|\rangle_{\mu,S}^{\alpha(\cdot)}1_{E(S)}.
	\end{align*}
	Hence 
	\[
	\|M^{\alpha(\cdot)}_{\mu,\mathcal D_N}f\|_{L^{q(\cdot)}}
	\lesssim  [1]_{A_{\frac{1}{\alpha(\cdot)}}(\mathcal{D})}
	\Bigl\|\sum_{S\in\mathcal{S}_N}\langle |f|\rangle_{\mu,S}^{\alpha(\cdot)}1_{E(S)}\Bigr\|_{L^{q(\cdot)}}
	\leq [1]_{A_{\frac{1}{\alpha(\cdot)}}(\mathcal{D})} \big\| \mathcal{A}^{\alpha(\cdot)}_{\mathcal S_N}\big\|_{L^{p(\cdot)}\to L^{q(\cdot)}}	\|f\|_{L^{p(\cdot)}}. 
	\]
	Since $\mathcal{S}_N$ is sparse with the parameter determined by $[1]_{A_{1/\alpha(\cdot)}(\mathcal D)}$, and since $[1]_{A_{1/\alpha(\cdot)}(\mathcal S_N)}\leq [1]_{A_{1/\alpha(\cdot)}(\mathcal D)}$, the monotonicity of $\eta(t)=1-\frac1{50t}$ shows that $\mathcal{S}_N$ is also sparse with the parameter required in Theorem \ref{thm:SparseLpLq}. Therefore, Theorem \ref{thm:SparseLpLq} yields 
\begin{align*}
\big\| \mathcal{A}^{\alpha(\cdot)}_{\mathcal S_N}\big\|_{L^{p(\cdot)}\to L^{q(\cdot)}}
&\lesssim [1]_{A_{p(\cdot),q(\cdot)}^{\alpha(\cdot)}(\mathcal{S}_N)}
\|M_{\mathcal S_N}\|_{ L^{q'(\cdot)}} \|M_{\mathcal S_N}\|_{ L^{p(\cdot)}} \\
&\leq [1]_{A_{p(\cdot),q(\cdot)}^{\alpha(\cdot)}(\mathcal{D})}
\|M_{\mathcal D}\|_{ L^{q'(\cdot)}} \|M_{\mathcal D}\|_{ L^{p(\cdot)}}.
\end{align*}
Combining the previous estimates gives a bound for $\|M^{\alpha(\cdot)}_{\mu,\mathcal D_N}f\|_{L^{q(\cdot)}}$ that is independent of $N$. By Proposition \ref{prop:Lp=Fatou}, the space $L^{q(\cdot)}(X,\mu)$ has the Fatou property, and hence
\[
\|M^{\alpha(\cdot)}_{\mu,\mathcal D}f\|_{L^{q(\cdot)}}
=\sup_N\|M^{\alpha(\cdot)}_{\mu,\mathcal D_N}f\|_{L^{q(\cdot)}}.\qedhere
\]
\end{proof}

\begin{proof}[Proof of Theorem \ref{thm:SparseLpLq}]  
	The lower bound is immediate by Proposition \ref{prop:Necessity}; indeed, since the sparse operator is bounded, the single scale averaging operators $M^{\alpha(\cdot)}_{\mu,S},$ for each fixed $S\in\mathcal{S},$ are uniformly bounded and hence
	\[
	[1]_{A_{p(\cdot),q(\cdot)}^{\alpha(\cdot)}(S)} \lesssim \| M^{\alpha(\cdot)}_{S} \|_{L^{p(\cdot)}\to L^{q(\cdot)}} \leq  \| \mathcal{A}^{\alpha(\cdot)}_{\mathcal{S}} \|_{L^{p(\cdot)}\to L^{q(\cdot)}}.
	\]

    We move to the upper bound. By positivity and the lattice property, it is enough to consider $f\geq0$. By Lemma \ref{lem:dualityFormula}, it suffices to estimate the pairing against arbitrary non-negative functions $g$ with $\|g\|_{L^{q'(\cdot)}}\leq1$. For such $g$,
	\begin{align*}
			\int \mathcal{A}^{\alpha(\cdot)}_{\mathcal{S}}f\,g\ud\mu
		&=
		\sum_{S\in\mathcal S}\frac{\int_Sf\,\mathrm d \mu}{\mu(S)^{1-\langle \alpha\rangle_S}}\int_Sg\,\mathrm d\mu
        =
		\sum_{S\in\mathcal S}\langle f\rangle_S \mu(S)^{1+\langle \alpha\rangle_S} \langle g\rangle _S
		\\
		&=\sum_{S\in\mathcal S}\frac{\mu(S)^{1+\langle \alpha\rangle_S}}{\mu\bigl(E(S)\bigr)^{1+\langle \beta\rangle_{E(S)}}}
		\langle f\rangle_S \mu\bigl(E(S)\bigr)^{\frac1{p^\mu_{E(S)}}} \langle g\rangle_S\mu\bigl(E(S)\bigr)^{\frac1{(q')^\mu_{E(S)}}},
	\end{align*}
	where $\beta(\cdot)$ is defined through the relation $1+\beta(\cdot) :=\tfrac1{p(\cdot)}+\frac1{q'(\cdot)}.$
	Taking the supremum over all such $g$, Lemma \ref{lem:Holder} gives
	\begin{equation}\label{eq:adhoc777}
		\begin{aligned}
				\bigl\| \mathcal{A}^{\alpha(\cdot)}_{\mathcal S}f \bigr\|_{L^{q(\cdot)}}&\lesssim \sup_{S\in\mathcal S}\Bigl(\frac{\mu(S)^{1+\langle \alpha\rangle_S}}{\mu\bigl(E(S)\bigr)^{1+\langle \beta\rangle_{E(S)}}}\Bigr) \Bigl\| \Bigl\{ \langle f\rangle_S \mu\bigl(E(S)\bigr)^\frac1{p^\mu_{E(S)}}\Bigr\}_{S\in \mathcal S}\Bigr\|_{\ell^{p^\mu_{E(S)}}}\\&
			\qquad\qquad
			\sup_{\substack{g\geq0\\ \|g\|_{L^{q'(\cdot)}}\leq1}}
			\Bigl\| \Bigl\{ \langle g\rangle_S \mu\bigl(E(S)\bigr)^\frac1{(q')^\mu_{E(S)}}\Bigr\}_{S\in \mathcal S}\Bigr\|_{\ell^{(p')^\mu_{E(S)}}}.
		\end{aligned}
	\end{equation}
	Applying Corollary \ref{corollary:carlesonembeddingapplication} gives
	\[
	\Bigl\| \Bigl\{ \langle f\rangle_S \mu\bigl(E(S)\bigr)^\frac1{p^\mu_{E(S)}}\Bigr\}_{S\in \mathcal S}\Bigr\|_{\ell^{p^\mu_{E(S)}}}\lesssim\|M_\mathcal S f\|_{L^{p(\cdot)}}.
	\]
	By Lemma \ref{lem:sequential}, using the standing assumption $p'(\cdot)\geq q'(\cdot)$, followed by Corollary \ref{corollary:carlesonembeddingapplication}, for every such $g$ we have
	\[
	\Bigl\| \Bigl\{ \langle g\rangle_S \mu\bigl(E(S)\bigr)^\frac1{(q')^\mu_{E(S)}}\Bigr\}_{S\in \mathcal S}\Bigr\|_{\ell^{(p')^\mu_{E(S)}}}\leq\Bigl\| \Bigl\{ \langle g\rangle_S \mu\bigl(E(S)\bigr)^\frac1{(q')^\mu_{E(S)}}\Bigr\}_{S\in \mathcal S}\Bigr\|_{\ell^{(q')^\mu_{E(S)}}}\lesssim\|M_\mathcal Sg\|_{L^{q'(\cdot)}}.
	\]
	Chaining the previous estimates gives
	\begin{align*}
		\bigl\| \mathcal{A}^{\alpha(\cdot)}_{\mathcal S}f \bigr\|_{L^{q(\cdot)}}&\lesssim 
		\sup_{S\in\mathcal{S}} \frac{\mu(S)^{1+\langle\alpha\rangle_S}}{\mu(E(S))^{1+\langle\beta\rangle_{E(S)}}} 
		\|M_{\mathcal S}\|_{L^{p(\cdot)}}\|M_{\mathcal S}\|_{L^{q'(\cdot)}}\|f\|_{L^{p(\cdot)}}.
	\end{align*}
It remains to bound the supremum in front. 
	By Lemma \ref{lem:IndicatorKopaliani} and the relation $1-\beta(\cdot) = 1/p'(\cdot)+1/q(\cdot)$ we have   
	\begin{align*}
		\frac{\mu(S)^{1+\langle\alpha\rangle_S}}{\mu(E(S))^{1+\langle\beta\rangle_{E(S)}}} &= \left(\frac{\mu(S)}{\mu(E(S))}\right)^2 \frac{\mu(E(S))^{1-\langle\beta\rangle_{E(S)}}}{\mu(S)^{1-\langle\alpha\rangle_S}} \\ &= [1]_{A_{p(\cdot),q(\cdot)}^{\alpha(\cdot)}(S)}\left(\frac{\mu(S)}{\mu(E(S))}\right)^2 \frac{\mu(E(S))^{1-\langle\beta\rangle_{E(S)}}}{\|1_S\|_{L^{p'(\cdot)}}\|1_S\|_{L^{q(\cdot)}}} \lesssim [1]_{A_{p(\cdot),q(\cdot)}^{\alpha(\cdot)}(S)}\left(\frac{\mu(S)}{\mu(E(S))}\right)^2.
	\end{align*}   
Denote 
\[
\eta := 1- \frac{1}{50[1]_{A_{1/\alpha(\cdot)}(\mathcal{S})}}.
\]
By the definition of $E(S)$ being an $(\eta,\mu(\cdot)^{1-\langle\alpha\rangle_{\cdot}})$-major subset of $S$ we have
\begin{align*}
    \mu(G)^{1-\langle\alpha\rangle_{G}}\leq (1-\eta)\mu(S)^{1-\langle\alpha\rangle_{S}},\qquad G := S\setminus E(S).
\end{align*} 
Thus, by using Lemma \ref{lem:IndicatorKopaliani}, and recalling the definition of $\eta,$ we have 
	\begin{align*}
			\frac{\mu(G)}{\mu(S)} &= \frac{\mu(G)^{1-\langle\alpha\rangle_G}}{\mu(S)^{1-\langle\alpha\rangle_S}} \frac{\mu(G)^{\langle\alpha\rangle_G}}{\mu(S)^{\langle\alpha\rangle_S}} \\ &\leq (1-\eta) 25[1]_{A_{1/\alpha(\cdot)}(\mathcal{S})} \frac{\|1_{G}\|_{L^{1/\alpha(\cdot)}}}{\|1_{S}\|_{L^{1/\alpha(\cdot)}}} \leq(1-\eta) 25[1]_{A_{1/\alpha(\cdot)}(\mathcal{S})} = \frac{1}{2}.
	\end{align*}
Thus $\mu(E(S)) = \mu(S) - \mu(G) \geq \mu(S)/2$ and hence, chaining the above estimates, we have shown that 
\[
\frac{\mu(S)^{1+\langle\alpha\rangle_S}}{\mu(E(S))^{1+\langle\beta\rangle_{E(S)}}} \lesssim [1]_{A_{p(\cdot),q(\cdot)}^{\alpha(\cdot)}(S)}\left(\frac{\mu(S)}{\mu(E(S))}\right)^2 \lesssim [1]_{A_{p(\cdot),q(\cdot)}^{\alpha(\cdot)}(S)}.
\]
\end{proof}

\subsection{Dyadic Riesz potential with variable fractionality}\label{sect:dyadic2}
The operator  $M^{\alpha(\cdot)}_{\mu,\mathcal{D}}$ 
picks pointwise the single scale with the largest variable fractional mean, while the dyadic variable Riesz potential
\begin{align*}
	I^{\alpha(\cdot)}_{\mu, \mathcal{D}}f := \sum_{Q\in\mathcal{D}} \frac{1_Q}{\mu(Q)^{1-\langle\alpha\rangle_Q}}\int_Q f\ud\mu
\end{align*}
sums them all up.   
When $\alpha(\cdot) = \alpha>0$ is a constant and $\mu = \mathcal{L}$ is the Lebesgue measure on $\mathbb{R}^n,$ a geometric sum on scales is possible: $\sum_{P\in\mathcal{D}(Q)}\mathcal{L}(P)^{\alpha}1_P\lesssim_{\alpha}\mathcal{L}(Q)^{\alpha},$ which can be used to control $I^{\alpha(\cdot)}_{\mathcal{D}}$ by $M^{\alpha(\cdot)}_{\mu,\mathcal{D}}$ on the level of norms.  When $\alpha(\cdot)$ is a function there is no reason to expect such bounds to hold. It turns out, by testing the operator with indicators of cubes, that the scales need to sum up in terms of a packing condition.

To formulate the result, we assume the collection $\mathcal{D}$ to have the $A_1$ self-improvement property ($A_1$ SIP), see Section \ref{sect:A1SIP} below. If $\mathcal{D}$ has the $A_1$ SIP, then $M_{\mathcal{D}}$ has the self-improving property on $L^{q'(\cdot)}:$ there exists $s>1$ so that $M_{s,\mathcal{D}}f := \sup_{Q\in\mathcal{D}}1_Q\langle |f|^s\rangle_Q^{1/s}$ is bounded on $L^{q'(\cdot)}$ provided $M_{\mathcal{D}}$ is bounded on $L^{q'(\cdot)},$ see Corollary \ref{cor:SIPMaxFunDyadic} below. 
Notably, any dyadic lattice $\mathcal{D}$ on $\mathbb{R}^n$, equipped with the Lebesgue measure, has the $A_1$ SIP. To state the result, we write 
\begin{align}\label{eq:thm:DyadRieszNec}
 C^{\alpha(\cdot)}_{\downarrow}(\mathcal{D},\mu) :=	\sup_{Q_0\in\mathcal{D}}  \frac{\sum_{Q\in\mathcal{D}(Q_0)}  \mu(Q)^{1+\langle\alpha\rangle_Q}}{\mu(Q_0)^{1+\langle\alpha\rangle_{Q_0}}} < \infty
\end{align}

\begin{theorem}\label{thm:RieszPot} 
	Let $\mathcal{D}$ be a collection of dyadic cubes.
	Let  $p(\cdot),$  $q(\cdot),$ $1/\alpha(\cdot)\in \mathcal{P}(X,\mu)$ be such that $1<p(\cdot)\leq q(\cdot)<\infty.$
	The following assertions hold. 
	\begin{itemize} 
		\item  Suppose that $\mathcal{D}$ has the $A_1$ SIP, that $M_{\mathcal{D}}$ is bounded on $L^{q'(\cdot)}$ and that $1\in A_{\frac{1}{\alpha(\cdot)}}(\mathcal{D}).$ If $M^{\alpha(\cdot)}_{\mathcal{D}}:L^{p(\cdot)}\to L^{q(\cdot)}$ is bounded and $C^{\alpha(\cdot)}_{\downarrow}(\mathcal{D},\mu)<\infty$, then  $I^{\alpha(\cdot)}_{\mathcal{D}}:L^{p(\cdot)}\to L^{q(\cdot)}$ is bounded.
		\item  Suppose that $\alpha(\cdot)= \frac{1}{p(\cdot)}-\frac{1}{q(\cdot)}$ and $1\in A_{p(\cdot)}(\mathcal{D})\cap A_{q(\cdot)}(\mathcal{D}).$ If $I^{\alpha(\cdot)}_{\mathcal{D}}:L^{p(\cdot)}\to L^{q(\cdot)}$ is bounded, then $M^{\alpha(\cdot)}_{\mathcal{D}}:L^{p(\cdot)}\to L^{q(\cdot)}$ is bounded and $C^{\alpha(\cdot)}_{\downarrow}(\mathcal{D},\mu)<\infty.$ 
\end{itemize}
\end{theorem}
\begin{remark}\label{rem:AlphaRelRieszPot}
In contrast to the maximal function theorems, we  assume above in the second bullet that $\alpha(\cdot)= \frac{1}{p(\cdot)}-\frac{1}{q(\cdot)}.$ The effect is that in the bound \eqref{eq:lem:DyadRieszNec2} below we do not start tracking the fractions
$\tfrac{\mu(Q_0)^{1/p^{\mu}_{Q_0}- 1/q^{\mu}_{Q_0}}}{ \mu(Q_0)^{\langle\alpha\rangle_{Q_0}}}.$
The interested reader can see that their uniform boundedness would act as a necessary condition weaker than $\alpha(\cdot)= \tfrac{1}{p(\cdot)}-\tfrac{1}{q(\cdot)}$ in the second bullet.
\end{remark}

By $M^{\alpha(\cdot)}_{\mathcal{D}}f \leq I^{\alpha(\cdot)}_{\mathcal{D}}|f|$ the first conclusion in the second bullet is immediate. The necessity of the packing condition will follow by testing the norm estimate with indicators and using the Muckenhoupt type conditions appropriately. We now start the proof. 
Before proceeding, we recall the following notation for $L^s$-averages 
$\|f\|_{\avL^{s}(Q)} := \frac{\|f1_Q\|_{L^{s}}}{\|1_Q\|_{L^{s}}}.$

\begin{lemma}\label{lem:DyadRieszNec} Let  $p(\cdot),q(\cdot),1/\alpha(\cdot)\in \mathcal{P}(X,\mu)$ be such that $\alpha(\cdot) = \frac{1}{p(\cdot)}-\frac{1}{q(\cdot)}.$ Suppose that $I^{\alpha(\cdot)}_{\mathcal{D}}:L^{p(\cdot)}\to L^{q(\cdot)}$ is bounded and $1\in A_{p(\cdot)}(\mathcal{D})\cap A_{q(\cdot)}(\mathcal{D}).$ 
	Then, \eqref{eq:thm:DyadRieszNec} holds, under  the standing assumptions of Theorem \ref{thm:RieszPot}, since we have 
	\begin{align}\label{eq:lem:DyadRieszNec}
	 \sup_{Q_0\in\mathcal{D}}  \frac{\sum_{Q\in\mathcal{D}(Q_0)}  \mu(Q)^{1+\langle\alpha\rangle_Q}}{\mu(Q_0)^{1+\langle\alpha\rangle_{Q_0}}} \leq  5^2[1]_{A_{p(\cdot)}(\mathcal{D})}  [1]_{A_{q(\cdot)}(\mathcal{D})}\|I^{\alpha(\cdot)}_{\mathcal{D}}\|_{L^{p(\cdot)}\to L^{q(\cdot)}}. 
	\end{align}
\end{lemma}
\begin{proof} By H\"older's inequality and $1\in A_{q(\cdot)}(\mathcal{D}),$ and by Proposition \ref{prop:Necessity}, we have 
	\begin{multline}
	\frac{1}{\mu(Q_0)}\sum_{Q\in\mathcal{D}(Q_0)}  \mu(Q)^{1+\langle\alpha\rangle_Q} =\Bigl\| \sum_{Q\in\mathcal{D}(Q_0)} \mu(Q)^{\langle\alpha\rangle_Q} 1_Q \Bigr\|_{\avL^{1}}  \\ \leq [1]_{A_{q(\cdot)}(\mathcal{D})} \Bigl\| \sum_{Q\in\mathcal{D}(Q_0)} \mu(Q)^{\langle\alpha\rangle_Q} 1_Q \Bigr\|_{\avL^{q(\cdot)}}	
	\leq	[1]_{A_{q(\cdot)}(\mathcal{D})} \frac{  \|1_{Q_0}\|_{L^{p(\cdot)}}}{  \|1_{Q_0}\|_{L^{q(\cdot)}}}  \|I^{\alpha(\cdot)}_{\mathcal{D}}\|_{L^{p(\cdot)}\to L^{q(\cdot)}}.
	\end{multline}
We estimate the fraction by Lemma \ref{lem:IndicatorKopaliani} as
	\begin{align}\label{eq:lem:DyadRieszNec2}
	\frac{\|1_{Q_0}\|_{L^{p(\cdot)}}}{\|1_{Q_0}\|_{L^{q(\cdot)}}} \leq 5^2[1]_{A_{p(\cdot)}(\mathcal{D})} \frac{\mu(Q_0)^{1/p^{\mu}_{Q_0}}}{\mu(Q_0)^{1/q^{\mu}_{Q_0}}}  = 5^2[1]_{A_{p(\cdot)}(\mathcal{D})}  \mu(Q_0)^{\langle\alpha\rangle_{Q_0}}.
\end{align} 
\end{proof}

The main goal in proving Theorem \ref{thm:RieszPot} is then to reverse the implication. 
We already saw above that the packing condition \eqref{eq:thm:DyadRieszNec} can be written in the equivalent form 
\[
\sup_{Q_0\in\mathcal{D}}\biggl\| \frac{1}{\mu(Q_0)^{\langle\alpha\rangle_{Q_0}}}\sum_{Q\in\mathcal{D}(Q_0)} \mu(Q)^{\langle\alpha\rangle_Q} 1_Q \biggr\|_{\avL^{1}} < \infty.
\]
An important step, for us, will be to replace the $\avL^{1}(Q)$ averages by $\avL^{s}(Q)$ averages for some large exponent $s>1$. That is the content of the following John-Nirenberg type lemma \ref{lem:JNforSequences}. Such proofs are standard, and the proof we present is largely the same as \cite{MuscaluSchlag2013vol2}*{Theorem 2.7}; however, we have the extra $\beta_Q$ coefficients and thus present the entire proof.

\begin{lemma}\label{lem:JNforSequences} Let  $\mathcal{D}$ be a collection of dyadic cubes that index two sequences $\beta_Q, \gamma_Q \geq 0$ such that 
	\begin{align}\label{eq:prop:CarEmbed1}
		A_1 := \sup_{Q\in\mathcal{D}}\sup_{P\in\mathcal{D}(Q)} \frac{\beta_Q}{\beta_P} < \infty,\qquad  A_2 :=	\sup_{Q\in\mathcal{D}}\beta_Q \Big\| \sum_{P\in\mathcal{D}(Q)} 1_P \gamma_P\Big\|_{\avL^{1}(Q)} < \infty.
	\end{align} 
	Let $s\in [1,\infty)$ be arbitrary. Then, there exists a finite constant $C(s,A_1,A_2)>0$ so that 
	\begin{align}\label{eq:prop:JNseq2}
		B := \sup_{Q\in\mathcal{D}}\beta_Q \Big\| \sum_{P\in\mathcal{D}(Q)} 1_P \gamma_P\Big\|_{\avL^{s}(Q)} \leq C(s,A_1,A_2).
	\end{align}
\end{lemma}

\begin{proof} 
Without loss of generality we can assume that $\mathcal{D}$ is finite.
Indeed, if \eqref{eq:prop:JNseq2} holds with $\mathcal{D}$ replaced with an arbitrary finite subcollection $\mathcal{D}'\subset\mathcal{D}$, then clearly the claim holds also with the countable $\mathcal{D}.$ 
	 By finiteness of $\mathcal{D}$ also $B$ is finite and there exists a cube $Q$ that realizes $B,$ i.e. $\beta_Q \Big\| \sum_{P\in\mathcal{D}(Q)} 1_P \gamma_P\Big\|_{\avL^{s}(Q)} = B.$  Let $M>0$ and consider the set 
	\[
	E_M := Q\cap \biggl\{ \sum_{P\in\mathcal{D}(Q)} 1_P \gamma_P > A_2\beta_Q^{-1}M \biggr\}.
	\]
	By the definition of the constant $A_2$ and Chebyshev's inequality $E_M$ is small
	\begin{align}\label{eq:prop;CarEmbed3}
	\frac{\mu(E_M)}{\mu(Q)}= \|1_{E_M}\|_{\avL^1(Q)} \leq \frac{1}{M}.
\end{align}
	Write $E_M$ as a disjoint union 
	\begin{align*}
		E_M = \bigcup_{J\in \mathcal{J}_{max} } J,\qquad 	{\mathcal{J}}_{max} = \Bigl\{\text{maximal } J \in \mathcal{D}(Q) :   \sum_{J \subset P\in\mathcal{D}(Q)} 1_P \gamma_P > A_2\beta_Q^{-1}M  \Bigr\}.
	\end{align*}
	Denote  $h_J := \sum_{P\in\mathcal{D}(J)} 1_P \gamma_P.$
	By the maximality of the cubes $J\in\mathcal{J}_{max}$, we have that 
	\begin{align*}
		\|h_Q\|_{L^s}&\leq \|1_{Q\setminus E_M}h_Q\|_{L^s} +  \|1_{E_M}h_Q\|_{L^s} \\ 
		&\leq \frac{A_2}{\beta_Q}M \|1_{Q\setminus E_M}\|_{L^s} 
		+  \Bigl\|  \sum_{J\in\mathcal{J}_{max}} \sum_{J\subsetneq P\in\mathcal{D}(Q)} 1_P \gamma_P \Bigr\|_{L^s} 
		+  \Bigl\|  \sum_{J\in\mathcal{J}_{max}} h_J \Bigr\|_{L^s} \\ 
		&\leq \frac{A_2}{\beta_Q}M \|1_{Q\setminus E_M}\|_{L^s} 
		+   \frac{A_2}{\beta_Q}M \|1_{E_M}\|_{L^s} 
		+  \Bigl\|  \sum_{J\in\mathcal{J}_{max}} h_J \Bigr\|_{L^s}  \\ 
		&\leq 2\frac{A_2}{\beta_Q}M \|1_{Q}\|_{L^s} 
		+  \Bigl(  \sum_{J\in\mathcal{J}_{max}} \|h_J\|_{L^s}^s\Bigr)^{1/s}.
	\end{align*}
	We estimate the right-most term as 
	\begin{align*}
		\Bigl(  \sum_{J\in\mathcal{J}_{max}} \|h_J\|_{L^s}^s\Bigr)^{1/s} &=  \Bigl(  \sum_{J\in\mathcal{J}_{max}} \mu(J) \|h_J\|_{\avL^s}^s\Bigr)^{1/s} \\  
		&\leq  \Bigl(  \sum_{J\in\mathcal{J}_{max}} \mu(J) \bigl( \frac{B}{\beta_J}\bigr)^s\Bigr)^{1/s}  \leq \sup_{J\in\mathcal{J}_{max}} \frac{B}{\beta_J}  \mu(E_M)^{1/s}.
	\end{align*}
	By the definition of $B$, chaining the above bounds and multiplying with $\beta_Q/\|1_Q\|_{L^s}$ yields
	\begin{align*}
		B = \beta_Q\|h_Q\|_{\avL^s} \leq 2A_2M +  \left(\sup_{J\in\mathcal{J}_{max}} \frac{\beta_Q}{\beta_J}   \frac{\mu(E_M)^{1/s}}{\mu(Q)^{1/s}} \right)B \leq 2A_2M +  A_1 \left( \frac{1}{M}\right)^{1/s} B.
	\end{align*}
	Choosing 
	$M := (2A_1)^{s}$ and rearranging gives $B\leq 4A_2 (2A_1)^s =: C(s, A_1, A_2) < \infty.$
\end{proof}

Combining the above lemmas we obtain the following corollary.

\begin{corollary}\label{cor:RieszPotPacking} Let $\mathcal{D}$ be a collection of dyadic cubes, $s\in [1,\infty)$ a constant and $1/\alpha(\cdot)\in\mathcal{P}(X,\mu)$. Suppose that \eqref{eq:thm:DyadRieszNec} holds and $1\in A_{\frac1{\alpha(\cdot)}}(\mathcal D).$
	Then, there exists a finite constant so that 
	\begin{align}\label{eq:cor:RieszPotPacking}
    \sup_{Q_0\in\mathcal{D}}\Bigl\|  \frac{1}{\mu(Q_0)^{\langle\alpha\rangle_{Q_0}}} \sum_{Q\in\mathcal{D}(Q_0)} 1_Q  \mu(Q)^{\langle\alpha\rangle_Q} \Bigr\|_{\avL^{s}(Q_0)} \leq  C(s,[1]_{A_{1/\alpha(\cdot)}(\mathcal{D})}) < \infty.
	\end{align}
\end{corollary}
\begin{proof} 
	We apply Lemma \ref{lem:JNforSequences} with the sequence $\beta_Q := \mu(Q)^{-\langle\alpha\rangle_Q}.$ The finiteness of $A_2$ is a standing assumption and hence it  remains to check that $A_1$ is finite. For each $P\subset Q$, we have
	\[
	\frac{\beta_Q}{\beta_P} = \frac{\mu(P)^{\langle\alpha\rangle_P}}{\mu(Q)^{\langle\alpha\rangle_Q}} \leq 25[1]_{A_{1/\alpha(\cdot)}(\mathcal{D})}\frac{\|1_P\|_{L^{1/\alpha(\cdot)}}}{\|1_Q\|_{L^{1/\alpha(\cdot)}}} \leq 25[1]_{A_{1/\alpha(\cdot)}(\mathcal{D})} < \infty.
	\qedhere\]
\end{proof}

\subsubsection{$A_1$ self-improvement property of a dyadic collection}\label{sect:A1SIP}
Before putting together the remaining part of the proof of Theorem \ref{thm:RieszPot}
we need one more property from $L^{q'(\cdot)}$, which is the self-improving property of the maximal function $M_{\mathcal{D}}$. 
We formulate a sufficient condition in terms of the $A_1$ self-improving property of bases $\mathcal{E},$ which results in the self-improving property of the maximal operator. We direct the reader to Theorem \ref{thm:SIPMaxFun} below. We closely follow Nieraeth \cite{Nier23Extrap}*{Section 2}.

\begin{definition}[Basis]\label{defn:basis} A countable collection $\mathcal{E}$ of measurable sets in a $\sigma$-finite measure space $(X,\mu)$ is a basis if (i) $0<\mu(E)<\infty,$ for all $E\in\mathcal{E},$ (ii) $\bigcup_{E\in\mathcal{E}} E = X,$ and (iii) for all $x_1,x_2\in X$ there exists $E\in\mathcal{E}$ so that $x_1,x_2\in E.$
\end{definition}

\begin{definition}[$A_1$ weights]\label{defn:A1} Let $\mathcal{E}$ be a countable collection of measurable sets in a $\sigma$-finite measure space $(X,\mu)$ such that $0<\mu(E)<\infty$ for all $E\in\mathcal{E}.$ Then we say that an almost everywhere positive function $w>0$ is an $A_1(\mathcal{E})$ weight if there exists a constant $C = C_w>0$ so that 
	\[
	M_{\mathcal{E}}w := \sup_{E\in\mathcal{E}} \frac{1_E}{\mu(E)}\int_E w\ud\mu \leq C w.
	\] 
\end{definition}

\begin{definition}[$A_1$ self-improvement for bases]\label{defn:A1SelfImprovement} We say that a basis of sets $\mathcal{E}$ has the $A_1$ self-improvement property if for every constant $C_1>0$ there exist \(s=s(C_1)>1\) and \(C_2=C(C_1)\) such that, for every weight \(w\), 
		\[
		[w]_{A_1(\mathcal{E})} \leq C_1 \Longrightarrow 	[w^s]_{A_1(\mathcal{E})} \leq C_2.
		\]
\end{definition}

To state the self-improving property of the maximal function, denote  
\[
M_{s,\mathcal{E}}f := \sup_{E\in\mathcal{E}} 1_E\left( \frac{1}{\mu(E)} \int_E|f|^s\ud\mu \right)^{1/s}.
\]
\begin{theorem}[{\cite{Nier23Extrap}*{Theorem~2.34}}]\label{thm:SIPMaxFun} Let $\mathcal{E}$ be a basis of sets in $X$ with the $A_1$ self-improvement property with $s>1$. Let $Y$ be a BFS over $X$ and suppose that $\|M_{\mathcal{E}}\|_Y < \infty$. Then, there holds that $\|M_{s,\mathcal{E}}\|_Y < \infty.$
\end{theorem}

An arbitrary collection of dyadic cubes $\mathcal{D}$ might fail to be a basis (for example, lattices with quadrants). However, for each fixed cube $Q_0\in\mathcal{D}$, the restricted collection $\mathcal{D}(Q_0)$ is a basis on the restricted measure space $(Q_0,\mu|_{Q_0})$. 

\begin{definition}\label{defn:A1SIPdyadic} A collection $\mathcal{D}$ of dyadic cubes has the $A_1$ self-improvement property if, for each fixed $Q_0\in\mathcal{D}$, the collection $\mathcal{D}(Q_0)$, viewed as a basis on $(Q_0,\mu|_{Q_0})$, has the $A_1$ self-improvement property with constants $s=s(C_1)>1$ and $C_2=C_2(C_1)$ of Definition \ref{defn:A1SelfImprovement} that do not depend on $Q_0.$
\end{definition}

The following corollary is an immediate consequence of Definition \ref{defn:A1SIPdyadic} and Theorem \ref{thm:SIPMaxFun}, applied on the restricted measure spaces $(Q_0,\mu|_{Q_0})$ with constants uniform in $Q_0$.

\begin{corollary}\label{cor:SIPMaxFunDyadic} Let $Y$ be a BFS over a measure space $(X,\mu)$. Let $\mathcal{D}$ be a collection of dyadic cubes with the $A_1$ self-improvement property. Then, there exists $s>1$ so that 
	\[
	\|M_{\mathcal{D}}\|_Y < \infty\Longrightarrow \sup_{Q_0\in\mathcal{D}}\|M_{s,\mathcal{D}(Q_0)}\|_Y < \infty.
	\]
    	Here $M_{s,\mathcal{D}(Q_0)}$ is computed on $(Q_0,\mu|_{Q_0})$ and extended by zero outside $Q_0$. Indeed, if $Y(Q_0)$ denotes the restriction of $Y$ to $Q_0$ with the zero-extension norm, then $\|M_{\mathcal{D}(Q_0)}\|_{Y(Q_0)}\leq \|M_{\mathcal{D}}\|_Y$. Definition \ref{defn:A1SIPdyadic} and Theorem \ref{thm:SIPMaxFun} therefore give a bound for $\|M_{s,\mathcal{D}(Q_0)}\|_{Y(Q_0)}$ with the same $s$ and with constants independent of $Q_0$, which is the displayed estimate.
\end{corollary}

Now we are in a position to conclude the proof of Theorem \ref{thm:RieszPot}.

\begin{proof}[Proof of Theorem \ref{thm:RieszPot}] The testing criteria were already checked above and it remains to prove these imply the boundedness of $I^{\alpha(\cdot)}_{\mathcal{D}}.$ We first reduce to finite subcollections. Enumerate \(\mathcal D=\{Q_j\}_{j=1}^{\infty}\), let \(\mathcal D_N:=\{Q_1,\ldots,Q_N\}\), and fix a non-negative function \(f\). It is enough to prove
\[
\|I^{\alpha(\cdot)}_{\mathcal D_N}f\|_{L^{q(\cdot)}}
\lesssim
\|M^{\alpha(\cdot)}_{\mathcal D}f\|_{L^{q(\cdot)}}
\]
with an implicit constant independent of \(N\). Indeed, \(I^{\alpha(\cdot)}_{\mathcal D_N}f\uparrow I^{\alpha(\cdot)}_{\mathcal D}f\) pointwise, so the Fatou property of \(L^{q(\cdot)}\), Proposition \ref{prop:Lp=Fatou}, gives the bound for \(I^{\alpha(\cdot)}_{\mathcal D}f\). Applying this to \(|f|\) then gives the estimate for general \(f\).

Fix \(N\), and let \(\mathcal T_N\) be the collection of maximal cubes in \(\mathcal D_N\). Since \(\mathcal D\) is dyadic, the cubes in \(\mathcal T_N\) are pairwise disjoint and every cube in \(\mathcal D_N\) is contained in exactly one cube of \(\mathcal T_N\). For each \(Q_0\in\mathcal T_N\), we use Proposition \ref{prop:SDomSub} on \(\mathcal D_N(Q_0)\) with the auxiliary function \(\psi=\mu\), which gives the usual principal stopping cubes. Equivalently, start with \(Q_0\) and, for each stopping cube \(Q\), stop at the maximal subcubes \(P\) satisfying \(\langle f\rangle_P>2\langle f\rangle_Q\). Let \(\mathcal S_f(Q_0)\) be the resulting sparse collection and put
\[
\mathcal S_N:=\bigcup_{Q_0\in\mathcal T_N}\mathcal S_f(Q_0).
\]
The union \(\mathcal S_N\) is still \(\mu\)-sparse. For \(Q\in\mathcal D_N\), let \(\Pi_{\mathcal S_N}Q\) denote the unique minimal cube \(P\in\mathcal S_N\) with \(Q\subset P\). Using the maximality of the stopping parent, we have
	\begin{align*}
		I^{\alpha(\cdot)}_{\mathcal D_N}f
		&= \sum_{P\in\mathcal{S}_N} \sum_{\substack{Q\in\mathcal D_N\\\Pi_{\mathcal{S}_N}Q = P}} \frac{1_Q}{\mu(Q)^{1-\langle\alpha\rangle_Q}}\int_Q f\ud\mu  \\
		&\lesssim \sum_{P\in\mathcal{S}_N} \langle f\rangle_P  \sum_{\substack{Q\in\mathcal D_N\\\Pi_{\mathcal{S}_N}Q = P}} 1_Q \mu(Q)^{\langle\alpha\rangle_Q} \\
		&\leq  \sum_{P\in\mathcal{S}_N} \langle f\rangle_P  \sum_{Q\in\mathcal{D}(P)} 1_Q \mu(Q)^{\langle\alpha\rangle_Q}
		= \sum_{P\in\mathcal{S}_N} \langle f\rangle_P 1_P h_P,
	\end{align*}
	where
	$
	h_P :=  \sum_{Q\in\mathcal{D}(P)} 1_Q \mu(Q)^{\langle\alpha\rangle_Q}.
	$
	By Corollary \ref{cor:SIPMaxFunDyadic}, applied to the restricted bases $\mathcal{D}(Q_0)$, we find $s = s(\|M_{\mathcal{D}}\|_{L^{q'(\cdot)}\to L^{q'(\cdot)}})>1$ so that 
	\[
	\sup_{Q_0\in\mathcal{D}}\|M_{s,\mathcal{D}(Q_0)}\|_{L^{q'(\cdot)}} < \infty, \qquad \sup_{P\in\mathcal{S}_N}\Bigl\| \frac{1}{\mu(P)^{\langle\alpha\rangle_P}} h_P\Bigr\|_{\avL^{s'}(P)} < \infty,
	\]
	where the latter bound follows from Corollary \ref{cor:RieszPotPacking}. 
	By Lemma \ref{lem:dualityFormula}, it is enough to bound the pairing uniformly over non-negative functions $g$ with $\|g\|_{L^{q'(\cdot)}} \leq 1$. For such \(g\), define
\[
G_N:=\sum_{Q_0\in\mathcal T_N}1_{Q_0}M_{s,\mathcal D(Q_0)}(g1_{Q_0}).
\]
The disjointness of the top cubes and the uniform restricted bound above, with the restricted maximal operators extended by zero, give \(\|G_N\|_{L^{q'(\cdot)}}\lesssim\|g\|_{L^{q'(\cdot)}}\), uniformly in \(N\). H\"older's inequality and \(\mu(\cdot)\)-sparsity give
\begin{equation}\label{eq:xcv}
	\begin{split}
		\int g \sum_{P\in\mathcal{S}_N} \langle f\rangle_P 1_P h_P \ud\mu&= 	\sum_{P\in\mathcal{S}_N} \langle f\rangle_P  \|gh_P\|_{\avL^1(P)}\mu(P) \\
				&\leq \sum_{P\in\mathcal{S}_N} \langle f\rangle_P \mu(P)^{\langle\alpha\rangle_P}  \left\| \frac{1}{\mu(P)^{\langle\alpha\rangle_P}} h_P\right\|_{\avL^{s'}(P)}   \|g\|_{\avL^{s}(P)} \mu(P) \\
		&\lesssim \sum_{P\in\mathcal{S}_N} \langle f\rangle_P \mu(P)^{\langle\alpha\rangle_P} \|g\|_{\avL^{s}(P)} \mu(P) \\ 
		&\lesssim  \int \sum_{P\in\mathcal{S}_N}\left( \frac{1}{\mu(P)^{1-\langle\alpha\rangle_P}}\int_Pf\ud\mu\right) \|g\|_{\avL^{s}(P)} 1_{E(P)} \ud\mu\\
		&\lesssim  \int \sum_{P\in\mathcal{S}_N}\left( \frac{1_{E(P)}}{\mu(P)^{1-\langle\alpha\rangle_P}}\int_Pf\ud\mu\right)  G_N \ud\mu\\
		&\lesssim  \left\|\sum_{P\in\mathcal{S}_N} \frac{1_{E(P)}}{\mu(P)^{1-\langle\alpha\rangle_P}}\int_Pf\ud\mu\right\|_{L^{q(\cdot)}} \|G_N\|_{L^{q'(\cdot)}} \\
		&\leq   \|M^{\alpha(\cdot)}_{\mathcal{D}}f\|_{L^{q(\cdot)}} \|G_N\|_{L^{q'(\cdot)}} \lesssim  \|M^{\alpha(\cdot)}_{\mathcal{D}}f\|_{L^{q(\cdot)}}.
	\end{split}
\end{equation}
Taking the supremum over all such $g$ and using Lemma \ref{lem:dualityFormula} gives the desired bound for \(I^{\alpha(\cdot)}_{\mathcal D_N}f\), uniformly in \(N\). The assumed boundedness of \(M^{\alpha(\cdot)}_{\mathcal D}\) then completes the finite-subcollection reduction and the proof.
\end{proof}

\subsection{Dyadic Morrey spaces with variable fractionality}\label{sect:dyadic3}
Next we consider the maximal and Riesz potential operators between dyadic variable Morrey spaces,
$$
M^{\alpha(\cdot)}_{\mathcal{D}}, I^{\alpha(\cdot)}_{\mathcal{D}}: \mathcal M^{r(\cdot)}_{p(\cdot)}(\mathcal{D})\to \mathcal M^{s(\cdot)}_{q(\cdot)}(\mathcal{D}).
$$
We recall Definition \ref{defn:MorreyNorm} above and Proposition \ref{prop:MorreyProperties}, where basic properties of these Morrey spaces were checked. 
The upper bounds in Theorems \ref{thm:MorreyMaxIff} and \ref{thm:MorreyRieszIff} below will involve the operator norms
$
M^{\alpha(\cdot)}_{\mathcal{D}}, I^{\alpha(\cdot)}_{\mathcal{D}}:  L^{p(\cdot)}\to L^{q(\cdot)}
$
and both the upper and lower bounds involve the necessary testing constants provided by Proposition \ref{prop:Necessity}. 
Moreover, we make the background assumption that both pairs $(p(\cdot),q(\cdot))$ and $(r(\cdot),s(\cdot))$ result in the same fractionality
\begin{align}\label{eq:sectMorreyFracAssumption}
	\frac1{p(\cdot)}-\frac1{q(\cdot)}
	=\frac1{r(\cdot)}-\frac1{s(\cdot)}.
\end{align}
It would be interesting to know how far we can move away from the identity \eqref{eq:sectMorreyFracAssumption} and how to quantify this in the estimates below. 
\begin{theorem}\label{thm:dMaxMorrey}
Let 
$p(\cdot), q(\cdot), r(\cdot),s(\cdot),1/\alpha(\cdot)\in \mathcal P(X,\mu)$ satisfy \eqref{eq:sectMorreyFracAssumption}.
Let $\mathcal D$ be a dyadic collection of cubes. Then
\begin{equation}\label{eq:thm:dMaxMorrey}
C_{\eqref{thm:dMaxMorrey}} \leq\|M^{\alpha(\cdot)}_{\mathcal{D}}\|_{\mathcal M^{r(\cdot)}_{p(\cdot)}(\mathcal{D})\rightarrow \mathcal M^{s(\cdot)}_{q(\cdot)}(\mathcal{D})}
\lesssim
\|M^{\alpha(\cdot)}_{\mathcal{D}}\|_{L^{p(\cdot)}\to L^{q(\cdot)}}
+C_{\eqref{thm:dMaxMorrey}},
\end{equation}
where 
\[
C_{\eqref{thm:dMaxMorrey}}:=\sup_{R\in\mathcal{D}}\frac{\|1_R\|_{\mathcal M^{s(\cdot)}_{q(\cdot)}}\|1_R\|_{(\mathcal M^{r(\cdot)}_{p(\cdot)})'}}{\mu(R)^{1-\langle \alpha\rangle_R}}.
\]
\end{theorem}
\begin{proof}
The left bound of \eqref{eq:thm:dMaxMorrey} follows from \eqref{eq:prop:NecessityMax} of Proposition \ref{prop:Necessity}. We move to the second estimate.
By the lattice property it suffices to estimate $|f|$, so we may assume $f\geq0$.
Fix $R\in\mathcal{D}$. Then
\begin{equation}\label{eq: splitting maximal operator based on cube}
1_{R} M^{\alpha(\cdot)}_{\mathcal{D}}f\leq\sup_{\substack{Q\in\mathcal{D}\\Q\subseteq R}}\langle |f|\rangle_Q^{\alpha(\cdot)}1_Q+\sup_{\substack{Q\in\mathcal{D}\\R\subsetneq Q}}\langle |f|\rangle_Q^{\alpha(\cdot)}1_R=:M^{\alpha(\cdot)}_{\mathcal{D}(R)}f+M^{\alpha(\cdot),\mathrm{glo}}_{R}f.
\end{equation}
We estimate these two terms separately.
For the local term, we have 
\begin{equation}\label{eq: estimate for morrey norm local}
\begin{aligned}
\mu(R)^{\frac1{s^\mu_R}-\frac1{q^\mu_R}}
\|M^{\alpha(\cdot)}_{\mathcal{D}(R)}f\|_{L^{q(\cdot)}}
&\le
\mu(R)^{\frac1{s^\mu_R}-\frac1{q^\mu_R}}
\|M^{\alpha(\cdot)}_{\mathcal{D}(R)}\|_{L^{p(\cdot)}\to L^{q(\cdot)}}
\|f1_R\|_{L^{p(\cdot)}} \\
&\le
\|M^{\alpha(\cdot)}_{\mathcal{D}(R)}\|_{L^{p(\cdot)}\to L^{q(\cdot)}}
\|f\|_{\mathcal M^{r(\cdot)}_{p(\cdot)}(\mathcal{D})},
\end{aligned}
\end{equation}
where we have used \eqref{eq:sectMorreyFracAssumption} in the second line.

For the global term, we use Lemma \ref{lem:dualityFormula} to obtain
\begin{align*}
	\bigl\|M^{\alpha(\cdot),\mathrm{glo}}_{R}f\bigr\|_{L^{q(\cdot)}}&\eqsim \sup_{\|g\|_{q'(\cdot)}=1}\sup_{\substack{Q\in\mathcal D\\R\subsetneq Q}}\frac{1}{\mu(Q)^{1-\langle \alpha\rangle_Q}}\int_Q f\,\mathrm d\mu\int_R g\,\mathrm d\mu
	\\& 
	\lesssim\sup_{\|g\|_{q'(\cdot)}=1}\sup_{\substack{Q\in\mathcal D\\R\subsetneq Q}} \frac{\|1_Q \|_{(\mathcal M^{r(\cdot)}_{p(\cdot)})'}}{\mu(Q)^{1-\langle \alpha\rangle_Q}}\|f\|_{\mathcal M^{r(\cdot)}_{p(\cdot)}} \|g\|_{L^{q'(\cdot)}}\|1_R\|_{L^{q(\cdot)}}, 
\end{align*}
where we used Lemma \ref{lem:dualityFormula} and K\"othe duality in the last estimate. Using \eqref{eq:prop:NecessityMax} of Proposition \ref{prop:Necessity} and the lattice monotonicity $\|1_R\|_{\mathcal M^{s(\cdot)}_{q(\cdot)}}\leq\|1_Q\|_{\mathcal M^{s(\cdot)}_{q(\cdot)}}$ for $R\subset Q$ gives
\begin{equation*}
	\bigl\|M^{\alpha(\cdot),\mathrm{glo}}_{R}f\bigr\|_{L^{q(\cdot)}}\lesssim C_{\eqref{thm:dMaxMorrey}}\|f\|_{\mathcal M^{r(\cdot)}_{p(\cdot)}}\frac{\|1_R\|_{L^{q(\cdot)}}}{\|1_R\|_{\mathcal M^{s(\cdot)}_{q(\cdot)}}}\leq C_{\eqref{thm:dMaxMorrey}}\|f\|_{\mathcal M^{r(\cdot)}_{p(\cdot)}}\mu(R)^{\frac{1}{q^\mu_{R}}-\frac1{s^\mu_R}}.
\end{equation*}
Combining the local and global estimates taking the supremum over $R\in\mathcal D$, we conclude that
\begin{align*}
	\|M^{\alpha(\cdot)}_{\mathcal D}f\|_{\mathcal M^{s(\cdot)}_{q(\cdot)}(\mathcal D)}
	&=
	\sup_{R\in\mathcal D}
	\mu(R)^{\frac1{s^\mu_R}-\frac1{q^\mu_R}}
	\|1_R M^{\alpha(\cdot)}_{\mathcal D}f\|_{L^{q(\cdot)}} \\
	&\lesssim
	\Bigl(
	\|M^{\alpha(\cdot)}_{\mathcal D}\|_{L^{p(\cdot)}\to L^{q(\cdot)}}
	+C_{\eqref{thm:dMaxMorrey}}
	\Bigr)
	\|f\|_{\mathcal M^{r(\cdot)}_{p(\cdot)}(\mathcal D)}.\qedhere
\end{align*}
\end{proof}

The dyadic potential estimate is analogous.

\begin{theorem}\label{thm:dRiezMorrey}
Let 
$p(\cdot), q(\cdot), r(\cdot),s(\cdot),1/\alpha(\cdot)\in \mathcal P(X,\mu)$ satisfy \eqref{eq:sectMorreyFracAssumption}.
Let $\mathcal D$ be a dyadic collection of cubes. Then
\begin{equation}\label{eq:thm:dRiezMorrey}
C_{\eqref{thm:dRiezMorrey}} \leq\|I^{\alpha(\cdot)}_{\mathcal{D}}\|_{\mathcal M^{r(\cdot)}_{p(\cdot)}(\mathcal{D})\rightarrow\mathcal M^{s(\cdot)}_{q(\cdot)}(\mathcal{D})}
\lesssim
\|I^{\alpha(\cdot)}_{\mathcal{D}}\|_{L^{p(\cdot)}\to L^{q(\cdot)}}
+C_{\eqref{thm:dRiezMorrey}},
\end{equation}
where
\[
C_{\eqref{thm:dRiezMorrey}} :=\sup_{R\in\mathcal{D}}
\|1_R\|_{\mathcal M^{s(\cdot)}_{q(\cdot)}} \, \Bigl\|\sum_{\substack{Q\in\mathcal{D}\\ R\subset Q}}
\frac{1_Q}{\mu(Q)^{1-\langle \alpha\rangle_Q}}\Bigr\|_{(\mathcal M^{r(\cdot)}_{p(\cdot)})'}.
\]
\end{theorem}
\begin{proof}
The left bound of \eqref{eq:thm:dRiezMorrey} follows immediately from \eqref{eq:prop:NecessityPotential1}. We turn our attention to the second estimate.
By the lattice property we may assume that $f\geq0$.
Fix $R\in\mathcal{D}$. Then
\begin{equation}\label{eq: splitting riesz operator based on cube}
1_{R} I^{\alpha(\cdot)}_{\mathcal{D}}f=\sum_{\substack{Q\in\mathcal{D}\\Q\subseteq R}}\langle f1_R\rangle_Q^{\alpha(\cdot)}1_Q+\sum_{\substack{Q\in\mathcal{D}\\R\subsetneq Q}}\langle f\rangle_Q^{\alpha(\cdot)}1_R= I^{\alpha(\cdot)}_{\mathcal{D}(R)}(f1_R)+I^{\alpha(\cdot),\mathrm{glo}}_{R}(f).
\end{equation}
For the local term, we have
\begin{equation}\label{eq: estimate for morrey norm local riesz potential}
\begin{aligned}
\mu(R)^{\frac1{s^\mu_R}-\frac1{q^\mu_R}}
\|I^{\alpha(\cdot)}_{\mathcal{D}(R)}(f1_R)\|_{L^{q(\cdot)}}
&\le
\mu(R)^{\frac1{s^\mu_R}-\frac1{q^\mu_R}}
\|I^{\alpha(\cdot)}_{\mathcal{D}}\|_{L^{p(\cdot)}\to L^{q(\cdot)}}
\|f1_R\|_{L^{p(\cdot)}} \\
&\le
\|I^{\alpha(\cdot)}_{\mathcal{D}}\|_{L^{p(\cdot)}\to L^{q(\cdot)}}
\|f\|_{\mathcal M^{r(\cdot)}_{p(\cdot)}(\mathcal{D})} 
\end{aligned}
\end{equation}
where we have used \eqref{eq:sectMorreyFracAssumption} in the second line.

For the global term, we use Lemma \ref{lem:dualityFormula} to obtain
\begin{align*}
\bigl\|I^{\alpha(\cdot),\mathrm{glo}}_{R}f\bigr\|_{L^{q(\cdot)}}&\eqsim \sup_{\|g\|_{q'(\cdot)}=1}\sum_{\substack{Q\in\mathcal{D}\\R\subset Q}}\mu(Q)^{\langle \alpha\rangle_Q-1}\int_Qf\,\mathrm d\mu\int_R g\,\mathrm d\mu
\\&
\lesssim\sup_{\|g\|_{q'(\cdot)}=1}\Bigl\|\sum_{\substack{Q\in\mathcal{D}\\R\subset Q}}\frac{1_Q}{\mu(Q)^{1-\langle \alpha\rangle_Q}} \Bigr\|_{(\mathcal M^{r(\cdot)}_{p(\cdot)})'}\|f\|_{\mathcal M^{r(\cdot)}_{p(\cdot)}} \|g\|_{L^{q'(\cdot)}}\|1_R\|_{L^{q(\cdot)}}, 
\end{align*}
where we used Lemma \ref{lem:dualityFormula} and K\"othe duality in the last estimate. By \eqref{eq:prop:NecessityPotential1} of Proposition \ref{prop:Necessity},
\begin{equation*}
\bigl\|I^{\alpha(\cdot),\mathrm{glo}}_{R}f\bigr\|_{L^{q(\cdot)}}\lesssim C_{\eqref{thm:dRiezMorrey}} \|f\|_{\mathcal M^{r(\cdot)}_{p(\cdot)}}\frac{\|1_R\|_{L^{q(\cdot)}}}{\|1_R\|_{\mathcal M^{s(\cdot)}_{q(\cdot)}}}\leq C_{\eqref{thm:dRiezMorrey}} \|f\|_{\mathcal M^{r(\cdot)}_{p(\cdot)}}\mu(R)^{\frac{1}{q^\mu_{R}}-\frac1{s^\mu_R}},
\end{equation*}
which implies
\begin{equation*}
\mu(R)^{\frac1{s^\mu_R}-\frac1{q^\mu_R}}\bigl\|I^{\alpha(\cdot),\mathrm{glo}}_{R}f\bigr\|_{L^{q(\cdot)}}\lesssim C_{\eqref{thm:dRiezMorrey}} \|f\|_{\mathcal M^{r(\cdot)}_{p(\cdot)}}.
\end{equation*}
Combining the local and global estimates taking the supremum over $R\in\mathcal{D}$, we conclude that
\begin{align*}
\|I^{\alpha(\cdot)}_{\mathcal{D}}f\|_{\mathcal M^{s(\cdot)}_{q(\cdot)}(\mathcal{D})}
=
\sup_{R\in\mathcal{D}}
\mu(R)^{\frac1{s^\mu_R}-\frac1{q^\mu_R}}
\|1_R I^{\alpha(\cdot)}_{\mathcal{D}}f\|_{L^{q(\cdot)}}\lesssim
\Bigl(
\|I^{\alpha(\cdot)}_{\mathcal{D}}\|_{L^{p(\cdot)}\to L^{q(\cdot)}}
+C_{\eqref{thm:dRiezMorrey}} 
\Bigr)
\|f\|_{\mathcal M^{r(\cdot)}_{p(\cdot)}(\mathcal{D})}.
\end{align*}
\end{proof}

\section{Comparison with the dyadic models}\label{sect:cont-to-dyadic}
In this section we bootstrap the dyadic results of the previous section to the non-dyadic ones and provide proofs of Theorems \ref{thm:MainIntro1} and \ref{thm:MainIntro4}. 
The following classical covering dyadic lattices play the key role.
\begin{lemma}\label{lem:CoveringLattices} There exist $j=1,\dots,3^n$ dyadic lattices of cubes $\mathcal{D}^j$ on $\mathbb{R}^n$ with the following property. For each ball $B\subset\mathbb{R}^n$ there exists a $j = j(B)$ and a cube $Q = Q(B)\in \mathcal{D}^j$ so that $B\subset Q \subset 10B.$ 
\end{lemma}

Next we compare the variable maximal and Riesz potential operators with their dyadic counterparts. 
\begin{proposition}\label{prop:mean-to-dyadic} Let $\frac{1}{\alpha(\cdot)}\in\mathcal P(\mathbb R^n)$ and $\mathcal B$ be the collection of all balls in $\mathbb R^n$. Let $\mathcal{D}$ stand for an arbitrary dyadic lattice of cubes and $\mathcal{D}^j$ for the $j=1,\dots, 3^n$ dyadic lattices of Lemma \ref{lem:CoveringLattices}. 

	Then, 
	\begin{align}
		M^{\alpha(\cdot)}f &\lesssim_n  [1]_{A_{\frac{1}{\alpha(\cdot)}}(\mathcal B)}\sum_{j=1}^{3^n} 	M^{\alpha(\cdot)}_{\mathcal{D}^j}f \label{eq:1},\\ 	
		M^{\alpha(\cdot)}_{\mathcal{D}}f&\lesssim_n  [1]_{A_{\frac{1}{\alpha(\cdot)}}(\mathcal B)}M^{\alpha(\cdot)}f,  \label{eq:2} \\ 
		I^{\alpha(\cdot)}|f| &\lesssim_n [1]_{A_{\frac{1}{\alpha(\cdot)}}(\mathcal B)}\sum_{j=1}^{3^n} 	I^{\alpha(\cdot)}_{\mathcal{D}^j}|f|, \label{eq:3} \\ 
				I^{\alpha(\cdot)}_{\mathcal{D}}|f| &\lesssim_n C^{\alpha(\cdot)}_{\uparrow}(\mathcal{D})[1]_{A_{\frac{1}{\alpha(\cdot)}}(\mathcal B)}	I^{\alpha(\cdot)}|f|\label{eq:4},
	\end{align}
where the constant in the bound \eqref{eq:4} is
\[
C^{\alpha(\cdot)}_{\uparrow}(\mathcal{D}) := \sup_{Q_0\in\mathcal{D}}
\sum_{Q_0\subset Q\in\mathcal{D}}
\frac{|Q_0|^{1-\langle\alpha\rangle_{Q_0}}}{|Q|^{1-\langle\alpha\rangle_Q}}.
\]
\end{proposition}
\begin{proof}
	Before verifying the bounds, we note the following. Let $s(\cdot)\in\mathcal{P}.$ For any cube $Q \subset B_Q,$ where $|B_Q|\lesssim_n |Q|,$ there holds that $[1]_{A_{s(\cdot)}(Q)} \lesssim_n [1]_{A_{s(\cdot)}(B_Q)} \leq [1]_{A_{s(\cdot)}(\mathcal{B})}.$ Conversely, using the $j=1,\dots, 3^n$ dyadic lattices of Lemma \ref{lem:CoveringLattices}, stating that for each ball $B$ there exists $Q=Q(B)$ so that $B\subset Q$ and $|Q|\lesssim_n|B|,$ we have $[1]_{A_{s(\cdot)}(Q)} \gtrsim_n [1]_{A_{s(\cdot)}(B)}.$ Thus, 
\begin{align}\label{eq:MuckAlphaBallVsCubes}
[1]_{A_{s(\cdot)}(\mathcal{B})} \sim_n [1]_{A_{s(\cdot)}(\cup_{j=1}^{3^n}\mathcal{D}^j)}.
\end{align}

Next we verify \eqref{eq:1}.
Let $x\in B.$ Let $Q = Q_B\in\mathcal{D}^j$ be the covering cube as provided by Lemma \ref{lem:CoveringLattices}. Then, by Lemma \ref{lem:IndicatorKopaliani}, used in the passage to the second line, and \eqref{eq:MuckAlphaBallVsCubes} with the choice $s(\cdot) = 1/\alpha(\cdot)$, we have 
	\begin{multline}\label{eq:11}
		\frac{1_B}{|B|^{1-\langle \alpha\rangle_B}}\int_B|f(y)|\ud y \leq \frac{|Q|^{1-\langle \alpha\rangle_Q}}{|B|^{1-\langle \alpha\rangle_B}} 	\frac{1_Q}{|Q|^{1-\langle \alpha\rangle_Q}}\int_Q|f(y)|\ud y  \lesssim_n \frac{|B|^{\langle \alpha\rangle_B}}{|Q|^{\langle \alpha\rangle_Q}} 	M^{\alpha(\cdot)}_{\mathcal{D}^j}f(x) \\ \lesssim[1]_{A_{\frac{1}{\alpha(\cdot)}}(Q)}\frac{\|1_B\|_{L^{\frac{1}{\alpha(\cdot)}}}}{\|1_Q\|_{L^{\frac{1}{\alpha(\cdot)}}}}	M^{\alpha(\cdot)}_{\mathcal{D}^j}f(x)\lesssim [1]_{A_{\frac{1}{\alpha(\cdot)}}(\mathcal B)}	M^{\alpha(\cdot)}_{\mathcal{D}^j}f(x) \lesssim [1]_{A_{\frac{1}{\alpha(\cdot)}}(\mathcal B)} \sum_{j=1}^{3^n}	M^{\alpha(\cdot)}_{\mathcal{D}^j}f(x).
	\end{multline}
The verification of \eqref{eq:2} is analogous, with the roles of balls and cubes swapped; we leave the details to the reader. 

For \eqref{eq:3}, note first that $|I^{\alpha(\cdot)}f|\leq I^{\alpha(\cdot)}|f|$, so it is enough to consider $f\geq0$. For $y\ne x$, set $B_{x,y}:=B(x,|x-y|)$ and choose, by Lemma \ref{lem:CoveringLattices}, a cube $Q_{x,y}\in\mathcal D^{j(x,y)}$ such that $B_{x,y}\subset Q_{x,y}\subset 10B_{x,y}.$
Then $x,y\in Q_{x,y}$ and the same bound employed in the bound \eqref{eq:11} above gives 
\[
\frac{1}{|x-y|^{n(1-\langle \alpha\rangle_{B_{x,y}})}}
\lesssim_n
[1]_{A_{\frac{1}{\alpha(\cdot)}}(Q_{x,y})}
\frac{1_{Q_{x,y}}(x)1_{Q_{x,y}}(y)}{|Q_{x,y}|^{1-\langle \alpha\rangle_{Q_{x,y}}}} \lesssim_n
[1]_{A_{\frac{1}{\alpha(\cdot)}}(\mathcal B)}\sum_{j=1}^{3^n}\sum_{Q\in\mathcal D^j}
\frac{1_Q(x)1_Q(y)}{|Q|^{1-\langle \alpha\rangle_Q}}.
\]
Multiplying by $f(y)$ and integrating in $y$ proves \eqref{eq:3}.

It remains to verify \eqref{eq:4}. Again it suffices to consider $f\geq0$. By Tonelli's theorem,
\begin{align}\label{eq:adhoc99}
	I^{\alpha(\cdot)}_{\mathcal D}f(x)
	= \sum_{Q\in\mathcal{D}} \frac{1_Q(x)}{|Q|^{1-\langle \alpha\rangle_Q}}\int_Q f(y)\,\mathrm dy =
	\int_{\mathbb R^n} f(y)
	\sum_{Q\in\mathcal D}
	\frac{1_{Q}(x)1_Q(y)}{|Q|^{1-\langle \alpha\rangle_Q}}
	\,\mathrm dy.
\end{align}
Let $Q_*$ be the smallest cube that contains both $x$ and $y$; if it does not exist we are done. We have 
\begin{align}\label{eq:adhoc999}
    \sum_{Q\in\mathcal D}
\frac{1_{Q}(x)1_Q(y)}{|Q|^{1-\langle \alpha\rangle_Q}} = \sum_{Q_*\subset Q\in\mathcal D}
\frac{1}{|Q|^{1-\langle \alpha\rangle_Q}} \leq  \frac{C^{\alpha(\cdot)}_{\uparrow}(\mathcal{D})}{|Q_*|^{1-\langle \alpha\rangle_{Q_*}}}.
\end{align}

Let $B:= B(x,|x-y|)$ and $B^*:=B(x,2\diam(Q_*))$. Since $x,y\in Q_*$, both $B\subset B^*$ and $Q_*\subset B^*$, while $|B^*|\lesssim_n |Q_*|$.
By Lemma \ref{lem:IndicatorKopaliani} we bound 
\begin{multline}\label{eq:adhoc9999}
		|B|^{1-\langle\alpha\rangle_B} \lesssim \|1_B\|_{L^{\frac{1}{1-\alpha(\cdot)}}} \leq  \|1_{B^*}\|_{L^{\frac{1}{1-\alpha(\cdot)}}} \lesssim_n [1]_{A_{\frac{1}{\alpha(\cdot)}}(\mathcal{B})} \frac{|B^*|}{\|1_{B^*}\|_{L^{1/\alpha(\cdot)}}} \\ \lesssim_n [1]_{A_{\frac{1}{\alpha(\cdot)}}(\mathcal{B})} \frac{|Q_{*}|}{\|1_{Q_{*}}\|_{L^{1/\alpha(\cdot)}}} \lesssim_n [1]_{A_{\frac{1}{\alpha(\cdot)}}(\mathcal{B})} \frac{|Q_{*}|}{|Q_*|^{\langle\alpha\rangle_{Q_*}}}  \lesssim_n  [1]_{A_{\frac{1}{\alpha(\cdot)}}(\mathcal{B})}|Q_*|^{1-\langle\alpha\rangle_{Q_*}}.
\end{multline}
Thus,
\[
\ \frac{1}{|Q_*|^{1-\langle \alpha\rangle_{Q_*}}}
\lesssim_n
\frac{[1]_{A_{\frac{1}{\alpha(\cdot)}}(\mathcal B)}}{|B|^{1-\langle\alpha\rangle_B}} = 
\frac{[1]_{A_{\frac{1}{\alpha(\cdot)}}(\mathcal B)}}{|B(x,|x-y|)|^{1-\langle\alpha\rangle_{B(x,|x-y|)}}}.
\]
Chaining the bounds \eqref{eq:adhoc99}, \eqref{eq:adhoc999} and \eqref{eq:adhoc9999} proves \eqref{eq:4}.
\end{proof}

Next we give the proofs of Theorems \ref{thm:MainIntro1} and \ref{thm:MainIntro4}.

\begin{proof}[Proof of Theorem \ref{thm:MainIntro1}] 
Fix a ball $B$. For $f\geq 0$ we have
\(
\frac{1_B}{|B|^{1-\langle\alpha\rangle_B}}\int_B f
\leq
M^{\alpha(\cdot)}f.
\)
Taking the $L^{q(\cdot)}$ norm and then supremum over
$\|f\|_{L^{p(\cdot)}}\leq1$ gives, by Lemma \ref{lem:dualityFormula}, that
\[
[1]_{A_{p(\cdot),q(\cdot)}^{\alpha(\cdot)}(\mathbb R^n)} = \sup_{B}\frac{\|1_B\|_{L^{q(\cdot)}}\|1_B\|_{L^{p'(\cdot)}}}
{|B|^{1-\langle\alpha\rangle_B}}
\lesssim
\|M^{\alpha(\cdot)}\|_{L^{p(\cdot)}\to L^{q(\cdot)}}.
\]
For the upper bound we use the pointwise bound \eqref{eq:1} of Proposition \ref{prop:mean-to-dyadic} and bound 
\begin{multline}
    \|M^{\alpha(\cdot)}\|_{L^{p(\cdot)}(\mathbb{R}^n)\to L^{q(\cdot)}(\mathbb{R}^n)} \lesssim [1]_{A_{\frac{1}{\alpha(\cdot)}}(\mathbb{R}^n)} \sum_{j=1}^{3^n}\|M^{\alpha(\cdot)}_{\mathcal{D}^j}\|_{L^{p(\cdot)}(\mathbb{R}^n)\to L^{q(\cdot)}(\mathbb{R}^n)} \\ 
    \lesssim_n [1]_{A_{p(\cdot),q(\cdot)}^{\alpha(\cdot)}(\mathcal{B})}  \big([1]_{A_{\frac{1}{\alpha(\cdot)}}(\mathcal B)}\big)^3   \|M_{\mathcal B}\|_{ L^{q'(\cdot)}} \|M_{\mathcal B}\|_{ L^{p(\cdot)}},
\end{multline}
where, in the second bound, for each $j$, we have by Theorem \ref{thm:MainIntro2} and using \eqref{eq:MuckAlphaBallVsCubes} that
\begin{align*}
    \|M^{\alpha(\cdot)}_{\mathcal{D}^j}\|_{L^{p(\cdot)}(\mathbb{R}^n)\to L^{q(\cdot)}(\mathbb{R}^n)} &\leq [1]_{A_{p(\cdot),q(\cdot)}^{\alpha(\cdot)}(\mathcal{D}^j)}  [1]_{A_{\frac{1}{\alpha(\cdot)}}(\mathcal{D}^j)}    \|M_{\mathcal{D}^j}\|_{ L^{q'(\cdot)}} \|M_{\mathcal{D}^j}\|_{ L^{p(\cdot)}} \\
    &\lesssim_n [1]_{A_{p(\cdot),q(\cdot)}^{\alpha(\cdot)}(\mathcal{B})}  ([1]_{A_{\frac{1}{\alpha(\cdot)}}(\mathcal B)})^2    \|M_{\mathcal B}\|_{ L^{q'(\cdot)}} \|M_{\mathcal B}\|_{ L^{p(\cdot)}},
\end{align*}
where we used the bound 
\[
[1]_{A_{p(\cdot),q(\cdot)}^{\alpha(\cdot)}(\mathcal{D}^j)} \lesssim_n [1]_{A_{\frac{1}{\alpha(\cdot)}}(\mathcal B)}[1]_{A_{p(\cdot),q(\cdot)}^{\alpha(\cdot)}(\mathcal{B})}.
\]
This proves the required estimate.
\end{proof}

\begin{proof}[Proof of Theorem \ref{thm:MainIntro4}]
We begin with the first bullet. 
Suppose that $M$ is bounded on $L^{q'(\cdot)}$ and $1\in A_{\frac{1}{\alpha(\cdot)}}(\mathbb{R}^n).$
We prove that if $M^{\alpha(\cdot)}:L^{p(\cdot)}\to L^{q(\cdot)}$ is bounded and $C^{\alpha(\cdot)}_{\downarrow}<\infty$, then $I^{\alpha(\cdot)}$ is bounded. 
By $1\in A_{\frac{1}{\alpha(\cdot)}}$ and \eqref{eq:3} it is enough to show that $I^{\alpha(\cdot)}_{\mathcal D^j}:L^{p(\cdot)}\to L^{q(\cdot)}$ is bounded for the $j=1,\ldots,3^n$ dyadic lattices of Lemma \ref{lem:CoveringLattices}. We verify their boundedness through the first bullet of Theorem \ref{thm:RieszPot}.
Since $M^{\alpha(\cdot)}:L^{p(\cdot)}\to L^{q(\cdot)}$ is bounded and $1\in A_{\frac{1}{\alpha(\cdot)}}$ it follows from \eqref{eq:2} that
\(
M^{\alpha(\cdot)}_{\mathcal D^j}:L^{p(\cdot)}\to L^{q(\cdot)}
\)
are bounded. On the other hand, the condition
$C^{\alpha(\cdot)}_{\downarrow}<\infty$ automatically gives $C^{\alpha(\cdot)}_{\downarrow}(\mathcal{D}^j)<\infty$, i.e. the packing condition of
Theorem \ref{thm:RieszPot}. Moreover, each
$\mathcal D^j$ has the $A_1$ self-improvement
property, which follows from the classical theory of $A_1$ weights, see e.g. the book of Duoandikoetxea \cite{Duoandikoetxea2001}. Moreover, $M_{\mathcal D^j}\lesssim_n M$, hence $M_{\mathcal D^j}$ is bounded on $L^{q'(\cdot)}$  and $1\in A_{\frac{1}{\alpha(\cdot)}}(\mathcal D^j)$ by
$[1]_{A_{\frac{1}{\alpha(\cdot)}}(\mathcal D^j)}
\lesssim_n[1]_{A_{\frac{1}{\alpha(\cdot)}}(\mathbb{R}^n)}$; hence the assumptions of 
Theorem \ref{thm:RieszPot} are in force and the proof of the first claim is concluded. 

We turn to the second bullet. 
Suppose that $1\in A_{q(\cdot)}$ and $\alpha(\cdot) = \tfrac{1}{p(\cdot)}-\tfrac{1}{q(\cdot)}$ and $C_{\uparrow}^{\alpha(\cdot)} <\infty.$ We prove that if 
\(
I^{\alpha(\cdot)}:L^{p(\cdot)}\to L^{q(\cdot)}
\)
is bounded, then both conclusions of the second bullet follow.
Let $\mathcal{D}$ be an arbitrary dyadic lattice. By $C_{\uparrow}^{\alpha(\cdot)} <\infty$ and \eqref{eq:4}, the boundedness of
\(
I^{\alpha(\cdot)}:L^{p(\cdot)}(\mathbb R^n)\to L^{q(\cdot)}(\mathbb R^n)
\)
implies that
\(
I^{\alpha(\cdot)}_{\mathcal{D}}:L^{p(\cdot)}(\mathbb{R}^n)\to L^{q(\cdot)}(\mathbb{R}^n)
\)
is bounded, with a bound independent of $\mathcal D$. The Muckenhoupt assumptions also pass uniformly to $\mathcal D$: for
$s(\cdot)\in\{\tfrac{1}{\alpha(\cdot)},p(\cdot),q(\cdot)\}$,
\(
[1]_{A_{s(\cdot)}(\mathcal D)}
\lesssim_n
[1]_{A_{s(\cdot)}(\mathbb R^n)}.
\)
Therefore the assumptions of the second bullet of Theorem \ref{thm:RieszPot} are in force, uniformly in $\mathcal D$. Thus
\(
M^{\alpha(\cdot)}_{\mathcal D}:L^{p(\cdot)}(\mathbb R^n)\to L^{q(\cdot)}(\mathbb R^n)
\)
is bounded uniformly in $\mathcal D$, and $C^{\alpha(\cdot)}_{\downarrow}(\mathcal{D}) < \infty$ with a bound that does not depend on $\mathcal D$. Taking the supremum over $\mathcal D$ gives $C_{\downarrow}^{\alpha(\cdot)}<\infty$.

It remains only to pass from dyadic maximal operators to the continuous one. Applying the preceding dyadic maximal bound to the finitely many lattices $\mathcal D^j$, $j=1,\ldots,3^n$, of Lemma \ref{lem:CoveringLattices}, and then using \eqref{eq:1}, gives
\(
M^{\alpha(\cdot)}:L^{p(\cdot)}(\mathbb R^n)\to L^{q(\cdot)}(\mathbb R^n).
\)
\end{proof}

The same covering lattices allow us to compare Morrey norms.

\begin{proposition}\label{prop:MorreyBallsDyadic}
Assume that $p(\cdot),r(\cdot)\in\mathcal P(\mathbb R^n)$.
Write $\mathcal B$ for the collection of all balls in $\mathbb R^n$.
Let $\mathcal{D}^1,\ldots,\mathcal{D}^{3^n}$ be the dyadic lattices of Lemma \ref{lem:CoveringLattices}. Then, we have 
\begin{align}
    \|f\|_{\mathcal M^{r(\cdot)}_{p(\cdot)}(\mathcal B)}
&\lesssim_n
[1]_{A_{r(\cdot)}(\mathcal B)}[1]_{A_{p(\cdot)}(\mathcal B)}
\max_{1\leq j\leq 3^n}
\|f\|_{\mathcal M^{r(\cdot)}_{p(\cdot)}(\mathcal D^j)}, \\ 
\max_{1\leq j\leq 3^n}
\|f\|_{\mathcal M^{r(\cdot)}_{p(\cdot)}(\mathcal D^j)}
&\lesssim_n
[1]_{A_{r(\cdot)}(\mathcal B)}[1]_{A_{p(\cdot)}(\mathcal B)}
\|f\|_{\mathcal M^{r(\cdot)}_{p(\cdot)}(\mathcal B)}.
\end{align}
\end{proposition}

\begin{proof}
We first compare the scaling factors for nested sets of comparable size. Take $s(\cdot)\in\mathcal P(\mathbb R^n)$. Let $E\subset F$ be either balls or dyadic cubes, and assume that $|F|\leq c|E|$. We will use
\begin{align}
|E|^{1/s_E}
&\lesssim_{n,c}
[1]_{A_{s(\cdot)}(\mathcal B)}|F|^{1/s_F},
\label{eq:MorreyNestedSmallLarge}
\\
|F|^{1/s_F}
&\lesssim_{n,c}
[1]_{A_{s(\cdot)}(\mathcal B)}|E|^{1/s_E}.
\label{eq:MorreyNestedLargeSmall}
\end{align}
For every dyadic cube $S$, the quantity $[1]_{A_{s(\cdot)}(S)}$ is controlled by $[1]_{A_{s(\cdot)}(\mathcal B)}$. Indeed, choose a ball $B$ such that $S\subset B$ and $|B|\lesssim_n |S|$. Then monotonicity of indicator norms gives
\begin{equation}\label{eq:MorreyCubeABallControl}
[1]_{A_{s(\cdot)}(S)}
=
\frac{\|1_S\|_{L^{s(\cdot)}}\|1_S\|_{L^{s'(\cdot)}}}{|S|}
\lesssim_n
\frac{\|1_{B}\|_{L^{s(\cdot)}}\|1_{B}\|_{L^{s'(\cdot)}}}{|B|}
\leq
[1]_{A_{s(\cdot)}(\mathcal B)}.
\end{equation}
We prove \eqref{eq:MorreyNestedSmallLarge}. By Lemma \ref{lem:IndicatorKopaliani} and monotonicity,
\[
|E|^{1/s_E}
\lesssim
\|1_E\|_{L^{s(\cdot)}}
\leq
\|1_F\|_{L^{s(\cdot)}}
\lesssim_{n,c}
[1]_{A_{s(\cdot)}(\mathcal B)}|F|^{1/s_F}.
\]
In the last step, if $F$ is a ball this follows from the definition of
$[1]_{A_{s(\cdot)}(\mathcal B)}$; if $F$ is a cube, it follows from
\eqref{eq:MorreyCubeABallControl}.

We next prove \eqref{eq:MorreyNestedLargeSmall}. The definition of the $A_{s(\cdot)}$ constant, Lemma \ref{lem:IndicatorKopaliani}, and \eqref{eq:MorreyCubeABallControl} when $F$ is a cube give
\[
|F|^{1/s_F}
\lesssim
[1]_{A_{s(\cdot)}(\mathcal B)}
\frac{|F|}{\|1_F\|_{L^{s'(\cdot)}}}.
\]
Since $E\subset F$ and $|F|\leq c|E|$, we also have
\[
\frac{|F|}{\|1_F\|_{L^{s'(\cdot)}}}
\lesssim_c
\frac{|E|}{\|1_E\|_{L^{s'(\cdot)}}}\lesssim |E|^{1/s_E}.
\]
Here we used Lemma \ref{lem:IndicatorKopaliani} in the last estimate.
Applying \eqref{eq:MorreyNestedSmallLarge} with $s(\cdot)=r(\cdot)$ and \eqref{eq:MorreyNestedLargeSmall} with $s(\cdot)=p(\cdot)$ gives
\begin{equation}\label{eq:MorreyScaleComparable}
|E|^{\frac1{r_E}-\frac1{p_E}}
\lesssim_{n,c}
[1]_{A_{r(\cdot)}(\mathcal B)}[1]_{A_{p(\cdot)}(\mathcal B)}
|F|^{\frac1{r_F}-\frac1{p_F}}.
\end{equation}

To prove the first inequality, fix $B\in\mathcal B$.
By Lemma \ref{lem:CoveringLattices}, there are $j$ and $Q\in\mathcal D^j$ such that $B\subset Q\subset 10B$. The estimate \eqref{eq:MorreyScaleComparable} and the lattice property of $L^{p(\cdot)}$ yield
\[
\begin{aligned}
|B|^{\frac1{r_B}-\frac1{p_B}}\|1_Bf\|_{L^{p(\cdot)}}
&\lesssim_n
[1]_{A_{r(\cdot)}(\mathcal B)}[1]_{A_{p(\cdot)}(\mathcal B)}
|Q|^{\frac1{r_Q}-\frac1{p_Q}}\|1_Qf\|_{L^{p(\cdot)}}
\\
&\leq
[1]_{A_{r(\cdot)}(\mathcal B)}[1]_{A_{p(\cdot)}(\mathcal B)}
\max_{1\leq j\leq 3^n}
\|f\|_{\mathcal M^{r(\cdot)}_{p(\cdot)}(\mathcal D^j)}.
\end{aligned}
\]
Taking the supremum over all balls $B$ proves the first inequality.

To prove the reverse inequality, fix $Q\in\mathcal D^j$ for some $j$.
Choose a ball $B_Q$ such that $Q\subset B_Q$ and $|B_Q|\lesssim_n |Q|$. Using \eqref{eq:MorreyScaleComparable} again, together with the lattice property, gives
\[
\begin{aligned}
|Q|^{\frac1{r_Q}-\frac1{p_Q}}\|1_Qf\|_{L^{p(\cdot)}}
&\lesssim_n
[1]_{A_{r(\cdot)}(\mathcal B)}[1]_{A_{p(\cdot)}(\mathcal B)}
|B_Q|^{\frac1{r_{B_Q}}-\frac1{p_{B_Q}}}\|1_{B_Q}f\|_{L^{p(\cdot)}}
\\
&\leq
[1]_{A_{r(\cdot)}(\mathcal B)}[1]_{A_{p(\cdot)}(\mathcal B)}
\|f\|_{\mathcal M^{r(\cdot)}_{p(\cdot)}(\mathcal B)}.
\end{aligned}
\]
Taking the supremum over $Q\in\mathcal D^j$ and then the maximum over $j$ proves the reverse inequality.
\end{proof}

The continuous Morrey bounds are characterized by the following testing criteria.

\begin{theorem}\label{thm:MorreyIff}\label{thm:MorreyMaxIff}\label{thm:MorreyRieszIff}
Let $\mathcal B$ be the collection of all balls in $\mathbb R^n$.
Let
$p(\cdot),q(\cdot),r(\cdot),s(\cdot),1/\alpha(\cdot)\in\mathcal P(\mathbb R^n)$
and assume that
\(
\frac1{p(\cdot)}-\frac1{q(\cdot)}=\frac1{r(\cdot)}-\frac1{s(\cdot)}
\)
almost everywhere on $\mathbb R^n$. Suppose also that
\[
[1]_{A_{p(\cdot)}(\mathcal B)}+
[1]_{A_{q(\cdot)}(\mathcal B)}+
[1]_{A_{r(\cdot)}(\mathcal B)}+
[1]_{A_{s(\cdot)}(\mathcal B)}+
[1]_{A_{\frac1{\alpha(\cdot)}}(\mathcal B)}
<\infty.
\]
Then the following criteria hold.
\begin{itemize}
\item[Maximal.] Assume that
\(
M^{\alpha(\cdot)}:L^{p(\cdot)}(\mathbb R^n)\to L^{q(\cdot)}(\mathbb R^n)
\)
is bounded. Then
\(
M^{\alpha(\cdot)}:
\mathcal M^{r(\cdot)}_{p(\cdot)}(\mathcal B)
\to
\mathcal M^{s(\cdot)}_{q(\cdot)}(\mathcal B)
\)
is bounded \textbf{if and only if}
\begin{equation}\label{eq:MorreyIffMaxTesting}
C_M^{\mathcal B}:=
\sup_{B\in\mathcal B}
\frac{
\|1_B\|_{\mathcal M^{s(\cdot)}_{q(\cdot)}(\mathcal B)}
\|1_B\|_{(\mathcal M^{r(\cdot)}_{p(\cdot)}(\mathcal B))'}
}
{|B|^{1-\langle\alpha\rangle_B}}
<\infty.
\end{equation}
\item[Riesz.] Assume that
\(
I^{\alpha(\cdot)}:L^{p(\cdot)}(\mathbb R^n)\to L^{q(\cdot)}(\mathbb R^n)
\)
is bounded. Then
\(
I^{\alpha(\cdot)}:
\mathcal M^{r(\cdot)}_{p(\cdot)}(\mathcal B)
\to
\mathcal M^{s(\cdot)}_{q(\cdot)}(\mathcal B)
\)
is bounded \textbf{if and only if}
\begin{equation}\label{eq:MorreyIffRieszTesting}
C_I^{\mathcal B}:=
\sup_{B_0\in\mathcal B}
\|1_{B_0}\|_{\mathcal M^{s(\cdot)}_{q(\cdot)}(\mathcal B)}
\left\|
K_{B_0}^{\alpha(\cdot)}
\right\|_{(\mathcal M^{r(\cdot)}_{p(\cdot)}(\mathcal B))'}
<\infty,
\end{equation}
where \(c_{B_0}\) is the centre of \(B_0\) and
\[
K_{B_0}^{\alpha(\cdot)}(y):=
\frac{1_{\mathbb R^n\setminus 2B_0}(y)}
{|y-c_{B_0}|^{n\left(1-\langle\alpha\rangle_{B(c_{B_0},|y-c_{B_0}|)}\right)}}.
\]
\end{itemize}
\end{theorem}

\begin{proof}
We prove the maximal statement first, and then the Riesz statement.

\emph{Maximal necessity.} Suppose that \(M^{\alpha(\cdot)}\) is bounded between
the Morrey spaces. Fix a ball \(B\). For every \(f\),
\[
\frac{1_B}{|B|^{1-\langle\alpha\rangle_B}}\int_B |f|
\leq
M^{\alpha(\cdot)}f.
\]
Taking the target Morrey norm and the suprema over
\(\|f\|_{\mathcal M^{r(\cdot)}_{p(\cdot)}(\mathcal B)}\le1\) and
\(B\in\mathcal B\) gives
\eqref{eq:MorreyIffMaxTesting}.

\emph{Maximal sufficiency.} Conversely, assume that
\eqref{eq:MorreyIffMaxTesting} holds and fix a ball
$B_0$. For \(x\in B_0\), split the balls in the supremum defining
\(M^{\alpha(\cdot)}f(x)\) into those contained in \(4B_0\) and the remaining
balls. The local part is dominated by \(M^{\alpha(\cdot)}(f1_{4B_0})\).
The assumed \(L^{p(\cdot)}\)-to-\(L^{q(\cdot)}\) boundedness, the relation
\(\frac1{p(\cdot)}-\frac1{q(\cdot)}=\frac1{r(\cdot)}-\frac1{s(\cdot)}\), and
\eqref{eq:MorreyScaleComparable} for \(B_0\) and \(4B_0\) give
\[
|B_0|^{\frac1{s_{B_0}}-\frac1{q_{B_0}}}
\|1_{B_0}M^{\alpha(\cdot)}(f1_{4B_0})\|_{L^{q(\cdot)}}
\lesssim
\|f\|_{\mathcal M^{r(\cdot)}_{p(\cdot)}(\mathcal B)}.
\]
For the remaining part, let \(x\in B_0\). For each ball \(B\ni x\) not contained
in \(4B_0\), there is a ball \(\widetilde B\supset B\cup B_0\) with
\(|\widetilde B|\sim_n |B|\) such that, by Lemma
\ref{lem:IndicatorKopaliani} applied with exponent \(1/(1-\alpha(\cdot))\),
\[
\frac{1}{|B|^{1-\langle\alpha\rangle_B}}\int_B |f|
\lesssim
\frac{1}{|\widetilde B|^{1-\langle\alpha\rangle_{\widetilde B}}}
\int_{\widetilde B}|f|.
\]
By the testing condition,
\[
\frac{1}{|\widetilde B|^{1-\langle\alpha\rangle_{\widetilde B}}}
\int_{\widetilde B}|f|
\leq
C_M^{\mathcal B}
\frac{\|f\|_{\mathcal M^{r(\cdot)}_{p(\cdot)}(\mathcal B)}}
{\|1_{\widetilde B}\|_{\mathcal M^{s(\cdot)}_{q(\cdot)}(\mathcal B)}}.
\]
Since \(B_0\subset\widetilde B\), the lattice property gives
\(\|1_{B_0}\|_{\mathcal M^{s(\cdot)}_{q(\cdot)}(\mathcal B)}
\leq\|1_{\widetilde B}\|_{\mathcal M^{s(\cdot)}_{q(\cdot)}(\mathcal B)}\).
This bounds the global part by
\(
C_M^{\mathcal B}\|f\|_{\mathcal M^{r(\cdot)}_{p(\cdot)}(\mathcal B)}.
\)
Taking the supremum over \(B_0\) gives the desired Morrey bound for
\(M^{\alpha(\cdot)}\).

\emph{Riesz necessity.} Suppose that \(I^{\alpha(\cdot)}\) is bounded between
the Morrey spaces. Fix a ball \(B_0\). If \(x\in B_0\) and
\(y\notin 2B_0\), then the balls \(B(x,|x-y|)\) and
\(B(c_{B_0},|y-c_{B_0}|)\) have comparable radii and each is contained in a
fixed multiple of the other. Applying Lemma \ref{lem:IndicatorKopaliani}
with exponent \(1/(1-\alpha(\cdot))\) therefore gives
\[
\frac{1}{|x-y|^{n(1-\langle\alpha\rangle_{B(x,|x-y|)})}}
\gtrsim
K_{B_0}^{\alpha(\cdot)}(y),
\qquad x\in B_0,\ y\notin 2B_0.
\]
Thus, for \(f\geq0\),
\[
1_{B_0}(x)\int f(y)K_{B_0}^{\alpha(\cdot)}(y)\,\mathrm dy
\lesssim
1_{B_0}(x)I^{\alpha(\cdot)}f(x).
\]
Taking the target Morrey norm and the suprema over
\(\|f\|_{\mathcal M^{r(\cdot)}_{p(\cdot)}(\mathcal B)}\leq1\) and
\(B_0\in\mathcal B\) gives
\eqref{eq:MorreyIffRieszTesting}.

\emph{Riesz sufficiency.} Conversely, assume \eqref{eq:MorreyIffRieszTesting}.
Fix a ball \(B_0\) and write
\[
1_{B_0}I^{\alpha(\cdot)}f
=
1_{B_0}I^{\alpha(\cdot)}(f1_{2B_0})
+
1_{B_0}I^{\alpha(\cdot)}(f1_{\mathbb R^n\setminus 2B_0}).
\]
The first term is bounded as in the local maximal estimate, using the assumed
\(L^{p(\cdot)}\)-to-\(L^{q(\cdot)}\) boundedness of
\(I^{\alpha(\cdot)}\) and \eqref{eq:MorreyScaleComparable} for \(B_0\) and
\(2B_0\). The reverse kernel comparison gives
\[
1_{B_0}I^{\alpha(\cdot)}(f1_{\mathbb R^n\setminus 2B_0})
\lesssim
1_{B_0}\int |f(y)|K_{B_0}^{\alpha(\cdot)}(y)\,\mathrm dy.
\]
Using \eqref{eq:MorreyIffRieszTesting} and
\(|B_0|^{\frac1{s_{B_0}}-\frac1{q_{B_0}}}\|1_{B_0}\|_{L^{q(\cdot)}}
\leq \|1_{B_0}\|_{\mathcal M^{s(\cdot)}_{q(\cdot)}(\mathcal B)}\), we obtain
\[
|B_0|^{\frac1{s_{B_0}}-\frac1{q_{B_0}}}
\|1_{B_0}I^{\alpha(\cdot)}(f1_{\mathbb R^n\setminus 2B_0})\|_{L^{q(\cdot)}}
\lesssim
C_I^{\mathcal B}\|f\|_{\mathcal M^{r(\cdot)}_{p(\cdot)}(\mathcal B)}.
\]
Taking the supremum over \(B_0\) gives the desired Morrey bound for
\(I^{\alpha(\cdot)}\).
\end{proof}

\appendix
\section{Interpretation in the logarithmic H\"older continuous case}\label{sect:logHolder}
Much of the classical variable exponent theory is phrased in terms of the local $\mathcal{LH}_0$
and asymptotic $\mathcal{LH}_{\infty}$ logarithmic H\"older continuity classes; see, for instance,
\cite{DienHHR2011Book}*{Chapter 4}. In this section
$X\subset\mathbb R^n$ stands for an open set. We say that $p(\cdot)\in \mathcal{LH}_0(X)$ if there
exists $C_0>0$ so that
\[
|p(x)-p(y)|\leq\frac{C_0}{-\log |x-y|},
\qquad x,y\in X,\quad 0<|x-y|<\frac12.
\]
We say that $p(\cdot)\in \mathcal{LH}_{\infty}(X)$ if there are
$p_\infty\in\mathbb R$ and $C_\infty>0$ so that 
\[
|p(x)-p_\infty|\leq\frac{C_\infty}{\log(e+|x|)},
\qquad x\in X.
\]
We write $\mathcal{LH}(X):=\mathcal{LH}_0(X)\cap \mathcal{LH}_{\infty}(X)$ for the class of log-H\"older continuous exponents. 
In this section we give one interpretation of our results in the log-H\"older continuous case by using the following consequence of log-H\"older continuity.
\begin{lemma}[{\cite{DienHHR2011Book}*{Corollary 4.5.9}}]\label{lem:LogHolderStability}
Let $X \subset\mathbb{R}^n$ and $B\subset X$ be a ball.  
\begin{itemize}
\item If $|B|\leq1$ and $p(\cdot)\in \mathcal{LH}_0(X)$, then
\[
|B|^{1/p_B}\sim |B|^{1/p(x)},\qquad x\in B.
\]
\item If $|B|\geq1$ and $p(\cdot)\in \mathcal{LH}_{\infty}(X)$, then
\[
|B|^{1/p_B}\sim |B|^{1/p_\infty}.
\]
\end{itemize}
The implicit constants do not depend on $B$ or $x$, and balls $B$ can be replaced with cubes $Q.$
\end{lemma}

On small domains (e.g. not containing balls with $|B|\gtrsim 1$) Lemma \ref{lem:LogHolderStability} allows us to compare the following pointwise-fractionality operators
\begin{align}\label{eq:defn:pwOperators}
M^{\alpha(x)}_{X, \mathrm{loc}}f(x)
&:= \sup_{10B\subset X}\frac{1_B(x)}{|B|^{1-\alpha(x)}}\int_B |f(y)|\,\mathrm dy,\\
I^{\alpha(x)}_{X, \mathrm{loc}}f(x)
&:= \int_{X\cap \{10|x-y|\leq \dist(x,\partial X)\}}\frac{f(y)}{|x-y|^{n(1-\alpha(x))}}\,\mathrm dy,
\end{align}
with the mean-evaluation operators considered in previous sections. 
The constant $10$ localizing the above operators is related to Lemma \ref{lem:CoveringLattices} and passing between the dyadic and non-dyadic results. 
Such or similar operators with pointwise evaluation of fractionality have appeared in \cite{KokSam04}, for instance.
For domains containing large scales one should take the asymptotic $\mathcal{LH}_{\infty}$ log-H\"older behaviour into
account, while restricting to small scales/balls avoids this case distinction. 
As a consequence of the main results of this paper, we record the following example.
\begin{corollary} \label{cor:logholderbounds}
Let $p(\cdot),q(\cdot),1/\alpha(\cdot)\in \mathcal{LH}_0(\mathbb{R}^n)\cap \mathcal{P}(\mathbb{R}^n)$. Suppose that
\(
1<p_-\leq p(\cdot)\leq q(\cdot) \leq q_+<\infty
\)
everywhere. Fix a ball $B_0\subset\mathbb R^n$ with $|B_0|\leq 1.$ 
The following are equivalent:
\begin{enumerate}[(i)]
\item \(\alpha(\cdot)\geq \tfrac1{p(\cdot)}-\tfrac1{q(\cdot)}\) in $B_0$,
\item $M^{\alpha(x)}_{B_0,\mathrm{loc}}:L^{p(\cdot)}(B_0) \to L^{q(\cdot)}(B_0)$ is bounded,
\item $I^{\alpha(x)}_{B_0,\mathrm{loc}}:L^{p(\cdot)}(B_0) \to L^{q(\cdot)}(B_0)$ is bounded.
\end{enumerate}
\end{corollary}

\begin{proof}
We verify that $(iii) \Rightarrow (i) \Rightarrow (ii) \Rightarrow (iii).$

We first show that either operator bound implies (i). Let
\[
\mathcal B_{10}:=\{B:10B\subset B_0\},\qquad
\mathcal B_{30}:=\{B:30B\subset B_0\}.
\]
If \(M^{\alpha(x)}_{B_0,\mathrm{loc}}\) is bounded, then Lemma
\ref{lem:LogHolderStability}, applied to $1/\alpha(\cdot)$, gives boundedness
of \(M^{\alpha(\cdot)}_{\mathcal B_{10}}\).
Suppose instead that \(I^{\alpha(x)}_{B_0,\mathrm{loc}}\) is bounded. Fix
\(B=B(c,r)\in\mathcal B_{30}\), \(x\in B\), and \(y\in B\). Since
\(30B\subset B_0\), we have \(\dist(x,\partial B_0)\geq 29r\), while
\(10|x-y|\leq 20r\). Hence \(y\) belongs to the integration region in
\(I^{\alpha(x)}_{B_0,\mathrm{loc}}f(x)\). Therefore, for \(f\geq0\),
\[
1_B(x)I^{\alpha(x)}_{B_0,\mathrm{loc}}f(x)
\geq
1_B(x)\int_B \frac{f(y)}{|x-y|^{n(1-\alpha(x))}}\,\mathrm dy
\gtrsim
\frac{1_B(x)}{|B|^{1-\alpha(x)}}\int_B f(y)\,\mathrm dy.
\]
Using Lemma \ref{lem:LogHolderStability} once more, we obtain
\(
M^{\alpha(\cdot)}_{\mathcal B_{30}}f
\lesssim
I^{\alpha(x)}_{B_0,\mathrm{loc}}f.
\)
Applying this to \(|f|\), boundedness of \(I^{\alpha(x)}_{B_0,\mathrm{loc}}\) gives boundedness of \(M^{\alpha(\cdot)}_{\mathcal B_{30}}\).
The collections \(\mathcal B_{10}\) and \(\mathcal B_{30}\)
differentiate every point of \(B_0\).
Corollary \ref{cor:AlphaRelation} on \(B_0\) gives
$\alpha(x)\geq \tfrac1{p(x)}-\tfrac1{q(x)}$
for almost every $x\in B_0$.
For the chosen log-H\"older representatives, the functions $p(\cdot),q(\cdot)$ and $1/\alpha(\cdot)$ are continuous, and hence, using $1<p_-\leq p(\cdot)\leq q(\cdot)\leq q_+<\infty$ and $1/\alpha(\cdot)\geq 1$, so are $\alpha(\cdot),1/p(\cdot)$ and $1/q(\cdot)$. If the displayed inequality failed at some point of \(B_0\), it would fail on a small ball of positive measure by continuity, contradicting the almost everywhere conclusion. Thus (i) holds pointwise for these representatives.

Then we verify that $(i) \Rightarrow (ii).$ Let $\mathcal D^j$, $j=1,\dots,3^n$, be the dyadic lattices of Lemma \ref{lem:CoveringLattices}, and write
\(
\mathcal D^j(B_0):=\{Q\in\mathcal D^j:Q\subset B_0\}.
\)
The dyadic covering argument used for \eqref{eq:1}, together with Lemma \ref{lem:LogHolderStability} applied to $1/\alpha(\cdot)$, gives
\[
\|M^{\alpha(x)}_{B_0,\mathrm{loc}}\|_{L^{p(\cdot)}\to L^{q(\cdot)}} \sim \|M^{\alpha(\cdot)}_{B_0,\mathrm{loc}}\|_{L^{p(\cdot)}\to L^{q(\cdot)}} \lesssim \sum_{j=1}^{3^n}\|M^{\alpha(\cdot)}_{\mathcal{D}^j(B_0)}\|_{L^{p(\cdot)}\to L^{q(\cdot)}}.
\]
For each fixed $j$, Theorem \ref{thm:MainIntro2}, applied on $B_0$, gives
\[
\|M^{\alpha(\cdot)}_{\mathcal{D}^j(B_0)}\|_{L^{p(\cdot)}\to L^{q(\cdot)}} \lesssim  [1]_{A_{p(\cdot),q(\cdot)}^{\alpha(\cdot)}(\mathcal{D}^j(B_0))} \left[ [1]_{A_{1/\alpha(\cdot)}(\mathcal{D}^j(B_0))}    \|M_{\mathcal{D}^j(B_0)}\|_{ L^{q'(\cdot)}} \|M_{\mathcal{D}^j(B_0)}\|_{ L^{p(\cdot)}}\right].
\]
The local log-H\"older assumptions and $(q')_-,p_->1$ give the boundedness of the dyadic maximal operators in the bracket. Lemma \ref{lem:LogHolderStability} and Lemma \ref{lem:IndicatorKopaliani}, applied with $s(\cdot)=1/\alpha(\cdot)$, give the finiteness of the local \(A_{1/\alpha(\cdot)}\) characteristic. The same two lemmas also give
\[
[1]_{A_{p(\cdot),q(\cdot)}^{\alpha(\cdot)}(\mathcal{D}^j(B_0))}
\lesssim
\sup_{Q\in \mathcal{D}^j(B_0)}	|Q|^{\langle \alpha-\frac1p+\frac1q\rangle_Q}
\leq 1.
\]
The last inequality follows from $|Q|\leq |B_0|\leq 1$ and the assumed almost everywhere inequality. Hence \(M^{\alpha(x)}_{B_0,\mathrm{loc}}\) is bounded.

Then we verify that $(ii) \Rightarrow (iii).$ Since $(ii)$ implies (i) by the first part of the proof, the dyadic maximal bounds above are available. We first reduce the local potential to the restricted dyadic potentials. By Lemma \ref{lem:LogHolderStability},
\[
I^{\alpha(x)}_{B_0,\mathrm{loc}}|f|(x)
\lesssim
\int_{\{y:\,10B(x,|x-y|)\subset B_0\}}
\frac{|f(y)|\ud y}{|x-y|^{n(1-\langle\alpha\rangle_{B(x,|x-y|)})}}.
\]
For \(y\ne x\) in this integration region, set \(B_{x,y}:=B(x,|x-y|)\) and choose, by Lemma \ref{lem:CoveringLattices}, a cube \(Q_{x,y}\in\mathcal D^{j(x,y)}\) such that
\[
B_{x,y}\subset Q_{x,y}\subset 10B_{x,y}.
\]
Since \(10B_{x,y}\subset B_0\), we have \(Q_{x,y}\in\mathcal D^{j(x,y)}(B_0)\). Lemma \ref{lem:IndicatorKopaliani}, applied with \(s(\cdot)=1/\alpha(\cdot)\), gives
\[
\frac{1}{|x-y|^{n(1-\langle \alpha\rangle_{B_{x,y}})}}
\lesssim_n
[1]_{A_{\frac{1}{\alpha(\cdot)}}(Q_{x,y})}
\frac{1_{Q_{x,y}}(x)1_{Q_{x,y}}(y)}{|Q_{x,y}|^{1-\langle \alpha\rangle_{Q_{x,y}}}}.
\]
Thus
\[
I^{\alpha(x)}_{B_0,\mathrm{loc}}|f|
\lesssim
[1]_{A_{\frac{1}{\alpha(\cdot)}}(\cup_{j=1}^{3^n}\mathcal D^j(B_0))}
\sum_{j=1}^{3^n}I^{\alpha(\cdot)}_{\mathcal D^j(B_0)}|f|.
\]
The local \(A_{1/\alpha(\cdot)}\) characteristic is finite by Lemma \ref{lem:LogHolderStability} and Lemma \ref{lem:IndicatorKopaliani}. Hence it remains to check that each dyadic Riesz potential \(I^{\alpha(\cdot)}_{\mathcal{D}^j(B_0)}\) is bounded. We use the first bullet of Theorem \ref{thm:RieszPot} on the restricted space $B_0$. The dyadic maximal bound was proved above, the local log-H\"older assumptions give the boundedness of \(M_{\mathcal D^j(B_0)}\) on \(L^{q'(\cdot)}(B_0)\), the restricted dyadic lattice has the \(A_1\) self-improvement property, and the local \(A_{1/\alpha(\cdot)}\) condition is finite. Thus the only remaining hypothesis to verify is that
\begin{equation}\label{eq:LogHolderPacking}
C^{\alpha(\cdot)}_{\downarrow}(\mathcal{D}^j(B_0)) := \sup_{Q_0\in \mathcal D^j(B_0)}
\frac{\sum_{Q\in\mathcal D^j(Q_0)} |Q|^{1+\langle\alpha\rangle_Q}}
{|Q_0|^{1+\langle\alpha\rangle_{Q_0}}}
<\infty.
\end{equation}
Since $1/\alpha(\cdot)\in\mathcal{LH}_0(\mathbb R^n)$, \cite{CruzFio2013Book}*{Proposition 2.3} gives $(1/\alpha)_+(B_0)<\infty$, and hence $\alpha_-(B_0)>0$.
For $x\in Q_0$, let $P_k(x)$ be the dyadic descendant of $Q_0$ of generation $k$ containing $x$. By Lemma \ref{lem:LogHolderStability},
\[
\sum_{Q\in\mathcal D^j(Q_0)}1_Q(x)|Q|^{\langle\alpha\rangle_Q}
=
\sum_{k=0}^\infty |P_k(x)|^{\langle\alpha\rangle_{P_k(x)}}
\lesssim
|Q_0|^{\alpha(x)}\sum_{k=0}^\infty 2^{-kn\alpha(x)}
\lesssim_{\alpha_-(B_0),n}
|Q_0|^{\langle\alpha\rangle_{Q_0}}.
\]
Integrating over $Q_0$ proves \eqref{eq:LogHolderPacking}. Theorem \ref{thm:RieszPot} therefore gives the boundedness of \(I^{\alpha(\cdot)}_{\mathcal D^j(B_0)}\) for each \(j\), and the dyadic reduction above gives the boundedness of \(I^{\alpha(x)}_{B_0,\mathrm{loc}}\). This concludes the proof of $(ii) \Rightarrow (iii).$

\end{proof}
\subsection{Other definitions of variable exponent Morrey spaces}
We then compare Definition \ref{defn:MorreyNorm} with the variable exponent Morrey spaces introduced in
\cite{AlmHasSam08}*{Section 3} by Almeida--Hasanov--Samko.
\begin{definition}[{\cite{AlmHasSam08}*{Section 3}}]
Let $X\subset\mathbb R^n$ be bounded, let
$p(\cdot)\in\mathcal P(X)$, and let
$\lambda:X\to[0,n].$ Then $L^{p(\cdot),\lambda(\cdot)}(X)$ consists of those functions $f$ for which 
\[
\|f\|_{L^{p(\cdot),\lambda(\cdot)}(X)}
:=
\sup_{\substack{x\in X\\0<\rho\leq \operatorname{diam}(X)}}
\left\|
\rho^{-\lambda(x)/p(\cdot)}
f1_{B(x,\rho)\cap X}
\right\|_{L^{p(\cdot)}(X)} < \infty.
\]
\end{definition}
This is one of the two characterisations of the variable exponent Morrey norm given in \cite{AlmHasSam08}*{Section 3}.
The Almeida--Hasanov--Samko parameter $\lambda$ is centered: the
factor in the norm uses the value $\lambda(x)$ at the center of the ball,
rather than a quantity averaged over the testing set.
 In our Definition \ref{defn:MorreyNorm} the
scaling is averaged through the harmonic mean. In the log-H\"older setting this distinction is harmless up to
constants, as recorded in the following proposition.

\begin{proposition}\label{prop:MorreyAHSNormEquivalence}
Let $X\subset\mathbb R^n$ be a bounded domain such that
$|B(x,\rho)\cap X|\sim \rho^n$ uniformly for
$x\in X$ and $0<\rho\leq\operatorname{diam}(X)$. Suppose that
$p(\cdot),r(\cdot)\in \mathcal{LH}_0(X)$ satisfy $1 < p_-\leq p(\cdot)\leq r(\cdot)\leq r_+ < \infty$ almost everywhere and let $\lambda(x)$ be defined through the relation $\lambda(x)/n := 1-p(x)/r(x).$
Then, there holds that 
\begin{align*}
&\|f\|_{L^{p(\cdot),\lambda(\cdot)}(X)} \sim \|f\|_{\mathcal M^{r(\cdot)}_{p(\cdot)}(X)}
\\ 
&\qquad :=
\sup_{\substack{x\in X\\0<\rho\leq \operatorname{diam}(X)}}
|B(x,\rho)\cap X|^{\frac1{r_{B(x,\rho)\cap X}}
	-\frac1{p_{B(x,\rho)\cap X}}}
\|f1_{B(x,\rho)\cap X}\|_{L^{p(\cdot)}(X)}.
\end{align*}
\end{proposition}

\begin{proof}
Write $\widetilde B=B(x,\rho)\cap X$. By 
$p(\cdot), r(\cdot)\in \mathcal{LH}_{0}(X)$, together with
$|\widetilde B|\sim\rho^n$, we have 
\[
|\widetilde B|^{\frac1{r_{\widetilde B}}-\frac1{p_{\widetilde B}}}
\sim
\rho^{n(\frac1{r(x)}-\frac1{p(x)})}
=
\rho^{-\lambda(x)/p(x)},
\]
uniformly in $x,\rho.$
Since \(0\leq\lambda(x)\leq n\), the same log-H\"older estimate gives
\[
\rho^{-\lambda(x)/p(y)}
\sim
\rho^{-\lambda(x)/p(x)},
\qquad y\in\widetilde B.
\]
Hence, by the lattice property and homogeneity of the norm,
\[
\left\|
\rho^{-\lambda(x)/p(\cdot)}f1_{\widetilde B}
\right\|_{L^{p(\cdot)}(X)}
\sim
\rho^{-\lambda(x)/p(x)}
\|f1_{\widetilde B}\|_{L^{p(\cdot)}(X)}
\sim
|\widetilde B|^{\frac1{r_{\widetilde B}}-\frac1{p_{\widetilde B}}}
\|f1_{\widetilde B}\|_{L^{p(\cdot)}(X)}.
\]
Taking the supremum over $x\in X$ and
$0<\rho\leq\operatorname{diam}(X)$ gives the claimed norm equivalence.
\end{proof}



	\bibliographystyle{plain}
	\bibliography{references}
\end{document}